\documentclass[noinfoline]{imsart}

\usepackage[a4paper,top=3cm, bottom=3cm, left=2.5cm, right=2.5cm]{geometry}
\usepackage{fancyhdr}

\RequirePackage{amsthm,amsmath,amsfonts,amssymb}

\allowdisplaybreaks
\RequirePackage[colorlinks,citecolor=blue,urlcolor=blue]{hyperref}
\usepackage[capitalize]{cleveref}
\usepackage{mathrsfs} 

\usepackage{graphicx}
\usepackage{subcaption} 

\usepackage{array}
\usepackage{booktabs} 
\usepackage{enumerate}
\usepackage[shortlabels]{enumitem} 
\usepackage{multicol}
\usepackage{tikz}
\usepackage{algorithm}
\usepackage[noend]{algpseudocode}

\makeatletter
\def\algbackskip{\hskip-\ALG@thistlm}
\makeatother

\newcommand{\bB}{\mathbb{B}} 
\newcommand{\bS}{\mathbb{S}} 

\newcommand{\bF}{\mathbb{F}} 
\newcommand{\bC}{\mathbb{C}} 
\newcommand{\bR}{\mathbb{R}} 
\newcommand{\bZ}{\mathbb{Z}} 
\newcommand{\bN}{\mathbb{N}} 

\newcommand{\bH}{\mathbb{H}} 
\newcommand{\bV}{\mathbb{V}} 

\newcommand{\bE}{\mathbb{E}} 
\newcommand{\bP}{\mathbb{P}} 

\newcommand{\cB}{\mathcal{B}} 
\newcommand{\cH}{\mathcal{H}} 
\newcommand{\cL}{\mathcal{L}} 
\newcommand{\cN}{\mathcal{N}} 
\newcommand{\cP}{\mathcal{P}} 
\newcommand{\cS}{\mathcal{S}} 
\newcommand{\cR}{\mathcal{R}} 
\newcommand{\cI}{\mathcal{I}} 

\newcommand{\mI}{\mathbf{I}} 
\newcommand{\mG}{\mathbf{\Gamma}} 
\newcommand{\mP}{\mathbf{P}} 
\newcommand{\mF}{\mathbf{F}} 
\newcommand{\mR}{\mathbf{R}} 
\newcommand{\ma}{\boldsymbol{\alpha}} 
\newcommand{\mW}{\mathbf{W}} 
\newcommand{\mY}{\mathbf{Y}} 
\newcommand{\mzero}{\mathbf{0}} 

\newcommand{\sK}{\mathscr{K}} 
\newcommand{\sP}{\mathscr{P}} 
\newcommand{\sI}{\mathscr{I}} 

\newcommand{\erf}{\text{erf}}
\newcommand{\vect}[1]{\ensuremath{\mathbf{#1}}}
\def\multiset#1#2{\ensuremath{\left(\kern-.3em\left(\genfrac{}{}{0pt}{}{#1}{#2}\right)\kern-.3em\right)}}

\newcommand*\Laplace{\mathop{}\!\mathbin\bigtriangleup}
\newcommand{\rvline}{\hspace*{-\arraycolsep}\vline\hspace*{-\arraycolsep}}

\newcommand{\rd}{\mathrm{d}} 

\makeatletter
\newcommand{\bigplus}{%
  \DOTSB\mathop{\mathpalette\mattos@bigplus\relax}\slimits@
}
\newcommand\mattos@bigplus[2]{%
  \vcenter{\hbox{%
    \sbox\z@{$#1\sum$}%
    \resizebox{!}{0.9\dimexpr\ht\z@+\dp\z@}{\raisebox{\depth}{$\m@th#1+$}}%
  }}%
  \vphantom{\sum}%
}
\makeatother

\DeclareMathOperator*{\spann}{span}
\DeclareMathOperator*{\argmin}{\arg\min}
\DeclareMathOperator*{\argmax}{\arg\max}

\theoremstyle{plain}
\newtheorem{theorem}{Theorem}[section]
\newtheorem{lemma}[theorem]{Lemma}
\newtheorem{corollary}[theorem]{Corollary}
\newtheorem{proposition}[theorem]{Proposition}
\newtheorem{definition}[theorem]{Definition}
\newtheorem{remark}[theorem]{Remark}

\begin{document}

\pagestyle{fancy}
\fancyhead{}
\renewcommand{\headrulewidth}{0pt} 
\fancyhead[CE]{H. Yun \& V.~M. Panaretos}
\fancyhead[CO]{Computerized Tomography and RKHS}

\begin{frontmatter}
\title{{\large Computerized Tomography and Reproducing Kernels}}

\begin{aug}
\author[A]{\fnms{Ho} \snm{Yun}\ead[label=e1]{ho.yun@epfl.ch}}
\and
\author[A]{\fnms{Victor M.} \snm{Panaretos}\ead[label=e2,mark]{victor.panaretos@epfl.ch}}

\thankstext{t1}{Research supported by a Swiss National Science Foundation grant.}
\address[A]{Ecole Polytechnique F\'ed\'erale de Lausanne, \printead{e1,e2}}
\end{aug}

\begin{abstract}
The X-ray transform is one of the most fundamental integral operators in image processing and reconstruction. In this article, we revisit the formalism of the X-ray transform by considering it as an operator between Reproducing Kernel Hilbert Spaces (RKHS). Within this framework, the X-ray transform can be viewed as a natural analogue of Euclidean projection. The RKHS framework considerably simplifies  projection image interpolation, and leads to an analogue of the celebrated representer theorem for the problem of tomographic reconstruction. It leads to methodology that is dimension-free and stands apart from conventional filtered back-projection techniques, as it does not hinge on the Fourier transform. It also allows us to establish sharp stability results at a genuinely functional level (i.e. without recourse to discretization), but in the realistic setting where the data are discrete and noisy. The RKHS framework is versatile, accommodating any reproducing kernel on a unit ball, affording a high level of generality. When the kernel is chosen to be rotation-invariant, explicit spectral representations can be obtained, elucidating the regularity structure of the associated Hilbert spaces. Moreover, the reconstruction problem can be solved at the same computational cost as filtered back-projection.
\end{abstract}


\end{frontmatter}

\maketitle


\section{Introduction}
The purpose of this paper is to introduce a novel framework for tomography based on Reproducing Kernel Hilbert Spaces (RKHS), with the aim of addressing important challenges inherent in the problem of tomographic reconstruction. At the core of our approach lies the reformulation of the \emph{domain space} of the X-ray transform: instead of a (weighted) $\cL_{2}$ space, we take the domain as being an RKHS. We demonstrate that doing so offers substantial conceptual and mathematical advantages, particularly when relating continuum models with finite, discrete, and noisy data. Importantly, the new framework is straightforward and arguably simpler than traditional settings, and its development and theory rely essentially only on basic Hilbert space theory and elementary arguments. 

Since its inception in the 1970s, the technique of Computerized Tomography (CT) has evolved considerably,  becoming a vital tool in fields ranging from radiology to structural biology, and various scientific disciplines \cite{sorenson1987physics, frank2006three}. This method enables the visualization of an object's internal features through projection images. At the core of this technology lies the X-ray transform $\sP_{R}$, a mathematical operation that, at a given orientation $R \in SO(n)$, maps an original object $f: \bR^{n} \rightarrow \bR$ to a tomograph through a line integral:
\begin{align*}
    \sP(R, \vect{x}) = \sP_{R} f(\vect{x}) =
    \int_{-\infty}^{+\infty} f(R^{\top} [\vect{x}:z]) \rd z, 
\end{align*}
where $[\vect{x}:z] \in \bR^{n}$ denotes the concatenation of $\vect{x} \in \bR^{n-1}$ with $z \in \bR^{1}$. The X-ray transform, in essence, calculates the line integral of an object's density along the axis $\vect{r}:=R^{\top} e_{n}$ of orientation. The entire collection of projection images $g = \sP f$ is often referred to as a sinogram because the X-ray transform of an off-center point source generates a sinusoidal wave pattern, as depicted in the middle of \cref{sino::img}. Specifically, a sinogram should satisfy two essential conditions: the \emph{compatibility principle} and the \emph{moment condition}, commonly referred to jointly as the Helgason-Ludwig Consistency Conditions (HLCC, see \cref{HLCC}).
The fundamental task in tomography is to reconstruct an unknown function $f^{0}$ from a series of discretized projection images captured at finitely many orientations $R_{i} \in SO(n)$ and mesh points $\vect{x}_{j} \in \bR^{n-1}$, with noise contamination $\varepsilon_{ij} \stackrel{iid}{\sim} N(0, \sigma^{2})$, as shown in the right side of \cref{sino::img}:
\begin{equation}\label{eq:data:discrete}
    y_{ij} = \sP_{R_{i}} f^{0}(\vect{x}_{j}) + \varepsilon_{ij}, \quad 1 \le i \le N, \, 1 \le j \le M.
\end{equation}

\begin{figure}[t]
\centering
\includegraphics[width=\textwidth, height=4cm]{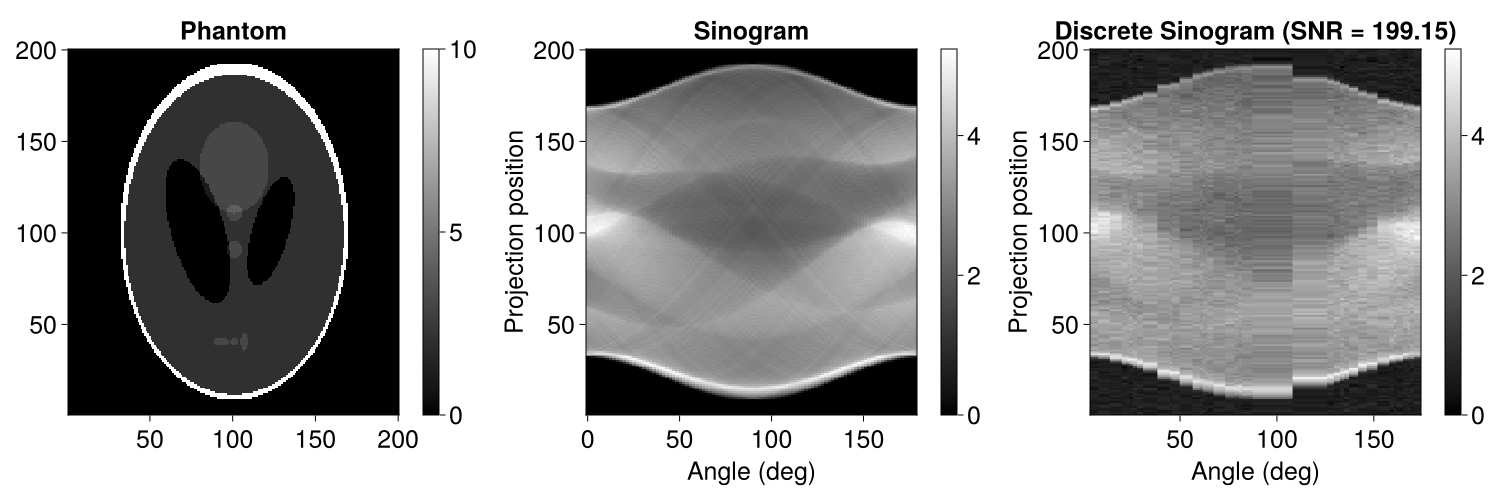}
\caption{Sinogram (middle) of the 2D Shepp-Logan phantom (left) over $180^{\circ}$, and its discretization (right) at randomly chosen $N=40$ angles with noise level $\sigma = 20$.}
\label{sino::img}
\end{figure}

While the X-ray transform itself $f \mapsto \sP_{R} f$ does not intrinsically depend on the domain and the range (as long as the line integral is well-defined), the analytical properties of the operator $\sP_{R}$ do, and thus the underlying image to be reconstructed depends on the chosen domain \cite{hanke2017taste}.
The X-ray transform coincides with the Radon transform when $n=2$, and a variety of reconstruction methods have been developed within the weighted $\cL_{2}$ setting, see \cite{natterer2001mathematics} and references therein. 
Departing from this conventional setup, we alter the domain to be an RKHS $\bH(K)$ equipped with a reproducing kernel $K$. We then show that $\sP_{R}$ becomes a surjective operator onto an induced RKHS $\bH(\tilde{K}^{R})$.
Notably, considering $\sP_{R}: \bH(K) \rightarrow \bH(\tilde{K}^{R})$ as an operator between RKHSs elicits a compelling analogy to the Euclidean projection. This is conceptually intuitive, as the X-ray transform simplifies to the Euclidean projection when examining a point mass. Several common properties align between the two, such as the adjoint operator (called the backprojection) acting as an isometry and constituting the right inverse of the X-ray transform: for a given  projection image $g_{R}: \bB^{n-1} \rightarrow \bR$, we have $\| \sP_{R}^{*} g_{R}\|_{\bH(K)}=\|g_{R}\|_{\bH(\tilde{K}^{R})}$ and $\sP_{R} (\sP_{R}^{*} g_{R})=g_{R}$. 
This property allows us to 
derive a \emph{dimension-free} representer theorem for 
\begin{align}\label{eq:tiko:intro}
    \hat{f}^{\nu}
    = \argmin_{f \in \bH(K)} \sum_{i=1}^{N} \sum_{j=1}^{M} &\left(y_{ij}- \sP_{R_{i}} f(\vect{x}_{j}) \right)^{2} +\nu \|f\|^{2}_{\bH(K)}, \quad \nu > 0,
\end{align}
which recasts the reconstruction problem as a penalized regression problem, with $\hat{f}^{\nu}$ being the reconstruction of $f^{0}$. The corresponding normal equation $(\mW+\nu \mI) \hat{\ma}^{\nu}= \mY$ involves a Gram matrix $\mW$ which can be \emph{precomputed} in tomographic hardware for fast computation. 

The RKHS framework affords a high degree of abstraction and flexibility, imposing no specific conditions on a reproducing kernel on the unit ball for reconstruction. The desired level of smoothness and oscillation of the solution could be tailored by manipulating the reproducing kernels. For instance, if we consider the $\cL_{2}$ norm of a function's Laplacian as a measure of its smoothness \cite{zhao2013fourier}, we may use a thin-plate spline derived from the Laplacian within the unit ball under the vanishing Dirichlet boundary condition \cite{iyer2016smoothing, wahba1981spline}. Compared to conventional approaches, the RKHS framework presents notable advantages on several fronts:

\begin{figure}[t]
\centering
    \includegraphics[width=\linewidth, height=6cm] {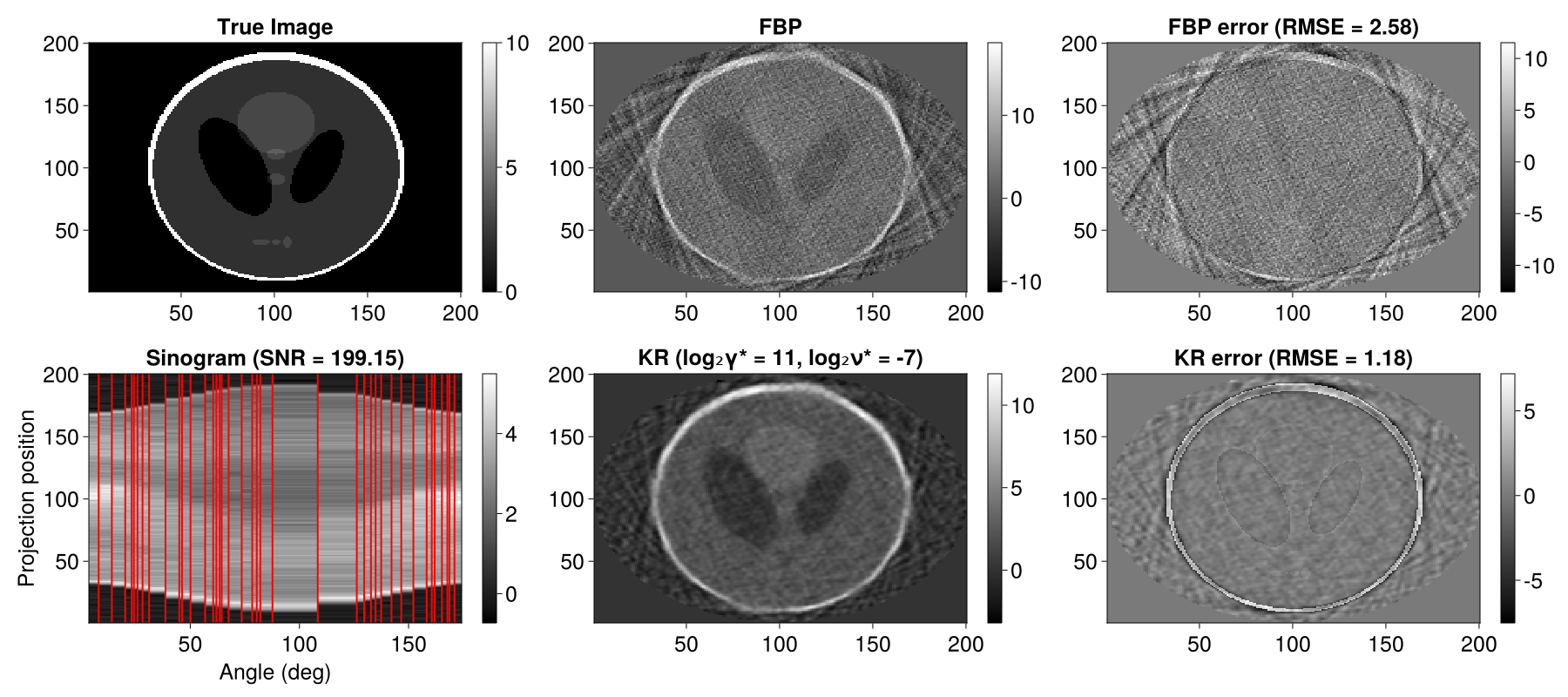}
\caption{FBP (top middle) and KR (bottom middle) reconstructions from a perturbed sinogram (bottom left) with $\sigma = 20$, along with the error plots for FBP (top right) and KR (bottom right). The red vertical lines in the sinogram  indicate the location of randomly chosen $N=40$ angles. For KR, we explore several choices for the parameter $\gamma > 0$ for Gaussian kernels $K(\vect{z}_{1}, \vect{z}_{2}) = \exp(-\gamma \|\vect{z}_{1}-\vect{z}_{2}\|^{2})$ and the penalty parameter $\nu > 0$ for the Tikhonov regularization, of which optimal values are found to be $\gamma^{*} = 2^{11}$ and $\nu^{*} = 2^{-7}$.}
\label{RKHS::error::0_20}
\end{figure}

\begin{itemize}[leftmargin=*]
\item \textit{It circumvents the inverse, dispenses with (Fourier) filtering and is dimension-indifferent}: In conventional approaches using a weighted $\cL_{2}$ domain space framework, na\"ive backprojection is \emph{not} the right inverse of the X-ray transform. To obtain the inverse, the sinogram must be filtered in the Fourier domain before backprojection -- hence the term \emph{Filtered Backprojection (FBP)}-- an approach upon which the majority of reconstruction algorithms rely. Such filtering is a local operation when the dimension $n$ is odd, but is global when $n$ is even, as it involves the Hilbert transform \cite{natterer2001mathematics}. Even so, the inverse operator $\sP^{-1}$ is unbounded, making reconstruction $f = \sP^{-1} g$  vulnerable to small measurement errors, even for a complete sinogram  \cite{hadamard1902problemes}. When only a finite number of projection images $\sP_{R_{i}}f, \, 1 \le i \le N$ can be observed, the inversion formula for the partial sinogram becomes impractical in dimensions $n \ge 3$, as the X-ray transform has a large null space \cite{natterer1980sobolev}. This requires additional regularization, e.g. in the form of truncating the transform's Singular Value Decomposition (SVD) \cite{izen1988series,maass1987x}. 

\hspace{5pt} In comparison, our Kernel Reconstruction (KR) operates directly in the spatial domain without any filtering techniques, as $\sP_{R} \circ \sP_{R}^{*}$ is the identity, ensuring consistent methodology across different dimensions $n$. Regularization is automatically enforced (in Tikhonov form) through the representer theorem, circumventing the approximation of the inverse to the complete sinogram. KR also eliminates the need for band-limitedness assumptions, commonly made in signal-processing approaches \cite{natterer2001mathematics, stefanov2022radon, stefanov2020semiclassical, monard2023sampling}. Consequently, the conventional Nyquist-Shannon sampling scheme does not strictly apply within our setup, which results in fewer artifacts. Finally, the RKHS framework can encompass the truncated SVD algorithm within the $\cL_{2}$ framework, so can be seen as a generalization (truncated SVD corresponds to a finite-rank reproducing kernel, see \cref{left::ang::ex}, which also crystallizes the drawbacks of the SVD approach).

\item \textit{It seamlessly and consistently links the continuum unknown with the discrete data}: To address practical scenarios involving discretization and errors 
(as in \eqref{eq:data:discrete}), the conventional $\cL_{2}$ setup typically follows one of two approaches. The \emph{functional} approach first interpolates the projection images to generate a sinogram adhering to the HLCC by restricting resolution \cite{huang2017restoration}; and, subsequently performing FBP to the resulting interpolant to recover the original image. On the other hand, the \emph{discrete} approach conceptualizes the object as an $n$-dimensional tensor, with discretized variants of FBP employing the Fast Fourier transform (FFT). Iteration is usually required, because such discretized algorithms typically violate the HLCC \cite{zhou2020limited, zhang2016low}. 

\hspace{5pt} In stark contrast, KR is a one-step algorithm that simultaneously addresses sinogram reconstruction (projection image interpolation) and the retrieval of the original structure, in strict adherence to the HLCC (see \cref{{RKHS:HLCC}}). Unlike the approach in \cite{monard2019efficient, huang2017restoration}, which fixes a finite-dimensional functional space \textit{a priori}, KR is intrinsically functional: the optimal solution of \eqref{eq:tiko:intro} over an infinite-dimensional function space is shown \emph{a posteriori} to belong to an explicit finite-dimensional space, 
tailored to the measurement setup.

\item \textit{It comprehensively outperforms FBP with no additional computational overhead}: KR is computationally efficient, taking only a few seconds for a single run ($\approx 0.3$ secs for $M = 100$ and $N = 40$, using 4 threads on Apple M1 Pro Chip). Additionally, KR consistently outperforms FBP across all configurations in numerical experiments, with the performance gap widening in adversarial setups. With noisy data, KR exhibits improved resistance over FBP without brute-force iterations. As angular irregularity increases, FBP algorithms often introduce severe artifacts, while KR maintains less pronounced artifacts, as illustrated in \cref{RKHS::error::0_20}. 

\hspace{5pt} This is noteworthy because the FBP reconstruction is widely used to initialize modern iterative methods leveraging deep learning to reduce scanning views while maintaining or improving reconstructed image quality \cite{jia2011gpu, jin2017deep, adler2017solving}. Our findings show that KR may well be preferable for initializing these methods.

\item \textit{It furnishes ``honest'' and non-asymptotic guarantees in the form of stability}: 
The non-asymptotic stability of discrete algorithms at a functional level is often unclear due to their operating in the Fourier domain, making the passage to continuum error estimation challenging. Several works \cite{stefanov2022radon, stefanov2020semiclassical, monard2023sampling, katsevich2024analysis}  have examined the \emph{asymptotic} stability of the discretized FBP algorithm to address its artifacts, assuming conditions like semi-classical band-limitedness or specific function smoothness. As for state-of-the-art neural network methods, stability results are still elusive, with performance guarantees come in the form of heuristics or numerical evidence.

\hspace{5pt} Against this background, the representer theorem simplifies functional reconstruction into a finite-dimensional setting, enabling us to present a non-asymptotic sharp stability inequality. \cref{stable::recons} reveals that the quality of KR in the worst-case scenario is inversely proportional to the sum of  the regularization parameter $\nu>0$ and the least non-zero eigenvalue of the Gram matrix $\mW$. Our stability theory is distinctive in being ``honest" to both the functional nature of the object $f^{0}$ and to the discrete/noisy nature of the tomographic data $y_{ij}$: it does not simplify the problem by discretizing the functional object, nor does it assume ideal/noiseless continuum observations.
\end{itemize}

The paper is organized as follows. \cref{sec:prelim} reviews basic properties of the Euclidean projections and the $\cL_{2}$ framework. \cref{sec::RKHS} introduces the RKHS framework and the Kernel Reconstruction (KR) method. This includes its theoretical stability analyses and  illustrative examples comparing to FBP.
\cref{sec:invker} focuses on rotation-invariant kernels, which induce projections belonging to the same RKHS, invariant to orientation.
Notably, versions of Schoenberg's theorem (\cref{scheon::gen}) and Mercer's theorem (\cref{Xray::Aronszajn}) are developed, providing a comprehensive spectral characterization. 
\cref{sec:circulant} investigates how the parallel geometry in the 2D plane simplifies solution of the normal equation $(\mW+\nu \mI) \hat{\ma}^{\nu}= \mY$ by way of the FFT, leading to complexity $O(NM^{2}+NM \log N)$, on par with FBP. 
\cref{sec:fur} describes how our setting can be straightorwardly extended to cover the Radon transform, and more generally the $k$-plane transform \cite{christ1984estimates}. \cref{sec:Gauss:ker} demonstrates how $\mW$ can be easily computed when the chosen kernel is a (truncated) Gaussian kernel. The proofs of all the results presented in the main paper are contained in the appendix. The appendix also contains additional experimental results, and summarizes aspects of group theory and harmonic analysis (\cref{sec::sphe::har}) made use of in  our theory.

\section{Preliminaries}\label{sec:prelim}
Throughout the paper, we assume that $n \ge 2$ and we adhere to the following notation: 
\begin{itemize}
    \item $\bR^{n}$ is the $n$-dimensional Euclidean space.
    \item $\bS^{n-1}=\{\vect{z} \in \bR^{n}: \|\vect{z}\| = 1\}$ is the unit sphere in $\bR^{n}$.
    \item $\bB^{n}=\{\vect{z} \in \bR^{n}: \|\vect{z}\| \le 1\}$ is the closed unit ball in $\bR^{n}$.   
    \item $w(s):=\sqrt{1-s^2}, \, |s| \le 1$ is a weight function on $[-1, +1]$.
   
    \item $W(\vect{x}):=w(\|\vect{x}\|)=\sqrt{1-\|\vect{x}\|^2}, \, \vect{x} \in \bB^{n-1}$ is a weight function on $\bB^{n-1}$.
     \item $\{e_k\}_{k=1}^{n}$ is the canonical basis of $\mathbb{R}^n$, i.e. $e_k$ has a 1 in the $k$th coordinate and zeroes elsewhere. 
    \item $J \in O(n)$ is the reflection along the $e_{n}$ axis given by $J:=
        \begin{pmatrix}
    \mI_{n-1} & \rvline & \mzero \\
    \hline
    \mzero & \rvline & -1
    \end{pmatrix}$.
    \item $E(\theta) \in SO(n)$ is the Euler rotation matrix for $\theta \in \bS^{n-1}$ satisfying $E(\theta)^{\top} e_{n}=\theta$, see \cref{sec:SO(n)} for its construction.
    \item For a given operator $\sP : X \rightarrow Y$, the null space and the range space are denoted by $\cN(\sP) \subset X$ and $\cR(\sP) \subset Y$, respectively.
\end{itemize}

Additionally, $\vect{z}$ and $\vect{x}$ represent an $n$-dimensional and an $(n-1)$-dimensional vector, respectively.
The \textit{tilde} symbol signifies the same notion, mutis mutandis, but on the reduced dimension: for instance, $\tilde{J} :=
\begin{pmatrix}
    \mI_{n-2} & \rvline & \mzero \\
    \hline
    \mzero & \rvline & -1
\end{pmatrix} \in O(n-1)$ represents the reflection along the $e_{n-1}$ axis.

\subsection{Euclidean Projection}\label{sec::Euclid}
To illustrate the similarity between plain Euclidean projection and the RKHS-based X-ray transform in \cref{sec::RKHS}, we collect and review some rudimentary properties of orthogonal projections in Euclidean spaces.

Given an orientation $R \in SO(n)$, let $\vect{r}=R^{\top}e_{n} \in \bS^{n-1}$ denote its axis of rotation. For any $\vect{z} \in \bR^{n}$, there is a unique element $\vect{x} \in \bR^{n-1}$ satisfying $\vect{z} = R^{\top} [\vect{x}:z] \ \Longleftrightarrow \ R\vect{z} = [\vect{x}:z]$ for some $z \in \bR$, which is seen to be $z=\vect{z}^{\top} \vect{r}$. As $\vect{x}$ is a linear function of $\vect{z}$, we can denote this operator as $\mP_{R}:\bR^{n} \rightarrow \bR^{n-1}$, where $\vect{x}=\mP_{R} (\vect{z})$. In matrix notation, we obtain $\mP_{R} =
\begin{pmatrix}
\mI_{n-1} & \rvline & \mzero
\end{pmatrix} \cdot R  \in \bR^{(n-1) \times n}$ and $\mP_{R}^{*} = \mP_{R}^{\top}$, so the following properties are straightforward:
\begin{proposition}\label{Euclid::image} For any $R \in SO(n)$, $\mP_{R}: \bR^{n} \rightarrow \bR^{n-1}$ is a contractive, surjective, linear map. Its adjoint operator $\mP_{R}^{*}: \bR^{n-1} \rightarrow \bR^{n}$ is an isometry. Also, $\cR(\mP_{R}^{*}) = \cN(\mP_{R})^{\perp}= \vect{r}^{\perp}$  and the following holds:
\begin{enumerate}
    \item $\mP_{R} \circ \mP_{R}^{*}: \bR^{n-1} \rightarrow \bR^{n-1}$ is the identity.
    \item Any $\vect{z} \in \bR^{n}$ can be uniquely expressed as a direct sum:
    \begin{equation*}
        \vect{z}=\mP_{R}^{*} \circ \mP_{R}(\vect{z}) \oplus (\vect{z}-(\mP_{R}^{*} \circ \mP_{R}(\vect{z})) \in \vect{r}^{\perp} \oplus \cN(\mP_{R}),
    \end{equation*}
    thus $\vect{z} \in \vect{r}^{\perp}$ if and only if $\vect{z}=\mP_{R}^{*} \circ \mP_{R}(\vect{z})$.
    \item For $\vect{x} \in \bR^{n-1}$, we have
    \begin{align*}
        &\left\{\vect{z} \in \bR^{n}: \vect{x} =\mP_{R} (\vect{z}) \right\}=\mP_{R}^{*} \vect{x} \oplus \cN(\mP_{R}),  \\
        &\|\vect{x}\|=\|\mP_{R}^{*} (\vect{x})\|= \min \left\{\|\vect{z}\| : \vect{x} =\mP_{R} (\vect{z}) \right\}.
    \end{align*}
    \item Euclidean projection is obtained as $\mP_{R}=\mP_{I} \cdot R: \bR^{n} \rightarrow \bR^{n-1}$.
    \item Backprojection is obtained as $\mP_{R}^{*}=R^{\top} \cdot \mP_{I}^{*}: \bR^{n-1} \rightarrow \bR^{n}$.
\end{enumerate} 
\end{proposition}

\subsection{The X-ray Transform}\label{ssec:xray}
Though there are other integral transforms used in tomographic image processing, for the purposes of this paper, our exclusive attention is directed toward the X-ray transform. We remark that when $n=2$, the Radon transform and the X-ray transform coincide. Further information on these integral transforms can be found in \cite{solmon1976x, hanke2017taste, natterer2001mathematics} and references therein. 

\begin{definition}
Let $f \in \cL_{2}(\bR^{n})$. For $(R,\vect{x}) \in SO(n) \times \bR^{n-1}$, the X-ray transform is given by
\begin{align*}
    \sP f(R, \vect{x}) = \sP_{R} f(\vect{x}) =
    \int_{-\infty}^{+\infty} f(R^{\top} [\vect{x}:z]) \rd z,
\end{align*}
provided that the integral exists.
\end{definition}

\begin{remark}
When $n>2$, the above formula is integrable over a.e. $(R,\vect{x})$, see \cite{smith1975lower}. In most situations, the object of our interest admits compact support, and the above integral is a.e. finite. By restricting the domain into $\cL_{2}(\bB^{n})$, 
we can consider the X-ray transform at orientation $R \in SO(n)$ as $\sP_{R}:\cL_{2}(\bB^{n}) \rightarrow \cL_{2}(\bB^{n-1})$ given by
\begin{align*}
    \sP_{R} f(\vect{x}) =
    \int_{I_{\|\vect{x}\|}} f(R^{\top} [\vect{x}:z]) \rd z, \quad \vect{x} \in \bB^{n-1}
\end{align*}
where the interval of integration is given by $I_{\|\vect{x}\|}:=[-W(\vect{x}), +W(\vect{x})]$, independently of $R \in SO(n)$.
\end{remark}

We establish an equivalence relation on $SO(n)$ as follows: $R_{1} \sim R_{2}$ if
\begin{align}\label{rot::equiv}
    \cR(\mP_{R_{1}}^{*}) = \cR(\mP_{R_{2}}^{*}) \quad \Longleftrightarrow \quad \vect{r}_{1}^{\perp} = \vect{r}_{2}^{\perp} \quad \Longleftrightarrow \quad \vect{r}_{1}^{\top}\vect{r}_{2} = \pm 1,
\end{align}
where $\vect{r}_{1}: =R_{1}^{\top} e_{n}$ and $\vect{r}_{2}:=R_{2}^{\top} e_{n}$ represent the axes of rotation $R_{1}$ and $R_{2}$, respectively.  Due to \cref{Euler::repre} in the appendix, the equivalence class of $R \in SO(n)$ is given by $[R]=[E(\vect{r})]=[I]E(\vect{r})$, where $E(\vect{r})$ is the Euler matrix for $\vect{r}=R^{\top} e_{n}$ and
\begin{align*}
    [I] = \left\{
    \begin{pmatrix}
    \tilde{R} & \rvline & \mzero \\
    \hline
    \mzero & \rvline & 1
    \end{pmatrix}, \,
    \begin{pmatrix}
    \tilde{J} \tilde{R} & \rvline & \mzero \\
    \hline
    \mzero & \rvline & -1
    \end{pmatrix} : \tilde{R} \in SO(n-1) \right\}.
\end{align*}
In an ideal noise-free scenario, the information we obtain from $\sP_{R_{1}}f$ and $\sP_{R_{2}}f$ is virtually identical, assuming their orientation axes are parallel: the range of Euclidean backprojection $\cR(\mP_{R_{1}}^{*})=\cR(\mP_{R_{2}}^{*})=\vect{r}_{1}^{\perp}$ are the same, and one can easily show that $\sP_{R_{1}} f(\vect{x}_{1}) = \sP_{R_{2}} f(\vect{x}_{2})$ whenever $\mP_{R_{1}}^{*} (\vect{x}_{1}) = \mP_{R_{2}}^{*} (\vect{x}_{2}) \in \vect{r}_{1}^{\perp}$. 
We shall refer to this property as the \textit{compatibility principle} henceforth. 
In this regard, while the Grassmannian $Gr(1, n)$ may offer the non-redundant parameterization for the viewing angle instead of $SO(n)$ \cite{solmon1976x}, we employ the latter to accommodate a more general framework which encompasses the setting of cryo-Electron Microscopy (cryo-EM) \cite{anden2018structural,singer2018mathematics, panaretos2009random}. Here,  biological macromolecules are imaged subject to random orientations, without any control of the within-plane angle. 

\begin{remark}\label{xray::Euclid}
In the case where the original structure $f=I_{\vect{z}}$ is a unit point mass at some $\vect{z} \in \bR^{n}$, then the X-ray transform simplifies to the orthogonal projection in Euclidean space in Section \ref{sec::Euclid}:
\begin{align*}
    \sP_{R} I_{\vect{z}}(\vect{x}) &=
    \int_{-\infty}^{+\infty} I_{\vect{z}}(R^{\top} [\vect{x}:z]) \rd z \\
    &= \delta (\vect{x} \in \bR^{n-1}: R^{\top}[\vect{x}:(\vect{z}^{\top} \vect{r})]=\vect{z})= \delta_{\mP_{R}(\vect{z})} (\vect{x}). 
\end{align*}
\end{remark}

\subsection{Conventional \texorpdfstring{$\cL_{2}$}{Lg} framework}\label{sec::convention}

Recall that we also write $\sP f(R, \vect{x}) = \sP_{R} f(\vect{x})$, and consider the X-ray transform that maps a function on $\bR^{n}$ to $SO(n) \times \bR^{n-1}$.
As we employ a different parametrization $(R,\vect{x}) \in SO(n) \times \bR^{n-1}$, in contrast to the more conventional parametrization $(\theta, \vect{z}) \in \bS^{n-1} \times \theta^{\perp}$ utilized in tomography \cite{solmon1976x, natterer2001mathematics}, we restate certain key properties. For the proofs of the statements in this subsection, please refer to \cref{sec::proof::convention} and \cite{natterer2001mathematics}.

\begin{proposition}\label{L2::conti}
The X-ray transform with respect to the given domains are bounded:
\begin{align*}
    \sP_{R} : \cL_{2}(\bB^{n}) \rightarrow \cL_{2}(\bB^{n-1}), \quad   
    \sP : \cL_{2}(\bB^{n}) \rightarrow \cL_{2}(SO(n) \times \bB^{n-1}).
\end{align*}
Hence, they admit bounded adjoint operators, given by
\begin{equation}\label{L2::adj}
   (\sP_{R}^{*} g)(\vect{z})=g(\mP_{R}(\vect{z})), \quad    
   (\sP^{*} g)(\vect{z})=\int_{SO(n)} g(R, \mP_{R}(\vect{z})) \rd R,    
\end{equation}
where $\rd R$ denotes the normalized Haar measure on $SO(n)$.
\end{proposition}

\begin{remark}
One can prove in the same way that $\sP_{R}: \cL_{2}(\bB^{n}) \rightarrow \cL_{2}(\bB^{n-1}, W^{-1})$ is also continuous \cite{natterer2001mathematics}, with the adjoint operator now given by 
\begin{equation*}
    \sP_{R}^{*} g(\vect{z})=W^{-1}(\mP_{R}(\vect{z})) g(\mP_{R}(\vect{z})).
\end{equation*}
Meanwhile, the X-ray transform becomes unbounded if we consider $\cL_{2}(\bR^{n})$ as the domain instead. Yet, the formal adjoint operator with respect to the $\cL_{2}$ norm is still given by \eqref{L2::adj}.
\end{remark}

Let $\cS(\bR^{n})$ denote the Schwarz space. For $f \in \cS(\bR^{n})$, the Fourier transform $\hat{f}$ and the inverse Fourier transform $\check{f}$ are defined by
\begin{align*}
    \hat{f}(\xi):=(2 \pi)^{-n/2} \int_{\bR^{n}} e^{-\imath \vect{z}^{\top} \xi} f(\vect{z}) \rd \vect{z}, \quad 
    \check{f}(\vect{z}):=(2 \pi)^{-n/2} \int_{\bR^{n}} e^{\imath \vect{z}^{\top} \xi} f(\xi) \rd \xi, \quad \vect{z}, \xi \in \bR^{n},
\end{align*}
and the Fourier inversion formula reads $(\hat{f})\check{}=(\check{f})\hat{}=f$ for any $f \in \cS(\bR^{n})$.

\begin{theorem}[Fourier slice theorem (FST)]\label{Fou::slice}
For $f \in \cS(\bR^{n})$ and $R \in SO(n)$,
\begin{equation*}
    \widehat{\sP_{R} f}(\eta)=\sqrt{2 \pi} \hat{f}(\mP_{R}^{*}(\eta)), \quad \eta \in \bR^{n-1}.
\end{equation*}
Thus, for any $R_{1}, R_{2} \in SO(n)$, whenever $\mP_{R_{1}}^{*}(\eta_{1})= \mP_{R_{2}}^{*}(\eta_{2}) \in \vect{r}_{1}^{\perp} \cap \vect{r}_{2}^{\perp}$, we have
\begin{equation*}
    \widehat{\sP_{R_{1}} f}(\eta_{1})=\widehat{\sP_{R_{2}} f}(\eta_{2}), \quad \eta_{1}, \eta_{2} \in \bR^{n-1}.
\end{equation*}
\end{theorem}

The FST reveals the \emph{common line property}:  two (non-identical) central slices of the Fourier transform intersect over a hyperplane $\vect{r}_{1}^{\perp} \cap \vect{r}_{2}^{\perp}$ of dimension $n-2$, on which their Fourier transforms agree. In cryo-EM, the FST plays a crucial role in determining the relative angle between two projection images when the signal-to-noise ratio (SNR) is relatively high \cite{van1987angular}. However, the FST trivializes when $n=2$ since the Fourier transform at the origin accounts for the total mass.

If practitioners had access to a perfect sinogram $\sP f$ with zero measurement error, a reconstruction would be obtained via an explicit inversion formula, as presented below. For $ \alpha < n$, let $\sI^{\alpha}$ be the Riesz potential defined by
\begin{equation*}
    (\sI^{\alpha} f) \hat \, (\xi) := |\xi|^{-\alpha} \hat{f}(\xi), \quad f \in \cS(\bR^{n}).
\end{equation*}

\begin{theorem}[Inversion Formula, \cite{natterer2001mathematics}]\label{inv::for}
Let $f \in \cS(\bR^{n})$. For $\alpha < n$, the inversion formula for the sinogram is given by
\begin{align*}
    f = \frac{|\bS^{n-1}|}{2 \pi |\bS^{n-2}|}  \sI^{-\alpha} \sP^{*} \sI^{\alpha-1} g, \quad g=\sP f.
\end{align*}
\end{theorem}

The inversion formula obviates how na\"ive backprojection is not sufficient for reconstruction: performing an X-ray transform after backprojecting the sinogram by no means provides the original sinogram, i.e. $g \neq \sP \sP^{*} g$.
Rather,  to recover the image, one has to filter the  projection images, and this operation is often called the \emph{filtered back projection} (FBP). Although the formula ensures unique recovery when a complete sinogram is available, practical scenarios will involve incomplete data due to limited scanning views. This prompts the question of what conditions a bona fide sinogram must satisfy, to achieve reliable reconstructions from partial sinograms. See \cite{natterer2001mathematics} for the proof of the following theorem.

\begin{theorem}[Helgason-Ludwig Consistency Condition (HLCC), \cite{ludwig1966radon, helgason1965radon}]\label{HLCC}
For $g \in \cS(SO(n) \times \bB^{n-1})$, there is some $f \in \cS(\bB^{n})$ such that $g=\sP f$ if and only if 
\begin{enumerate}
    \item If $R_{1} \sim R_{2} \in SO(n)$ and $\mP_{R_{1}}^{*} (\vect{x}_{1}) = \mP_{R_{2}}^{*} (\vect{x}_{2}) \in \vect{r}_{1}^{\perp}$, then $g(R_{1}, \vect{x}_{1})=g(R_{2}, \vect{x}_{2})$.
    \item For any $l \in \bZ_{+}$, there is some homogeneous polynomial $q_{l} \in \cP_{l}(\bR^{n})$ of degree $l$, independent on $R \in SO(n)$, such that
    \begin{equation*}
        \int_{\bR^{n-1}} (\vect{x}^{\top} \vect{y})^{l} g(R,\vect{x}) \rd \vect{x} = q_{l}(\mP_{R}^{*}(\vect{y})), \quad \vect{y} \in \bR^{n-1}.
    \end{equation*}
\end{enumerate}
\end{theorem}

The consistency condition stipulates that any complete sinogram must adhere to the compatibility principle and the moment condition. Conversely, when a suitable smoothness condition is imposed on the range space to ensure that the inverse Fourier transform is well-defined, these two conditions collectively guarantee the existence of a solution \cite{helgason1965radon,solmon1976x, sharafutdinov2016reshetnyak}. In challenging scenarios involving limited viewing angles and contaminated  projection images, the standard FBP algorithm utilizing FFT approximation can result in pronounced smearing artifacts \cite{maass1987x}. In such cases,  projection image interpolation following the HLCC becomes a necessary pre-processing step before applying FBP for reconstruction \cite{basu2000uniqueness, huang2017restoration}. However, even in noiseless situations, this interpolation has its limitations, particularly concerning resolution. Specifically, when the number of viewing angles $N$ is smaller than the dimension $\binom{n+l-1}{l}$ of the space $\cP_{l}(\bR^{n})$ of homogeneous polynomials of degree $l$, there exists a non-trivial homogeneous polynomial of degree $l$ that vanishes at these $N$ viewing angles \cite{dai2013approximation}.

This observation motivates us to reinterpret the X-ray transform as an operator between two RKHS, i.e. Hilbert spaces where point evaluation corresponds to the inner product with a kernel generator. Such spaces offer a powerful mechanism for incorporating desired levels of smoothness through Green's functions, such as the Paley–Wiener space and Sobolev spaces  \cite{hsing2015theoretical, paulsen2016introduction}. 
In our RKHS-based approach, the tomographic reconstruction is elegantly transformed into a linear regression problem, exploiting the kernel trick. This novel perspective allows us to reconstruct sinograms efficiently even with limited angles. Importantly, our reconstruction algorithm within the RKHS framework does not require sinogram interpolation beforehand, as it simultaneously reconstructs both the sinogram and the original structure. Furthermore, our reconstruction method automatically satisfies both the FST and the HLCC. This notable advantage renders our reconstruction algorithm more robust and provides smoother outputs even when the data is incomplete and noisy. 

\section{RKHS Framework}\label{sec::RKHS}
Given a finite number of viewing angles, tomographic reconstruction is an ill-posed inverse problem as there are an infinite number of solutions \cite{natterer2001mathematics}. Furthermore, \cref{L2::conti} indicates that the inverse of the X-ray transform is an unbounded operator between $\cL_{2}$ spaces. Consequently, addressing the stability of an inversion algorithm has led to numerous efforts to improve error estimates by altering the function space domain, such as Sobolev spaces \cite{natterer1980sobolev} or their variants \cite{sharafutdinov2016reshetnyak, sharafutdinov2012integral}. In contrast, the RKHS framework offers a stable reconstruction approach because the backprojection is not only isometric but also \emph{measurement consistent}: for any  projection image $g: \bB^{n-1} \rightarrow \bR$, we have $\|\sP_{R}^{*} g\|=\|g\|$ and $\sP_{R} (\sP_{R}^{*} g)=g$. We develop the framework in further detail in what follows.  {For the proofs of the statements in this section, we refer to \cref{sec:proof:RKHS} in the appendix.}

\subsection{Setup}\label{ssec:Setup}
\begin{definition}
Let $\bH$ be a Hilbert space of real-valued functions defined on a set $E$. A bivariate function $K:E \times E \rightarrow \bR$ is called a reproducing kernel for $\bH$ if
\begin{enumerate}
    \item For any $\vect{z} \in E$, a generator $k_{\vect{z}}(\cdot):= K(\cdot, \vect{z})$ at $\vect{z} \in E$ belongs to $\bH$.
    \item $K$ satisfies the reproducing property: for any $f \in \bH$, the point evaluation at $\vect{z} \in E$ is given by $f(\vect{z})=\langle f, k_{\vect{z}} \rangle_{\bH}$.
\end{enumerate}
A Hilbert space equipped with a reproducing kernel is called an RKHS. \end{definition}
By the Moore–Aronszajn theorem \cite{aronszajn1950theory}, any reproducing kernel is symmetric and positive semidefinite. Conversely, a kernel with these properties induces a unique RKHS, denoted by $\bH=\bH(K)$. We refer to \cite{paulsen2016introduction,hsing2015theoretical}. 

Given a reproducing kernel $K: \bB^{n} \times \bB^{n} \rightarrow \bR$ on $\bB^{n}$ and an orientation $R \in SO(n)$, we define a bivariate function $\tilde{K}^{R}: \bB^{n-1} \times \bB^{n-1} \rightarrow \bR$ as follows: for any $\vect{x}_{1}, \vect{x}_{2} \in \bB^{n-1}$,
\begin{equation}\label{induced::ker}
    \tilde{K}^{R}(\vect{x}_{1}, \vect{x}_{2}):=
    \int_{I_{\|\vect{x}_{2}\|}} \int_{I_{\|\vect{x}_{1}\|}} K(R^{\top} [\vect{x}_{1}:z_{1}], R^{\top} [\vect{x}_{2}:z_{2}]) \rd z_{1} \rd z_{2},
\end{equation}
provided that the integral exists for all $\vect{x}_{1}, \vect{x}_{2} \in \bB^{n-1}$. The induced kernel $\tilde{K}^{R}=(\sP_{R} \times \sP_{R})K$, is obtained by performing the X-ray transform on each component. Note that $\tilde{K}^{R}$ is \emph{not} a push-forward kernel, as the X-ray transform is an integral operator, not a function between Euclidean spaces.

\begin{remark}
In practice, we often opt for $K: \bB^{n} \times \bB^{n} \rightarrow \bR$ to be continuous. The compact nature of $\bB^{n}$ ensures that $K$ is bounded, allowing the integral to be well-defined for any $\vect{x}_{1}, \vect{x}_{2} \in \bB^{n-1}$.
\end{remark}

\begin{proposition}\label{push::kernel}
Let $K: \bB^{n} \times \bB^{n} \rightarrow \bR$ be a kernel and $R \in SO(n)$. Then $\tilde{K}^{R}: \bB^{n-1} \times \bB^{n-1} \rightarrow \bR$ defined in \eqref{induced::ker} is also a kernel. If $K$ is continuous, then so is $\tilde{K}^{R}$. Moreover, if $K$ is strictly positive definite (p.d.), then the restriction of $\tilde{K}^{R}$ to the open unit ball $(\bB^{n-1})^{\circ}=\{\vect{z} \in \bR^{n}: \|\vect{z}\| < 1\}$ is also strictly p.d.
\end{proposition}

Denote by $k_{\vect{z}}(\cdot):= K(\cdot, \vect{z})$  the generator at the point $\vect{z} \in \bB^{n}$. Due to \cref{push::kernel}, the induced generator $\tilde{k}^{R}_{\vect{x}}$ corresponding to the induced kernel $\tilde{K}^{R}$ at the point $\vect{x} \in \bB^{n-1}$ is given by
\begin{equation*}
    \tilde{k}^{R}_{\vect{x}}(\cdot)=\tilde{K}^{R}(\cdot, \vect{x})=
    \int_{I_{\|\vect{x}\|}} \int_{I_{\|\cdot\|}} K(R^{\top} [\cdot:z_{1}], R^{\top} [\vect{x}:z_{2}]) \rd z_{1} \rd z_{2}.
\end{equation*}
Note that  $\tilde{k}^{R}_{\vect{x}}=0$ whenever $\vect{x} \in \partial \bB^{n-1} \Longleftrightarrow I_{\|\vect{x}\|}=\{0\}$. Consequently, any function $g \in \bH(\tilde{K}^{R})$ vanishes at the boundary. Below, we present the central theorem that grounds our framework.

\begin{theorem}\label{rkhs::image::re} For any $R \in SO(n)$, $\sP_{R}: \bH(K) \rightarrow \bH(\tilde{K}^{R})$ is a contractive, surjective, linear map. Its adjoint operator $\sP_{R}^{*}: \bH(\tilde{K}^{R}) \rightarrow \bH(K)$ is an isometry, uniquely determined by
\begin{equation}\label{adj::Boch}
    \sP_{R}^{*} ( \tilde{k}^{R}_{\vect{x}} )
    = \int_{I_{\|\vect{x}\|}} k_{R^{\top} [\vect{x}:z]} \rd z, \quad \vect{x} \in \bB^{n-1},
\end{equation}
in the sense of Bochner integration, or equivalently, for any $\vect{x} \in \bB^{n-1}, \vect{z} \in \bB^{n}$,
\begin{equation}\label{adj::ind::gen}
    \sP_{R}^{*} ( \tilde{k}^{R}_{\vect{x}} )(\vect{z})
    = \langle \sP_{R}^{*} ( \tilde{k}^{R}_{\vect{x}} ), k_{\vect{z}} \rangle_{\bH(K)} 
    =\int_{I_{\|\vect{x}\|}} K(\vect{z}, R^{\top} [\vect{x}:z]) \rd z = \sP_{R} k_{\vect{z}} (\vect{x}).
\end{equation}
Let $\bH_{R}(K):=\cR (\sP_{R}^{*})$ be the range space of the adjoint operator. Then,
\begin{enumerate}
    \item $\sP_{R} \circ \sP_{R}^{*}: \bH(\tilde{K}^{R}) \rightarrow \bH(\tilde{K}^{R})$ is the identity.
    \item $\bH_{R}(K)=(\cN(\sP_{R}))^{\perp}, (\bH_{R}(K))^{\perp}=\cN(\sP_{R})$.
    \item Any $f \in \bH(K)$ can be uniquely expressed as a direct sum:
    \begin{equation*}
        f=(\sP_{R}^{*} \circ \sP_{R})f \oplus (f-(\sP_{R}^{*} \circ \sP_{R})f) \in \bH_{R}(K) \oplus \cN(\sP_{R}),
    \end{equation*}
    thus $f \in \bH_{R}(K)$ if and only if $f=(\sP_{R}^{*} \circ \sP_{R})f$.
    \item For $g \in \bH(\tilde{K}^{R})$, we have
    \begin{align}\label{inv::soln}
        &\left\{f \in \bH(K): g=\sP_{R}f \right\}=\sP_{R}^{*} g \oplus \cN(\sP_{R}), \\ \nonumber 
        &\|g\|_{\bH(\tilde{K}^{R})}=\|\sP_{R}^{*} g\|_{\bH(K)}= \min \left\{\|f\|_{\bH(K)} : g=\sP_{R}f \right\}.
    \end{align}
\end{enumerate} 
\end{theorem}

\begin{figure}[H]
\centering
\includegraphics[width=\textwidth, height=4cm]{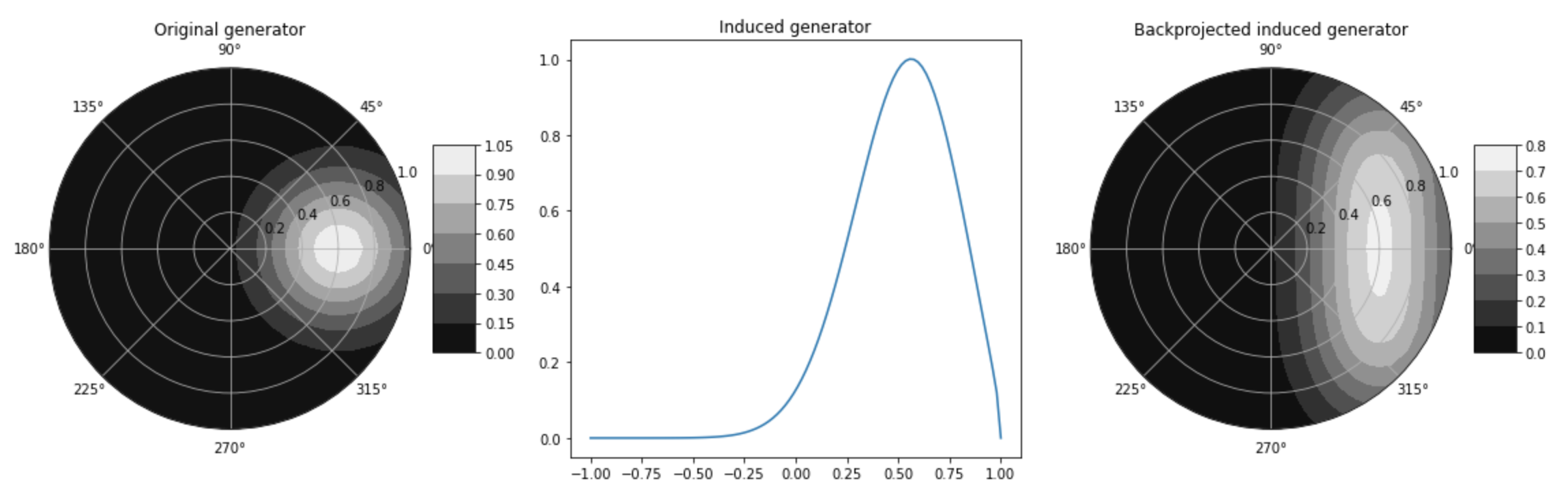}
\caption{Gaussian kernel $K: \bB^{2} \times \bB^{2} \rightarrow \bR, (\vect{z}_{1}, \vect{z}_{2}) \mapsto \exp(-\gamma \|\vect{z}_{1}-\vect{z}_{2}\|^{2})$ with $\gamma=6$, used to plot three images: the original generator $k_{\vect{z}}: \bB^{2} \rightarrow \bR$ evaluated at $\vect{z}=(0.6,0)$, the induced generator $\tilde{k}^{R}_{\vect{x}}: \bB \rightarrow \bR$ evaluated at $(R, \vect{x})= (I, 0.6)$, and its kernel backprojection $\sP_{R}^{*} ( \tilde{k}^{R}_{\vect{x}}): \bB^{2} \rightarrow \bR$ evaluated at $(R, \vect{x})= (I, 0.6)$, respectively.}
\label{gen::img}
\end{figure}

As the adjoint operator $\sP_{R}^{*}$ takes a function defined on $\bB^{n-1}$ and smears over $\bB^{n}$ to produce an a function defined on $\bB^{n}$, we call it the \emph{kernel backprojection} from this point onward, to distinguish from the usual backprojection in the $\cL_{2}$ setting. In the $\cL_{2}$ setting, $\sP_{R} \circ \sP_{R}^{*}$ does \emph{not} result in the identity operator, necessitating the use of pre-filtering. In contrast, within the RKHS framework, we draw a natural analogy with the Euclidean projection in \cref{Euclid::image}: the kernel backprojection $\sP_{R}^{*}$ is an isometry, which might be interpreted as a smoothness-preserving operator, and it serves as the right inverse of the X-ray transform $\sP_{R}$. 
Furthermore, given a finite number of viewing angles, although the reconstruction is an ill-posed problem in both the $\cL_{2}$ and the RKHS framework \cite{natterer2001mathematics}, \eqref{inv::soln} demonstrates that the kernel backprojection generates the $n$-dimensional image of minimal norm among the solution set $\left\{f \in \bH(K): g=\sP_{R}f \right\}$.

\subsection{Kernel Decomposition}\label{ssec:decomp}
To derive a representer theorem for the reconstruction in the following subsection, we introduce the decomposition of $\bH(K)$ relevant to the discrete setup. Given limited access to the sinogram, in the form of finitely many  projection images at orientations (say) $\mR:=\{R_{i} \in SO(n): \ i=1,2,\dots, N\}$ (e.g.  \cref{sino::img}, right side), we aim to reconstruct the complete sinogram and the original $n$-dimensional structure $f^{0} \in \bH(K)$. Let $\mF:= \{\vect{x}_{1}, \dots, \vect{x}_{M} \} \subset \bB^{n-1}$ denote the mesh where the transformed images $g_{i}=\sP_{R_{i}} f^{0}+ W_{i}$ are evaluated under the presence of independent white noise $W_{i}$ in $\bH(\tilde{K}^{R_{i}})$ of variance $\sigma^{2}>0$. Consequently, our observations are given by
\begin{align*}
    y_{ij}:=g_{i}(\vect{x}_{j})=\sP_{R_{i}} f^{0}(\vect{x}_{j})+\varepsilon_{ij}, \quad 1 \le i \le N, \ 1 \le j \le M, 
\end{align*}
where $\varepsilon_{ij}:=W_{i}(\vect{x}_{j})$ are i.i.d. $N(0, \sigma^{2})$. 

\begin{proposition}\label{repre::proj}
For any $R \in SO(n)$, let $\bV_{\mF}(\tilde{K}^{R}) := \spann \{\tilde{k}^{R}_{\vect{x}_{1}}, \tilde{k}^{R}_{\vect{x}_{2}}, \dots, \tilde{k}^{R}_{\vect{x}_{n}} \}$ denote the subspace of $\bH(\tilde{K}^{R})$ spanned by the induced generators at mesh points.
\begin{enumerate}
\item $\sP_{R}f(\vect{x}_{k})=0$ for all $k=1, \dots, M$ if and only if $f \in \sP^{-1}_{R}((\bV_{\mF}(\tilde{K}^{R}))^{\perp})$.
\item $\sP^{-1}_{R}((\bV_{\mF}(\tilde{K}^{R}))^{\perp})=\sP_{R}^{*}((\bV_{\mF}(\tilde{K}^{R}))^{\perp}) \oplus \cN(\sP_{R})$ is a closed subspace of $\bH(K)$.
\item $(\sP^{-1}_{R}((\bV_{\mF}(\tilde{K}^{R}))^{\perp}))^{\perp}=\sP_{R}^{*}(\bV_{\mF}(\tilde{K}^{R}))$ is a closed subspace of $\bH(K)$.
\end{enumerate}
\end{proposition}

Owing to the proposition above, we get an elegant decomposition of the domain $\bH(K)$ for any $R \in SO(n)$:
\begin{alignat}{3}\label{3::decomp}
    \bH(K)&=\sP_{R}^{*}(\bV_{\mF}(\tilde{K}^{R})) &&\oplus 
    \sP_{R}^{*}((\bV_{\mF}(\tilde{K}^{R}))^{\perp}) &&\oplus \, \cN(\sP_{R}), \\ \nonumber
    f &= \quad f_{\mF}^{R} &&\oplus \, (\sP^{*}_{R} \sP_{R} f- f_{\mF}^{R}) &&\oplus \, (f-\sP^{*}_{R} \sP_{R} f),
\end{alignat}
where $f_{\mF}^{R}:=\sum_{j=1}^{M} \alpha_{j} \sP^{*}_{R}  \tilde{k}_{\vect{x}_{j}}^{R}$    
is uniquely determined by the linear system:
\begin{equation}\label{de::facto::dim}
    \sP_{R} f(\vect{x}_{j})= \sP_{R} f_{\mF}^{R}(\vect{x}_{j})=\sum_{j'=1}^{M} \alpha_{j'} \tilde{K}^{R}(\vect{x}_{j}, \vect{x}_{j'}), \quad j=1, \dots, M.
\end{equation}
  
 The unidentifiable nature of the component in $\cN(\sP_{R})$ in \eqref{3::decomp} persists irrespective of the mesh points used. On the other hand, the identifiability (or lack thereof) of $\sP_{R}^{*}((\bV_{\mF}(\tilde{K}^{R}))^{\perp})$ is dependent on the selected grid.
Lastly, as evident from \eqref{de::facto::dim}, the dimension of the identifiable component, $\sP_{R}^{*}(\bV_{\mF}(\tilde{K}^{R}))$, does not exceed the number of mesh points used in the reconstruction.


For subspaces $\bV_{1}, \dots, \bV_{N}$ in $\bH(K)$, we denote the sum of subspaces by $\bigplus_{i=1}^{N} \bV_{i}:=\{\sum_{i=1}^{N} v_{i} : v_{1} \in \bV_{1}, \dots, v_{N} \in \bV_{N} \}$. When $\bV_{1}, \dots, \bV_{N}$ are closed subspaces, then $( \bigcap_{i=1}^{N} \bV_{i})^{\perp}=\overline{\bV_{1}^{\perp}+\dots+\bV_{N}^{\perp}}$. Hence, for multiple orientations $\mR=\{R_{i} \in SO(n): \ i=1,2,\dots, N\}$, \eqref{3::decomp} becomes
\begin{align}\label{space::repre::decomp}
    \bH(K)=\left(\bigplus_{i=1}^{N} \sP^{*}_{R_{i}}(\bV_{\mF}(\tilde{K}^{R_{i}}))\right) \oplus \left( \bigcap_{i=1}^{N} \sP^{-1}_{R_{i}}((\bV_{\mF}(\tilde{K}^{R_{i}}))^{\perp}) \right).
\end{align}
Let $\bH^{\mR}_{\mF}:=\bigplus_{i=1}^{N} \sP^{*}_{R_{i}}(\bV_{\mF}(\tilde{K}^{R_{i}})) =\text{span} \{\sP^{*}_{R_{i}} \tilde{k}_{\vect{x}_{j}}^{R_{i}}: 1 \le i \le N, 1 \le j \le M \}$, a subspace of $\bH(K)$ with dimension at most $NM$, and define $\cP_{\mF}^{\mR}:\bH(K) \rightarrow \bH^{\mR}_{\mF}$ to be the projection operator onto $\bH^{\mR}_{\mF}$.  \cref{repre::proj} part (1) reveals  {that} the evaluation of the projected image solely depends on $\cP_{\mF}^{\mR} f$, i.e.
\begin{equation}\label{eval::proj}
    \sP_{R_{i}}f(\vect{x}_{j})=(\sP_{R_{i}} \circ \cP_{\mF}^{\mR}) f(\vect{x}_{j}), \quad 1 \le i \le N, \ 1 \le j \le M.
\end{equation}
Again, when we observe  projection images from angles $R_{1}, \dots, R_{N}$ with the mesh points  $\vect{x}_{1}, \dots, \vect{x}_{M}$, it is impossible to extract any information about the image component that belongs to $(\bH^{\mR}_{\mF})^{\perp}$.  The best we can do is to estimate
\begin{equation*}
    \cP_{\mF}^{\mR} f=\sum_{i=1}^{N} \sum_{j=1}^{M} \alpha_{ij} \sP^{*}_{R_{i}} \tilde{k}_{\vect{x}_{j}}^{R_{i}}.
\end{equation*}

To introduce the representer theorem, we establish some notation. We define the vectorization of $\alpha_{ij}$ and the weight matrix, respectively, as
\begin{align*}
    &\ma:=\mathrm{vec}[\alpha_{ij}] \in \bR^{NM}, \\ \nonumber
    &\mW:=\left[w_{ij, i'j'} = \langle \sP_{R_{i}}^{*} \tilde{k}^{R_{i}}_{\vect{x}_{j}}, \sP_{R_{i'}}^{*} \tilde{k}^{R_{i'}}_{\vect{x}_{j'}} \rangle_{\bH(K)} \right] \in \bR^{NM \times NM}. \nonumber
\end{align*}
The \emph{Gram matrix} $\mW$ is clearly positive semidefinite. Additionally, due to \eqref{adj::ind::gen}, we have an equivalent expression for $w_{ij, i'j'} \in \bR$:
\begin{align}\label{weight::equiv}
    w_{ij, i'j'} &=\langle \sP_{R_{i}}^{*} \tilde{k}^{R_{i}}_{\vect{x}_{j}}, \sP_{R_{i'}}^{*} \tilde{k}^{R_{i'}}_{\vect{x}_{j'}} \rangle_{\bH(K)} 
    =(\sP_{R_{i}} \circ \sP_{R_{i'}}^{*} \tilde{k}^{R_{i'}}_{\vect{x}_{j'}})(\vect{x}_{j}) \\ \nonumber
    &=\int_{I_{\|\vect{x}_{j'}\|}} \int_{I_{\|\vect{x}_{j}\|}} K(R_{i}^{\top} [\vect{x}_{j}:z_{1}], R_{i'}^{\top} [\vect{x}_{j'}:z_{2}]) \rd z_{1} \rd z_{2}.
\end{align}

\begin{proposition}\label{iso::full::rk}
The projection $\cP_{\mF}^{\mR} f=\sum_{i=1}^{N} \sum_{j=1}^{M} \alpha_{ij} \sP^{*}_{R_{i}} \tilde{k}_{\vect{x}_{j}}^{R_{i}}$ of $f \in \bH(K)$ onto $\bH^{\mR}_{\mF}$ is uniquely determined by the linear system:
\begin{equation*}
    \mathrm{vec}[\sP_{R_{i}}f(\vect{x}_{j})]=\mathrm{vec}[(\sP_{R_{i}} \circ \cP_{\mF}^{\mR}) f(\vect{x}_{j})]=\mW \ma \in \bR^{NM}.
\end{equation*}
Denote by $\cR(\mW)$ the range space of the Gram matrix $\mW$, equipped with the inner product $\langle \mW \alpha, \mW \beta \rangle_{\mW} = \alpha^{\top} \mW \beta$. Then, the map 
\begin{equation}\label{iso::proj}
    \cI^{\mR}_{\mF}: (\bH^{\mR}_{\mF}, \langle \cdot, \cdot \rangle_{\bH(K)}) \rightarrow (\cR(\mW), \langle \cdot, \cdot \rangle_{\mW}), \ \sum_{i=1}^{N} \sum_{j=1}^{M} \alpha_{ij} \sP^{*}_{R_{i}} \tilde{k}_{\vect{x}_{j}}^{R_{i}} \mapsto \mW \ma
\end{equation}
is an isomorphism 
Furthermore, if the kernel $K: \bB^{n} \times \bB^{n} \rightarrow \bR$ is strictly p.d. and $\mF= \{\vect{x}_{1}, \dots, \vect{x}_{M} \} \subset (\bB^{n-1})^{\circ}$, then $\mW \in \bR^{NM \times NM}$ is also strictly p.d., thus $\dim (\bH^{\mR}_{\mF})=\dim (\mW) = NM$.
\end{proposition}

\begin{remark}
With regard to implementation, we inscribe the pixelated projected images in the unit ball $\bB^{n-1}$ and register coordinates $\mF=\{\vect{x}_{j}: j=1,\dots, M\}$ to each pixel. If a pixel is matched to a point on the boundary, say $\vect{x}_{1} \in \partial \bB^{n-1}$, it is virtually the same as neglecting the corresponding pixel. The weights $w_{ij, i'j'}$ in \eqref{weight::equiv} will be zero whenever $j=1$ or $j'=1$, resulting in entirely zero columns and rows within the weight matrix $\mW$, which renders $\mW$ non-invertible, while its Moore-Penrose inverse is equivalent to that of $\mF-\{x_{1}\}$. For strictly p.d. kernels, the singularity of $\mW$ is not tied to inherent relationships between mesh points, but rather to whether $\mF \cap \partial \bB^{n-1}=\emptyset$ or not. In conclusion, it is prudent to avoid situating a mesh point on the boundary, as such a point jeopardizes computational stability without playing any role.
\end{remark}

\subsection{Kernel Reconstruction}\label{ssec:recons}
We now consider the reconstruction problem, adopting a (penalized) Maximum Likelihood Estimator (MLE) approach within the RKHS framework. This reduces reconstruction to a linear regression problem. Recall that our observations are
\begin{align*}
    y_{ij} =\sP_{R_{i}} f^{0}(\vect{x}_{j})+ \varepsilon_{ij}, \quad i=1,2,\dots, N, \, j=1,2,\dots,M,
\end{align*}
where $\varepsilon_{ij}=W_{i}(\vect{x}_{j})$ are i.i.d. $N(0, \sigma^{2})$. Hence, an MLE satisfies
\begin{align*}
    \hat{f}^{0} \ \in& \ \argmax_{f \in \bH(K)} \prod_{i=1}^{N} \prod_{j=1}^{M} \frac{1}{\sqrt{2 \pi} \sigma} \exp \left(-\frac{(y_{ij}-\sP_{R_{i}}f^{0}(\vect{x}_{j}))^{2}}{2 \sigma^{2}} \right) = \argmin_{f \in \bH(K)} \hat{L}^{0}(f),
\end{align*}
where the empirical risk functional 
\begin{equation*}
    \hat{L}^{0}(f):=\sum_{i=1}^{N} \sum_{j=1}^{M} \left(y_{ij}-\sP_{R_{i}}f(\vect{x}_{j}) \right)^{2},
\end{equation*}
is associated with the risk functional 
\begin{equation*}
    L^{0}(f):=\sum_{i=1}^{N} \sum_{j=1}^{M} \bE [ \left(y_{ij}-\sP_{R_{i}}f(\vect{x}_{j}) \right)^{2} ].
\end{equation*}
Denote by $\mY:=\mathrm{vec}[y_{ij}] \in \bR^{NM}, \ \ma^{0}:=\mathrm{vec}[\alpha^{0}_{ij}] \in \bR^{NM}$ the vectorizations of $y_{ij}, \alpha^{0}_{ij}$,
where $\cP_{\mF}^{\mR} f^{0}=\sum_{i=1}^{N} \sum_{j=1}^{M} \alpha^{0}_{ij} \sP^{*}_{R_{i}} \tilde{k}^{R_{i}}_{\vect{x}_{j}}$.
Using the vectorization, the risk functionals become 
\begin{align*}
    &\hat{L}^{0}(f)=\hat{L}^{0}(\cP_{\mF}^{\mR} f)=\hat{L}^{0}(\ma)=\|\mY-\mW \ma\|^{2}, \\
    &L^{0}(f)=L^{0}(\cP_{\mF}^{\mR} f)=L^{0}(\ma)=\|\mW \ma^{0}-\mW \ma\|^{2}+NM \sigma^{2}.
\end{align*}
Therefore, we recover the following normal equation:
\begin{align}\label{normal::eq::lse}
    \mW^{2} \hat{\ma}^{0}= \mW \mY \quad &\Longleftrightarrow \quad \mW \hat{\ma}^{0}=\mW \mW^{\dagger} \mY \\ \nonumber
    \quad &\Longleftrightarrow \quad \cP_{\mF}^{\mR} \hat{f}^{0}=(\cI^{\mR}_{\mF})^{-1}(\mW \hat{\ma}^{0})=(\cI^{\mR}_{\mF})^{-1}(\mW \mW^{\dagger} \mY).
\end{align}

It is worthwhile to highlight the logical pathway through which we can derive the classical least square estimation for the reconstruction. As indicated in \eqref{eval::proj}, given $R_{1}, \dots, R_{N}$ with the mesh points $\vect{x}_{1}, \dots, \vect{x}_{M}$, we cannot identify components that belong to $(\bH^{\mR}_{\mF})^{\perp}$. As a result, instead of $\bH(K)$, the \emph{de facto} space we are working on is 
$(\bH^{\mR}_{\mF}, \langle \cdot, \cdot \rangle_{\bH(K)}) \cong (\cR(\mW), \langle \cdot, \cdot \rangle_{\mW})$, as outlined in \eqref{iso::proj}.

\begin{theorem}[X-ray Linear Regression]\label{Xray::MLE}
Let $\mR =\{R_{1}, \dots, R_{N} \} \subset SO(n)$ be viewing angles and $\mF = \{\vect{x}_{1}, \dots, \vect{x}_{M} \} \subset \bB^{n-1}$ be an evaluating grid. Also, let $\cI^{\mR}_{\mF}: (\bH^{\mR}_{\mF}, \langle \cdot, \cdot \rangle_{\bH(K)}) \rightarrow (\cR(\mW), \langle \cdot, \cdot \rangle_{\mW})$ be the isomorphism defined as in \eqref{iso::proj}.
\begin{enumerate}
    \item Let $\hat{\ma}^{0}:=\mW^{\dagger} \mY \in \cR(\mW)$ and $\hat{f}^{0}:= \sum_{i=1}^{N} \sum_{j=1}^{M} \hat{\alpha}^{0}_{ij} \sP^{*}_{R_{i}} \tilde{k}^{R_{i}}_{\vect{x}_{j}} \in \bH^{\mR}_{\mF}$. Then, $\argmin_{f \in \bH(K)} \hat{L}^{0}(f) =\hat{f}^{0} \oplus (\bH^{\mR}_{\mF})^{\perp}$, i.e. $\hat{f}^{0}$ is the minimizer of $\hat{L}^{0}$ with minimum norm. 
    \item $\argmin_{f \in \bH(K)} L^{0}(f)= \cP_{\mF}^{\mR} f^{0} \oplus (\bH^{\mR}_{\mF})^{\perp}$, i.e. $\cP_{\mF}^{\mR} f^{0}$ is the minimizer of $L^{0}$ with minimum norm. 
    \item $\hat{\ma}^{0} \sim N(\mW \mW^{\dagger} \ma^{0}, \sigma^{2} (\mW^{\dagger})^{2})$, $\hat{f}^{0} \sim N(\cP_{\mF}^{\mR} f^{0}, \sigma^{2} (\cI^{\mR}_{\mF})^{-1} \circ  \mW^{\dagger} \circ (\cI^{\mR}_{\mF}))$, where $\mW^{\dagger}$ is the Moore-Penrose inverse of $\mW$.
\end{enumerate}
\end{theorem}

Our regression setting differs from the traditional $\cL_{2}$ setting. Methods based on the SVD of the compact operator $\sP: \cL_{2}(\bB^{n}) \rightarrow \cL_{2}(SO(n) \times \bB^{n-1}, W^{-1})$ typically fix the choice of basis as the singular vectors, taking the form of $f_{mlk}(r \theta)=Z_{ml}(r) Y_{lk}(\theta)$, where the radial basis $Z_{ml}$ consists of Zernike polynomials, and $Y_{lk}$ represents spherical harmonics \cite{natterer2001mathematics, maass1987x}. However, in our framework:
\begin{enumerate}
    \item The MLE $\hat{f}^{0}$ can vary based on the selection of the kernel $K: \bB^{n} \times \bB^{n} \rightarrow \bR$, granting us more flexibility in selecting a basis.
    \item The rank of the design matrix $\mW$ is not determined by the truncation level, but by the total number of observations $NM$.
\end{enumerate}
Arguably, our reconstruction method represents a more general framework than the truncated SVD, as the latter corresponds to a specific choice of non-strictly p.d. kernels. For further details, we refer to the discussion following \cref{left::ang::ex}.

The design matrix $\mW$ is always square by construction, so there is no \emph{parsimony} at play here. In \cref{iso::full::rk}, we demonstrate that when using a strictly positive definite kernel with interior mesh points $\mF= \{\vect{x}_{1}, \dots, \vect{x}_{M} \} \subset (\bB^{n-1})^{\circ}$, $\mW$ becomes full rank, i.e., $\cR(\mW)=\bR^{NM}$. Consequently, the MLE $\hat{f}^{0}$ generates a sinogram that perfectly interpolates our observations $\mY=\mathrm{vec}[y_{ij}] \in \bR^{NM}$ regardless of the presence of the noise, i.e. $\min_{f \in \bH(K)} \hat{L}^{0}(f) = \min_{\ma \in \bR^{MN}} \hat{L}^{0}(\ma)= 0$. 
To address this issue of overfitting, we consider Tikhonov regularization, leading to a \emph{penalized} MLE. For a tuning parameter $\nu >0$, we introduce a penalty term to the empirical risk functional $\hat{L}^{\nu}:\bH(K) \rightarrow \bR$ as follows: 
\begin{align*}
    \hat{L}^{\nu}(f)
    =\sum_{i=1}^{N} \sum_{j=1}^{M} &\left(y_{ij}-\sP_{R_{i}} f(\vect{x}_{j}) \right)^{2} +\nu \|f\|^{2}_{\bH(K)}.
\end{align*}
 {Due to} \eqref{eval::proj} , we have $\hat{L}^{\nu}(f) \ge \hat{L}^{\nu}(\cP_{\mF}^{\mR} f)$, hence any minimizer $\hat{f}^{\nu}$ of  {the penalized risk} $\hat{L}^{\nu}$ belongs to $\bH^{\mR}_{\mF}$, i.e. there are some $\hat{\alpha}^{\nu}_{ij}$ such that $\hat{f}^{\nu}=\cP_{\mF}^{\mR} \hat{f}^{\nu}=\sum_{i=1}^{N} \sum_{j=1}^{M} \hat{\alpha}^{\nu}_{ij} \sP^{*}_{R_{i}} \tilde{k}^{R_{i}}_{\vect{x}_{j}}$.
By \cref{iso::full::rk}, the empirical risk functional becomes
\begin{equation*}
    \hat{L}^{\nu}(f) \ge \hat{L}^{\nu}(\cP_{\mF}^{\mR} f)=\|\mY-\mW \ma\|^{2}+\nu \ma^{\top} \mW \ma,
\end{equation*}
which yields the normal equation as follows:
\begin{align}\label{normal::eq::tikho}
    (\mW^{2}+\nu \mW) \hat{\ma}^{\nu}= \mW \mY \,&\Longleftrightarrow \, \mW \hat{\ma}^{\nu}=\mW (\mW+\nu \mI)^{-1} \mY. 
\end{align}
In summary, we deduce the representer theorem below for the regularized reconstruction. We also highlight that, whenever the penalty function is a strictly monotone increasing function of $\|f\|_{\bH(K)}$, the same argument remains valid: the minimizer belongs to $\bH_{\mF}^{\mR}$. Additionally, if the penalty function is also convex, then the minimizer is unique.

\begin{theorem}[Tikhonov Kernel Reconstruction]\label{tikhonov} 
Let $\mR =\{R_{1}, \dots, R_{N} \} \subset SO(n)$ be viewing angles, $\mF = \{\vect{x}_{1}, \dots, \vect{x}_{M} \} \subset \bB^{n-1}$ be an evaluating grid, and $\nu > 0$ be a regularization parameter. Also, let $\cI^{\mR}_{\mF}: (\bH^{\mR}_{\mF}, \langle \cdot, \cdot \rangle_{\bH(K)}) \rightarrow (\cR(\mW), \langle \cdot, \cdot \rangle_{\mW})$ be the isomorphism defined as in \eqref{iso::proj}.
\begin{enumerate}
    \item The unique minimizer of $\hat{L}^{\nu}(f)$ is given by
    $\hat{f}^{\nu}=\sum_{i=1}^{N} \sum_{j=1}^{M} \hat{\alpha}^{\nu}_{ij} \sP^{*}_{R_{i}} \tilde{k}^{R_{i}}_{\vect{x}_{j}} \in \bH^{\mR}_{\mF}$, where $\hat{\ma}^{\nu}:=(\mW+\nu \mI)^{-1} \mY$.
    \item The minimum of the empirical risk is given by
    \begin{equation*}
        \min_{f \in \bH(K)} \hat{L}^{\nu}(f)=\min_{\ma \in \bR^{MN}} \hat{L}^{\nu}(\ma)=\hat{L}^{\nu}(\hat{\ma}^{\nu})= \nu \mY^{\top} (\mW + \nu \mI)^{-1} \mY.
    \end{equation*}
    \item $\hat{\ma}^{\nu} \sim N(\mW (\mW+\nu \mI)^{-1} \ma^{0}, \sigma^{2} (\mW+\nu \mI)^{-2})$ and $\hat{f}^{\nu}$ is the finite-dimensional degenerate Gaussian process:
    \begin{equation*}
    \hat{f}^{\nu} \sim N((\cI^{\mR}_{\mF})^{-1}(\mW^{2} (\mW+\nu \mI)^{-1} \ma^{0}), \sigma^{2} (\cI^{\mR}_{\mF})^{-1} \circ (\mW (\mW+\nu \mI)^{-2}) \circ (\cI^{\mR}_{\mF})).    
    \end{equation*}
    \item As $\nu \downarrow 0$, $\hat{\ma}^{\nu}$ and $\hat{f}^{\nu}$ converge to the MLEs $\hat{\ma}^{0}$ and $\hat{f}^{0}$ in $\bR^{NM}$ and $\bH(K)$, respectively: $\lim_{\nu \downarrow 0} \hat{\ma}^{\nu}=\hat{\ma}^{0}, \ \lim_{\nu \downarrow 0} \hat{f}^{\nu}=\hat{f}^{0}$.
\end{enumerate}
\end{theorem}

\begin{remark}
Our Tikhonov regularizer can be interpreted as a Maximum a posteriori (MAP) estimator with a Gaussian prior. If we equip $f$ with the Gaussian white noise prior of level $\sigma^{2}/\nu > 0$, following Theorem 2.12 of \cite{da2014stochastic}, i.e., for any $h \in \bH(K)$, $\langle h, f \rangle_{\bH(K)} \sim N(0, \sigma^{2}/\nu \cdot \|h\|_{\bH(K)}^{2})$, then the prior of $\mathrm{vec}[\sP_{R_{i}}f(\vect{x}_{j})]=\mW \ma \in \bR^{NM}$ becomes $N(0, \sigma^{2}/\nu \mW)$. Consequently, its prior density with respect to the Lebesgue measure on $\cR(\mW) \subset \bR^{NM}$ is given by
\begin{equation*}
    p( \mW \ma) = \frac{1}{\sqrt{\det (2\pi \mW \vert_{\cR(\mW)})}} \exp \left( - \frac{\nu}{2 \sigma^{2}} (\mW \ma)^{\top} \mW^{\dagger} (\mW \ma) \right) \propto \exp \left( - \frac{\nu}{2 \sigma^{2}} \ma^{\top} \mW \ma \right)
\end{equation*}
Thus, the posterior density is proportional to
\begin{align*}
    p( \mW \ma \vert \mY) \propto p( \mY | \mW \ma) \cdot p( \mW \ma) 
    \propto \exp \left( - \frac{1}{2 \sigma^{2}} \|\mY-\mW \ma\|^{2} - \frac{\nu}{2 \sigma^{2}} \ma^{\top} \mW \ma \right),
\end{align*}
and the MAP estimator of $\ma$ is given by
\begin{equation*}
    \hat{\ma}^{\mathrm{MAP}} = \argmin_{\ma \in \bR^{NM}} \|\mY-\mW \ma\|^{2} + \nu \ma^{\top} \mW \ma = \hat{\ma}^{\nu}.
\end{equation*}
\end{remark}

Our X-ray representer theorem produces a reconstructed image at a functional level within the discrete setup. This reconstruction bypasses the  projection image interpolation step, as it automatically satisfies the consistency condition in \cref{HLCC}:
\begin{proposition}\label{RKHS:HLCC}
The Tikhonov regularized reconstruction in \cref{tikhonov} produces an interpolation of the  projection images as follows: for any $R \in SO(n)$ and $\vect{x} \in \bB^{n-1}$,
\begin{align*}
    \sP \hat{f}^{\nu}(R, \vect{x}) = \langle \sP_{R}^{*} \tilde{k}^{R}_{\vect{x}}, \hat{f}^{\nu} \rangle_{\bH(K)}
    =\sum_{i=1}^{N} \sum_{j=1}^{M} \hat{\alpha}^{\nu}_{ij} \langle \sP_{R}^{*} \tilde{k}^{R}_{\vect{x}}, \sP_{R_{i}}^{*} \tilde{k}^{R_{i}}_{\vect{x}_{j}} \rangle_{\bH(K)}. 
\end{align*}
Furthermore, the interpolant $\sP \hat{f}^{\nu} \in \bH(\tilde{K}^{R})$ is a bona fide sinogram, i.e. it satisfies the HLCC in \cref{HLCC}: for any $l \in \bZ_{+}$, we have
\begin{equation*}
    \int_{\bB^{n-1}} (\vect{x}^{\top} \vect{y})^{l} \sP \hat{f}^{\nu}(R,\vect{x}) \rd \vect{x} = q_{l}^{\nu}(\mP_{R}^{*}(\vect{y})), \quad \vect{y} \in \bR^{n-1},
\end{equation*}
where $q_{l}^{\nu} (\cdot) :=\int_{\bB^{n}} (\vect{z} ^{\top} \cdot)^{l}  \hat{f}^{\nu}(\vect{z}) \rd \vect{z}$ is a homogeneous polynomial of degree $l$, independent of $R \in SO(n)$. 
\end{proposition}

\subsection{Stability}\label{ssec:stab}
Reconstruction stability is influenced by two primary sources of error. The first source is obviously measurement errors, and several stability inequalities have been established in terms of $\|\sP(f-f^{0})\|$ and $\|f^{0}\|$ under some regularity conditions \cite{rullgaard2004stability, natterer2001mathematics}. These results prove valuable only if we have access to an entire perturbed sinogram. However, our situation involves only a finite number of discretized  projection images, which in effect constitutes a second error source. In this discrete setup, previous studies \cite{maass1987x, natterer2001mathematics} have addressed the issue of \emph{resolution}, i.e. characterizing the degree of oscillation in reconstructed images. Nonetheless, to the best of our knowledge, there exists no \emph{non-asymptotic} stability inequality that specifically addresses the discretized measurement error. To bridge this gap, we establish a \emph{sharp} stability result for our algorithm, in the realistic  setting of discretized measurements corrupted by errors, in two setups: random errors of given variance, \cref{MSE:Tikho}, and deterministic errors with given bounded norm, \cref{stable::recons}.

\begin{proposition}\label{MSE:Tikho}
The mean squared error of the Tikhonov regularizer with tuning parameter $\nu > 0$ in \cref{tikhonov} is given by
\begin{align*}
    \text{MSE} &(\hat{f}^{\nu}) := \bE[ \|\hat{f}^{\nu}-f^{0}\|^{2}_{\bH(K)}] \\
    =& \nu^{2} (\ma^{0})^{\top} \mW (\mW + \nu \mI)^{-2} \ma^{0} + \sigma^{2} \text{tr} (\mW (\mW + \nu \mI)^{-2}) + \|f^{0}-\cP_{\mF}^{\mR} f^{0}\|^{2}_{\bH(K)}.    
\end{align*}
\end{proposition}


\begin{theorem}\label{stable::recons}
Let be $\nu>0$ a tuning parameter, and $d > 0$ be the smallest non-zero eigenvalue of the Gram matrix $\mW \in \bR^{NM \times NM}$.
Then, the worst-case squared error of the Tikhonov regularizer in \cref{tikhonov} is given as follows:
\begin{equation*}
    \text{err}(\rho, \varepsilon) := \max \left\{\|\hat{f}^{\nu}- f^{0}\|^{2}_{\bH(K)} : \|f^{0}\|_{\bH(K)} \le \rho,\, \sum_{i=1}^{N} \sum_{j=1}^{M} \varepsilon_{ij}^{2} \le \varepsilon^{2} \right\} 
    =\rho^{2} + \frac{\varepsilon^{2}}{d+ 2 \nu}.
\end{equation*}
\end{theorem}

\begin{remark}
The proof of \cref{stable::recons} exemplifies the technical advantage of the RKHS framework. In the context of discrete setup, the RKHS framework has the ability to explicitly define the identifiable space $\bH^{\mR}_{\mF}$, which is a finite-dimensional space. We can then apply a technique similar to demonstrating the shrinkage effect in ridge regression, allowing us to derive the sharp inequality.
\end{remark}

The above theorem reveals the stabilizing effect of introducing a penalty term. Whereas $1/(d+2 \nu) \approx 1/(2 \nu)$ when $\nu \gg 1$, $1/(d+2 \nu) \approx 1/d$ when $\nu \ll 1$. Hence, in cases where the dimension of the Gram matrix increases substantially, leading to a situation where the smallest non-zero eigenvalue $d > 0$ approaches zero, the MLE in \cref{Xray::MLE} may exhibit instability, see \cref{RKHS::recons::no::pen::img}  {in the appendix}. We further remark that the above result also immediately yields pointwise stability since $|\hat{f}^{\nu} (\vect{z})- f^{0} (\vect{z})|^{2} \le K(\vect{z}, \vect{z}) \|\hat{f}^{\nu}- f^{0}\|^{2}_{\bH(K)}$ by the Cauchy-Schwarz inequality.

\subsection{Illustrative Examples}\label{sec::simul}
We provide several illustrative examples in this section, with two additional scenarios covered in \cref{sec:add:exper}. These scenarios highlight how our Kernel Reconstruction (KR) significantly outperforms conventional Filtered Backprojection (FBP). In our experiments, we take a 10-times intensified $M \times M$ image of the 2D Shepp-Logan phantom as the original structure $f^{0}$. For FBP reconstruction, we apply a ramp filter during the filtering process. 

\medskip
\noindent \textbf{First Setup}

In our initial setup, with fixed $M = 200$ and $N = 40$, we investigate six different scenarios for generating a sinogram. These scenarios involve three noise levels $\sigma \in \{0, 20, 100 \}$ and two types of angle grids: the equiangular grid $\phi_{j} = \pi (j-1)/N , \, j = 1, \dots, N$, and the random grid $\phi_{j}, \dots, \phi_{N}$ drawn from $\mathrm{Unif}[0, \pi]$. With each angle grid, we generate a $M \times N$ pixelated sinogram, perturbed with white noise of level $\sigma \ge 0$. Given a sinogram, we then run the FBP and KR algorithm, exploring several choices for the parameter $\gamma > 0$ for Gaussian kernels and the penalty parameter $\nu > 0$ for the normal equation of Tikhonov regularization. A higher value of $\gamma > 0$ accounts for a sharper Gaussian kernel.

\begin{figure}[h]
\centering
    \begin{subfigure}{.45\linewidth}
        \centering
        \includegraphics[width=\linewidth, height=6.5cm]
        {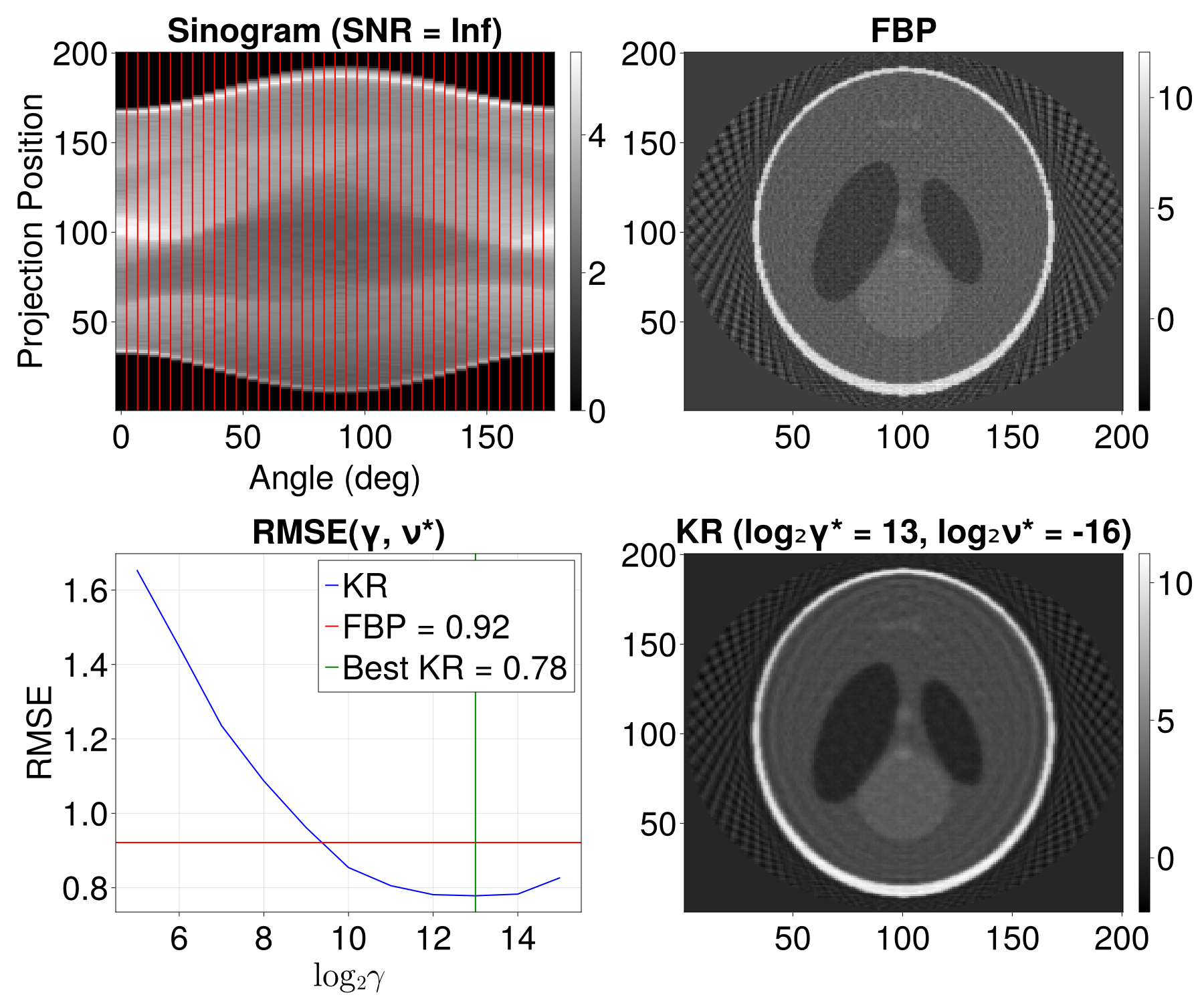}
    \end{subfigure}
    \hspace{0.5cm}
    \begin{subfigure}{.45\linewidth}
        \centering
        \includegraphics[width=\linewidth, height=6.5cm]{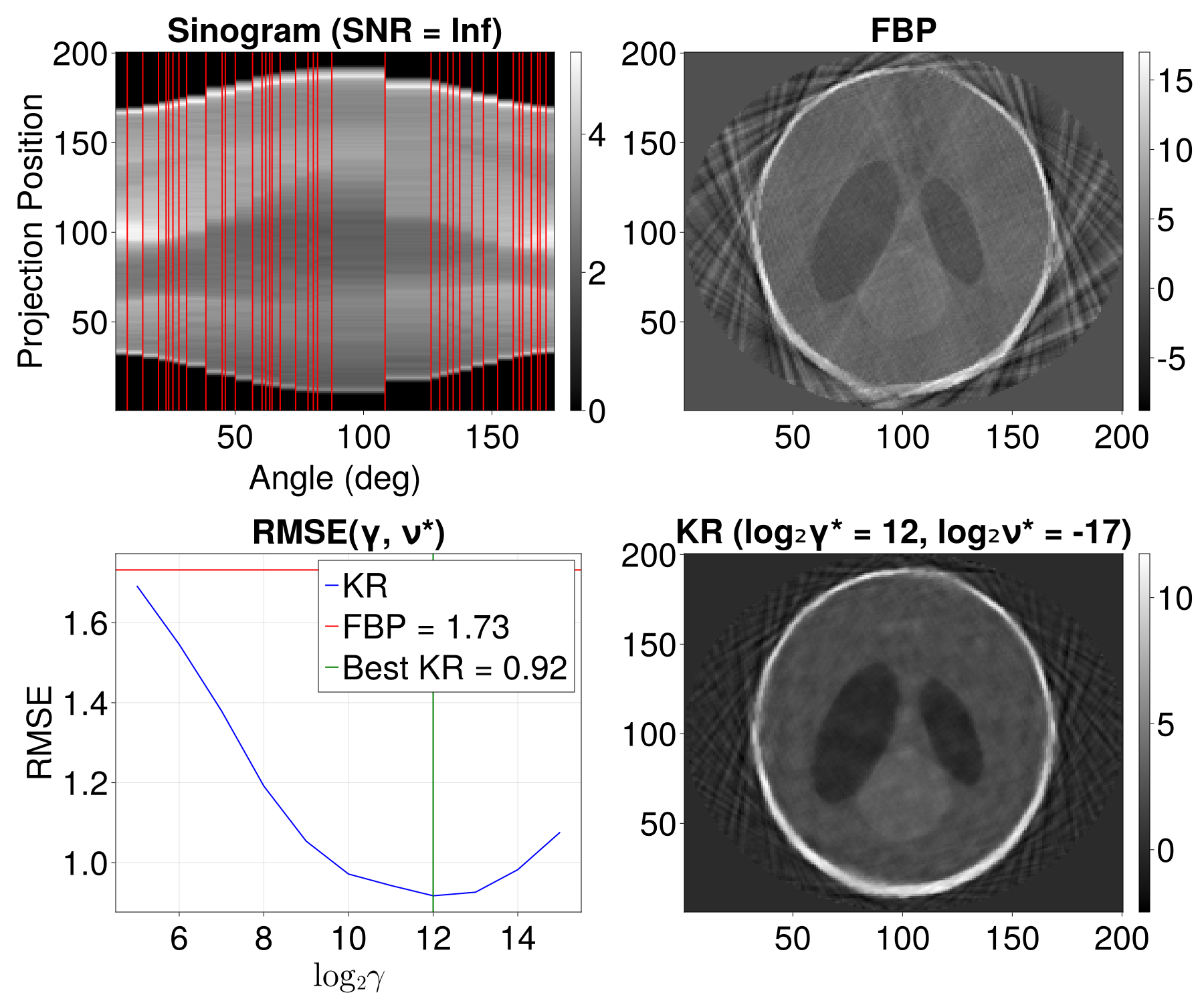}
    \end{subfigure}
\caption{Left panel: Equiangular grid. Right panel: Random grid (red vertical lines in the sinogram  indicate the location of the angular grid nodes). In each panel, FBP (top right) and KR (bottom right) reconstructions from a noiseless sinogram (top left). The bottom left plot shows the minimum RMSE achieved for our method (blue curve) for each $\gamma > 0$, the minimum RMSE across all $\gamma, \nu >0$ (green vertical line), and the RMSE for FBP (red horizontal line).}
\label{RKHS::recons::0}
\end{figure}

When using the cross-validated choices of $\gamma$ and $\nu$, the RKHS method consistently outperforms FBP across all scenarios. For the noiseless sinogram with parallel geometry, depicted in the left panel of \cref{RKHS::recons::0}, we note that the FBP algorithm might outperform our method when $\gamma$ is excessively small, as indicated by the Root Mean Squared Error (RMSE). However, as $\gamma$ increases, the RMSE rapidly diminishes, and when we reach the cross-validation values for $\gamma$ and $\nu$, the RMSE of our method becomes smaller than that of FBP.

\begin{figure}[h]
\centering
    \begin{subfigure}{.45\linewidth}
        \centering
        \includegraphics[width=\linewidth, height=6.5cm]{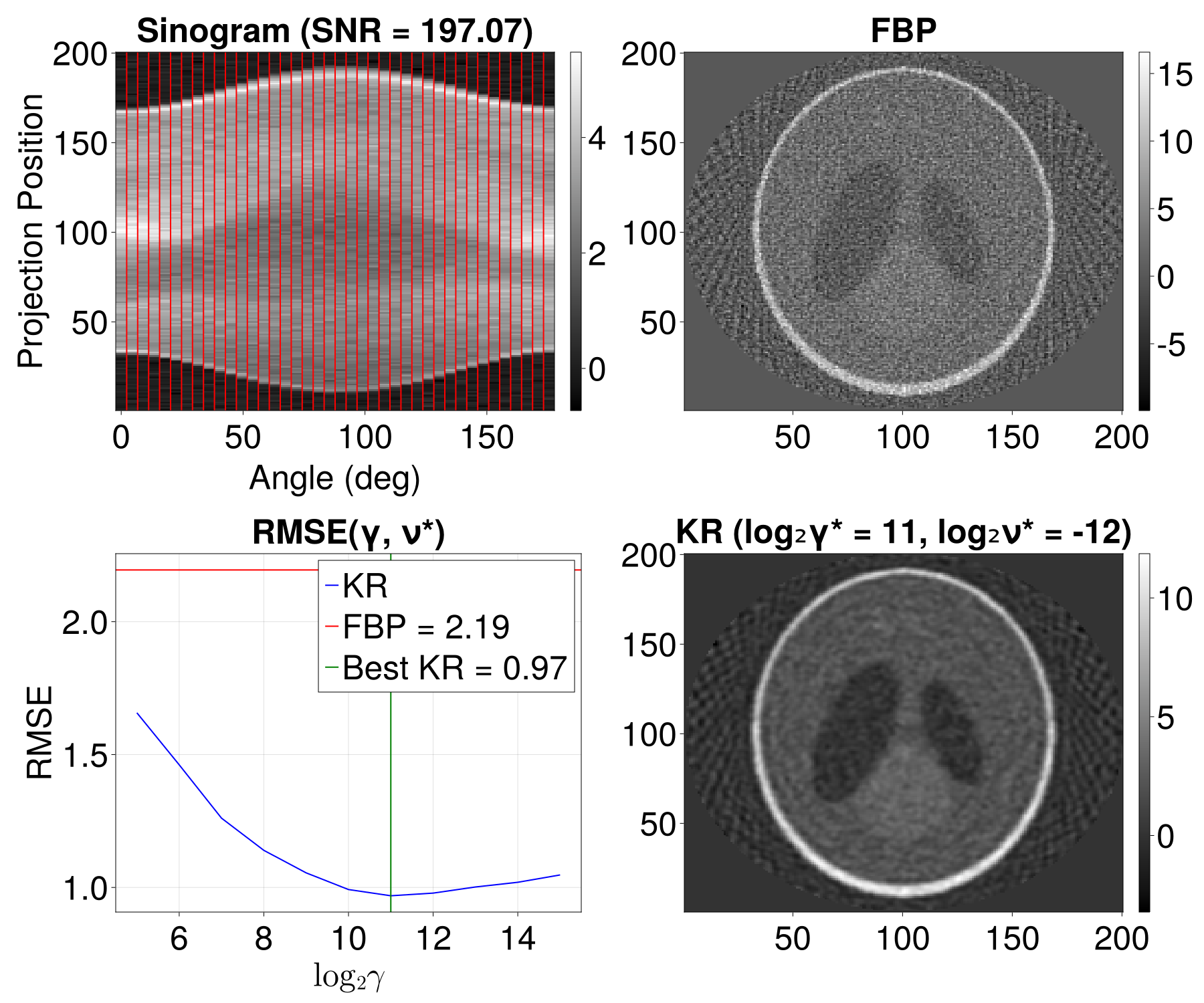}
    \end{subfigure}
    \hspace{0.5cm}
    \begin{subfigure}{.45\linewidth}
        \centering
        \includegraphics[width=\linewidth, height=6.5cm]{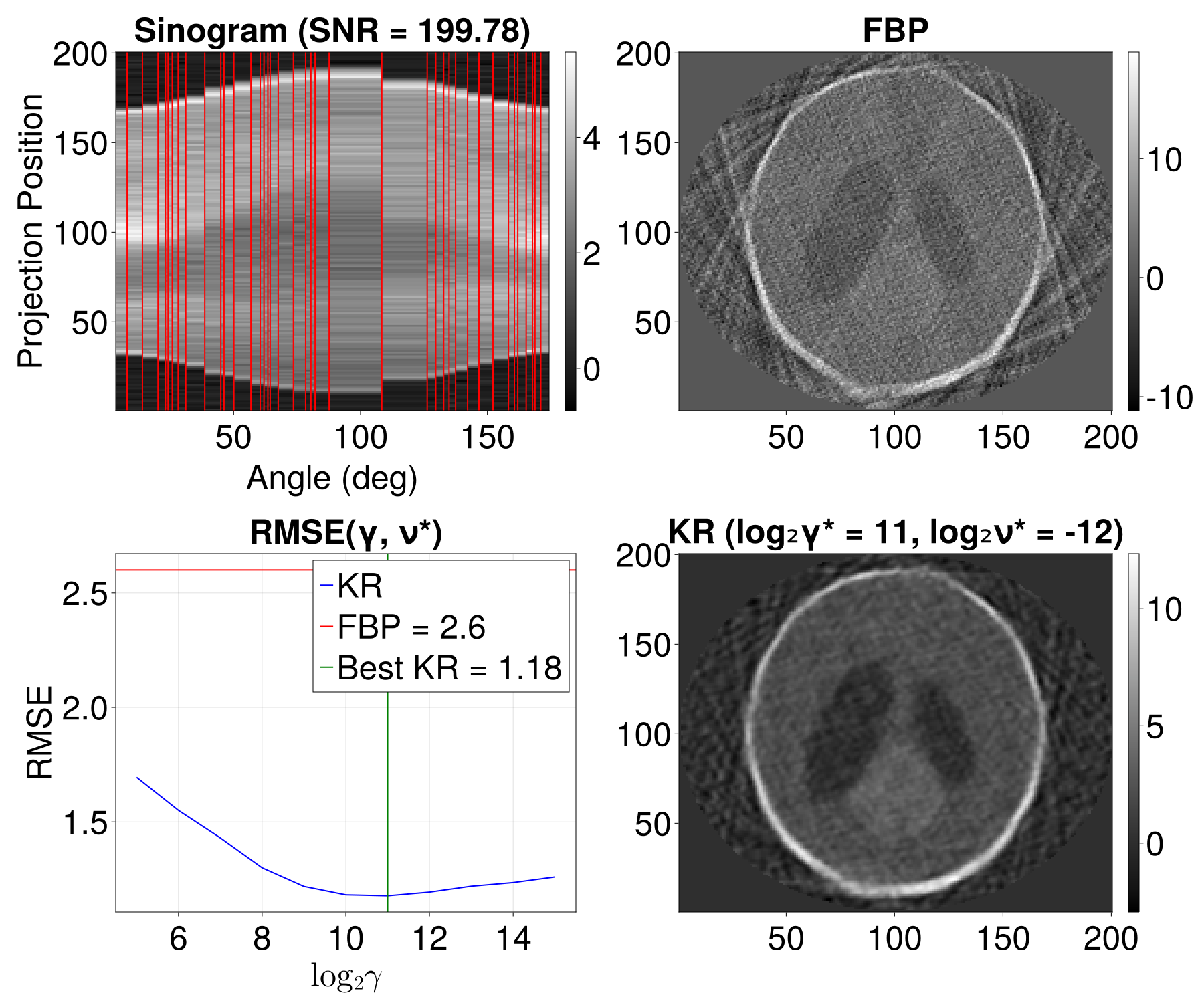}
    \end{subfigure}
    
    \begin{subfigure}{.45\linewidth}
        \centering
        \includegraphics[width=\linewidth, height=6.5cm]{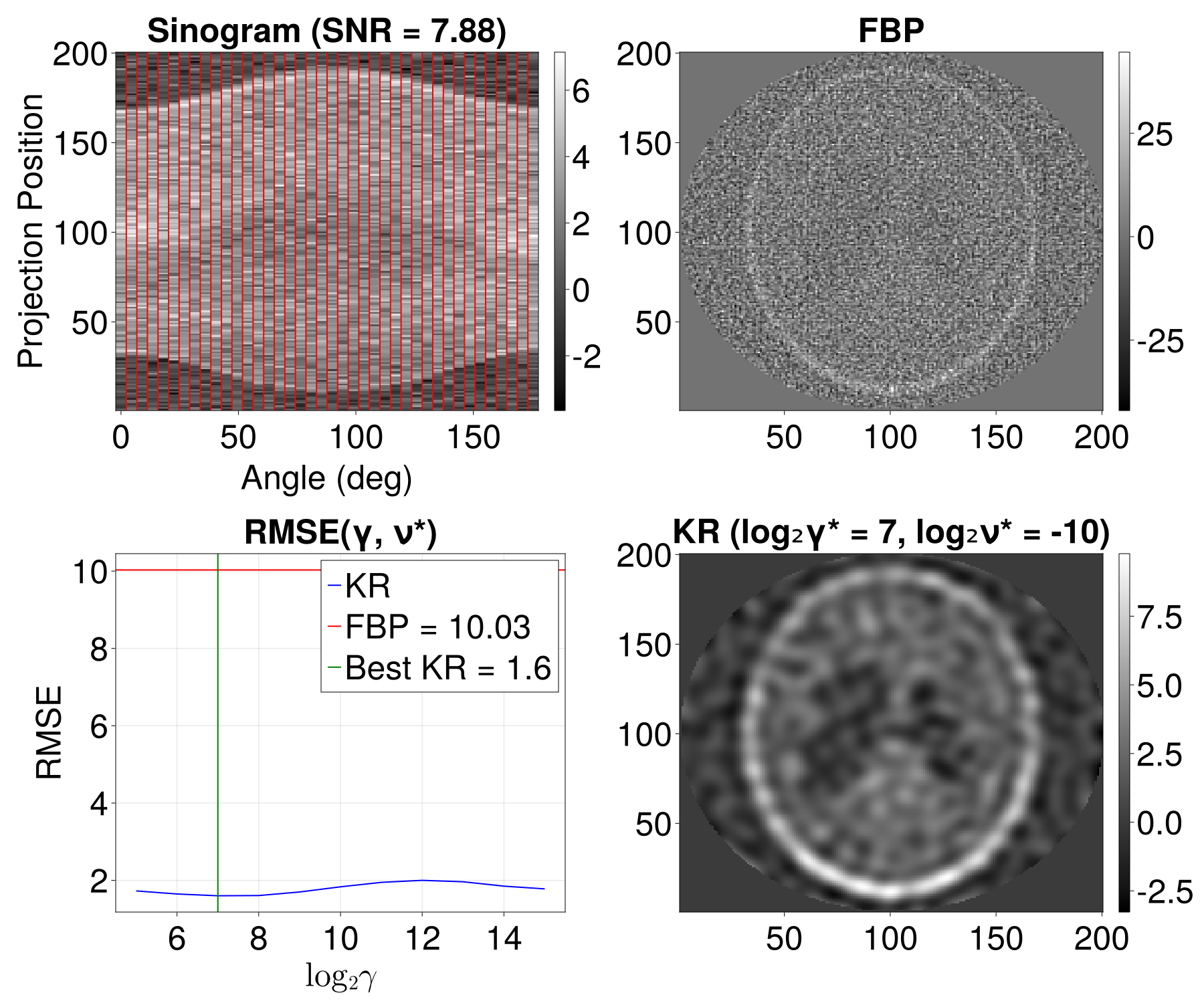}
    \end{subfigure}
    \hspace{0.5cm}
    \begin{subfigure}{.45\linewidth}
        \centering
        \includegraphics[width=\linewidth, height=6.5cm]{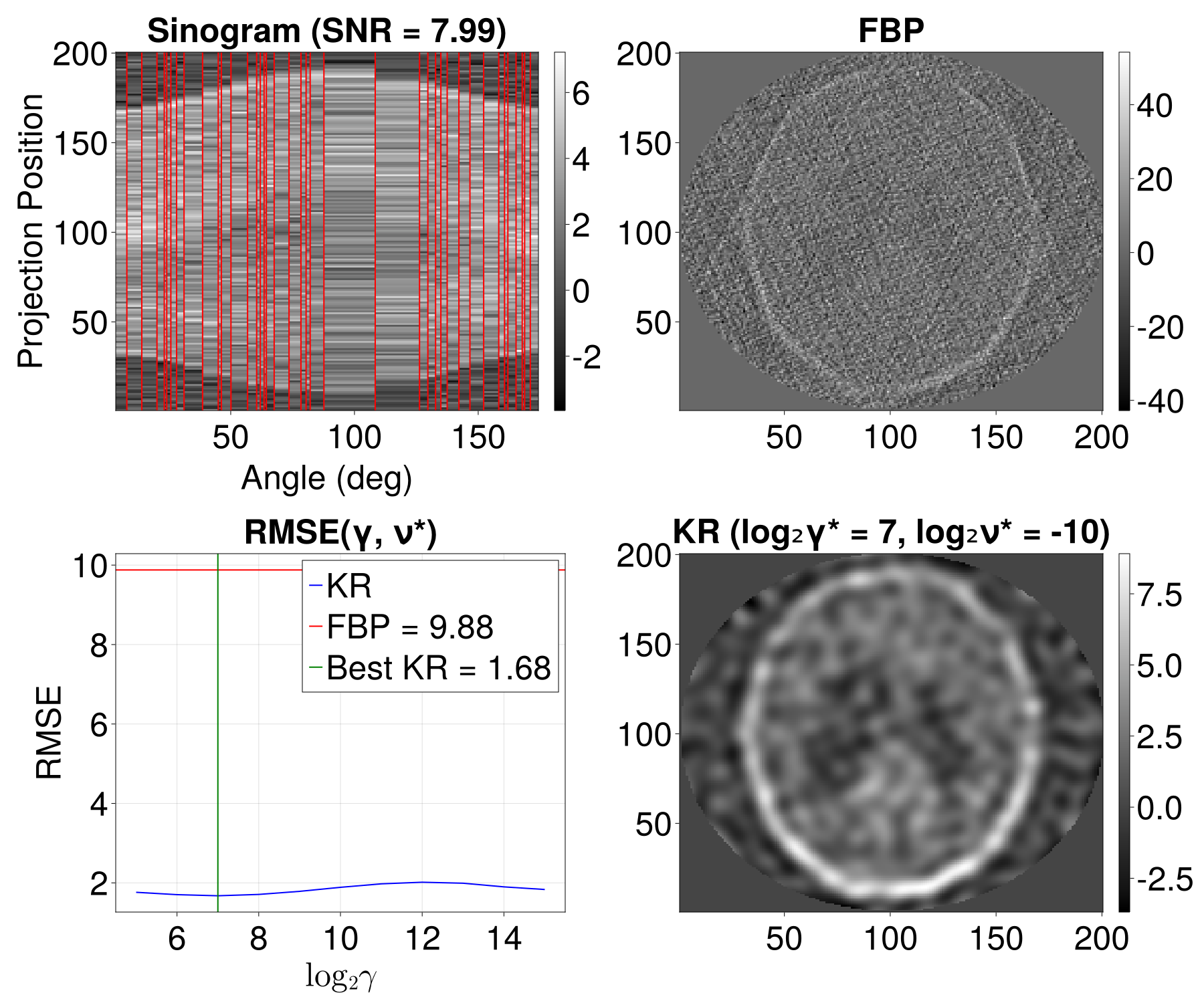}
    \end{subfigure}
    \caption{Left panels: Equiangular grid, Right panel: Random grid (Top: $\sigma = 20$, Bottom: $\sigma = 100$). In each panel, FBP (top right) and KR (bottom right) reconstructions from a perturbed sinogram (top left).}
\label{RKHS::recons::20}
\end{figure}

In the remaining five scenarios with a higher level of noise and increasing angular irregularity, our method tends to yield a smaller RMSE compared to FBP without the need for optimally tuned parameters. Additionally, as the kernel backprojection $\sP_{R}^{*} ( \tilde{k}^{R}_{\vect{x}} )(\vect{z}) = \sP_{R} k_{\vect{z}} (\vect{x})$ rapidly fades away as $\vect{z} \in \bB^{2}$ approaches the boundary of the unit ball, the KR algorithm yields a more regularized image with fewer pronounced artifacts, as evidenced in \cref{RKHS::recons::0,RKHS::recons::20}. Moreover, in scenarios where viewing angles are irregular, as in the right panels of \cref{RKHS::recons::0,RKHS::recons::20}, the FBP algorithm suffers from aliasing degradation caused by discretization of the projection images \cite{shi2015novel}. Even under challenging conditions with very low signal-to-noise ratio (SNR) ($\sigma = 100$) as in \cref{RKHS::recons::20}, where the FBP algorithm fails to capture the inner ellipsoids in the phantom and exhibits significant scale discrepancy, the KR reconstruction continues to provide some meaningful information.

As RMSE serves as a global measure of error across the unit ball, we complement our analysis by visualizing the reconstruction error $\hat{f} - f^{0}$ in \cref{RKHS::error::0_20} for the scenario where the noise level is $\sigma = 20$ with the random grid. In \cref{RKHS::error::0_20}, the KR reconstruction demonstrates greater scale coherence with the true phantom image, while the FBP algorithm produces highly oscillating reconstruction errors across the entire image due to inherent Fourier transform \cite{natterer2001mathematics}. As the sinogram is increasingly perturbed by noise, strictly adhering to it would significantly inflate the norm of the reconstructed image $\|\hat{f}\|_{\bH(K)}$. Consequently, the roughness of the function estimate becomes more important than fidelity to the data, necessitating a balance between these aspects in adversarial situations. This observation underscores the increased significance of the penalty term in Tikhonov regularization, and employing overly sharp kernels with high $\gamma > 0$ to enhance resolution may not necessarily lead to superior results. Indeed, the trends observed in \cref{RKHS::recons::0,RKHS::recons::20} further substantiate this, where the optimal $\gamma^{*}$ tends to decrease, while the optimal $\nu^{*}$ tends to increase as the sinogram becomes noisier.

\medskip
\noindent \textbf{Second Setup}

For the second setup, we conduct Monte Carlo simulations in various configurations to explore the disparities between the two algorithms concerning perturbation and the regularity of the angle grid. Here, we introduce the angle regularity parameter $\lambda \in [0, 1]$ to distribute $N$ orientations across $[0, \pi]$:
\begin{equation*}
    \phi_{j} = (1-\lambda) \cdot \pi U_{(j)}+ \lambda \cdot \frac{\pi (j-1)}{N}, \quad j = 1, \dots, N,
\end{equation*}
where $U_{(j)}$ represents the $j$-th ordered statistic of the sample $U_{1}, \dots, U_{N} \stackrel{iid}{\sim} \mathrm{Unif}[0, 1]$. Higher $\lambda \in [0, 1]$ values indicate more regular viewing, with $\lambda = 1$ representing parallel geometry and $\lambda = 0$ corresponding to a uniformly random angle grid.

\begin{figure}[h]
\centering
\includegraphics[width=\textwidth, height=8cm]{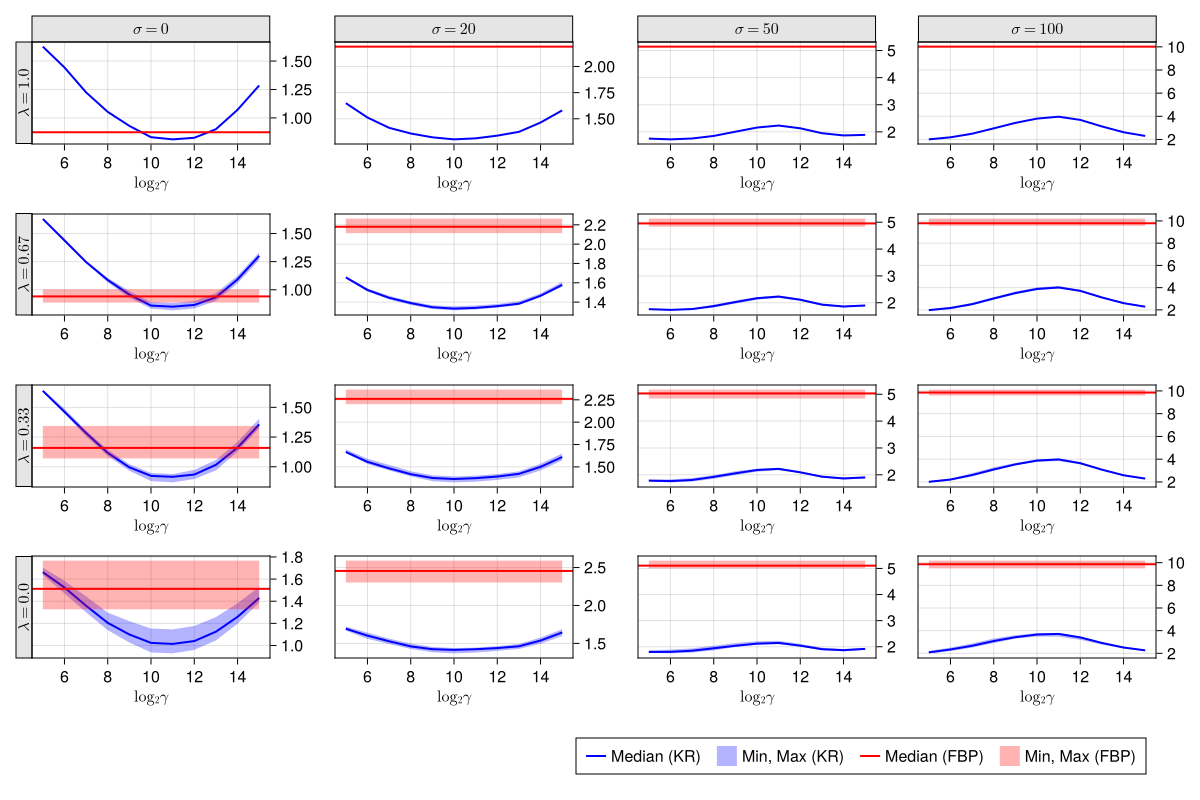}
\caption{For a fixed $M=100$ and $N=40$, $4 \times 4$ panels with different levels of noise $\sigma$ (column) and the regularity of the viewing angles $\lambda$ (row). The upper row and leftmost column represent more regular viewing and less perturbation, respectively. In each panel, we plot the band of RMSE with the maximums and minimums of $n_{MC} = 21$ Monte-Carlo simulations, obtained from the FBP (blue) and KR (red) algorithms. The blue and red curves represent their medians.}
\label{RKHS::MC::recons}
\end{figure}

For each noise level $\sigma \in \{0, 20, 50, 100\}$ and $\lambda \in \{0, 1/3, 2/3, 1 \}$, we generate sinograms $n_{MC} = 21$ times and execute the FBP and KR algorithms with $\gamma = 2^{k}, k = 5, \dots, 15$, $\nu = 2^{-l}, l = 5, \dots, 20$. Thus, we perform the FBP reconstruction $4 \times 4 \times 21 = 336$ times and the KR reconstruction $4 \times 4 \times 21 \times 11 \times 16 = 59136$ times. In each MC simulation with fixed $\sigma$ and $\lambda$, we calculate the RMSE for the FBP algorithm and $\min_{\nu} \mathrm{RMSE} (\gamma, \nu)$ for every grid of $\gamma$. Subsequently, we depict the band of their minima and maxima, alongside the line of their medians in \cref{RKHS::MC::recons}.

We arrive at similar conclusions as before. In noiseless sinograms, the FBP algorithm may outperform the KR algorithm if its tuning parameters $\gamma, \nu > 0$ stray far from their optimal values. Nonetheless, when the parameters are set to be optimal, the KR algorithm yields smaller RMSEs compared to the FBP algorithm. On the other hand, as sinogram noise increases, even the worst RMSE for our method remains smaller than the best case for the FBP algorithm, regardless of the tuning parameters. Furthermore, the scale of RMSEs in \cref{RKHS::MC::recons} indicates that the KR algorithm is much more robust to noise and irregular viewing angles.

The two algorithms exhibit an important difference with regards to \emph{dimension}. The FBP algorithm's approach varies between even and odd dimensions: filtering is a local operation for $n=3$, whereas filtering for $n=2$, the Hilbert transform, involves convolution with a nowhere-zero function, necessitating global information for reconstruction \cite{natterer2001mathematics}. In contrast, our methodology remains consistent across different dimensions $n$.

Though our methodology is dimension-agnostic, specific integral geometries can be leveraged to  ease the computational burden associated with solving the normal equation \eqref{normal::eq::tikho}. We demonstrate in \cref{sec:circulant} that when $n=2$, a parallel geometry only requires $O(NM^{2}+ NM \log N)$ operations, i.e. the same amount of computation as FBP implemented with the aid of the FFT.
\section{Invariant Kernels}\label{sec:invker}
\subsection{Rotation Invariant Kernel}\label{sec::RIK}
We have thus far explored the tomographic reconstruction problem without making any assumptions about the reproducing kernel $K$ on $\bB^{n}$. At this level of generality, the induced reproducing kernel for the projection image space will  depend on the orientation $R \in SO(n)$: for any $\vect{x}_{1}, \vect{x}_{2} \in \bB^{n-1}$,
\begin{equation}\label{induced::rk}
    \tilde{K}^{R}(\vect{x}_{1}, \vect{x}_{2})=
    \int_{I_{\|\vect{x}_{2}\|}} \int_{I_{\|\vect{x}_{1}\|}} K(R^{\top} [\vect{x}_{1}:z_{1}], R^{\top} [\vect{x}_{2}:z_{2}]) \rd z_{1} \rd z_{2}.
\end{equation}
When practitioners suspect a high degree of asymmetry in the original structure, they might opt to choose a suitable kernel that matches this asymmetry. In such cases, multiple induced kernels are likely useful, each corresponding to a different orientation. However, for most purposes, it would be desirable to use a kernel that ensures the induced kernel (and corresponding RKHS) remains \emph{invariant} regardless of the choice of orientations. Additionally, when the orientations $R_{1},\dots, R_{N} \in SO(n)$ are unknown as in cryo-EM, the dependence of \eqref{induced::rk} on $R \in SO(n)$ poses a major obstruction, as it becomes unclear which domain $\bH(\tilde{K}^{R})$ should be used for the micrographs. To address this issue, we consider a specific class of kernels known as rotation-invariant kernels. 

\begin{definition}
Let $E$ be a subset of $\bR^{n}$ containing $0$, on which the rotation group $SO(n)$ acts from the left. A bivariate function $K: E \times E \rightarrow \bR$ is called rotation-invariant if for any $R \in SO(n)$,
\begin{equation*}
    K(\vect{z}_{1},\vect{z}_{2})= K(R^{\top} \vect{z}_{1}, R^{\top} \vect{z}_{2}), \quad \vect{z}_{1},\vect{z}_{2} \in E. 
\end{equation*}
\end{definition}
When $K:\bB^{n} \times \bB^{n} \rightarrow \bR$ is a rotation-invariant kernel on $\bB^{n}$, it is straightforward that the induced kernel $\tilde{K}^{R}=(\sP_{R} \times \sP_{R})K$ in \eqref{induced::rk} does not depend on $R \in SO(n)$. 
Consequently, we will suppress the dependency on $R$ and will denote the induced kernel by $\tilde{K}$, and the induced generator at $\vect{x} \in \bB^{n-1}$ by $\tilde{k}_{\vect{x}}$. Moving forward, we will consistently assume that our chosen kernel $K:\bB^{n} \times \bB^{n} \rightarrow \bR$ is a continuous rotation-invariant kernel. To simplify notation, when there is no potential for ambiguity, we will write the RKHSs $\bH, \tilde{\bH}$ and $\bH_{R}$, instead of $\bH(K), \bH(\tilde{K})$ and $\bH_{R}(K)=\cR(\sP_{R}^{*})$,  respectively.

Below, we show  that the rotation-invariance on $K:\bB^{n} \times \bB^{n} \rightarrow \bR$ induces the reflection-invariance when $n \ge 3$. Note that the orthogonal group $O(n) = SO(n) \rtimes \{I, J \}$ is the semidirect product of $SO(n)$ and $\{I, J \}$. 

\begin{theorem}\label{RIK2OIK}
Let $n \ge 3$ and $K:\bB^{n} \times \bB^{n} \rightarrow \bR$ be a rotation-invariant kernel. Then $K$ is invariant under the left action of $O(n)$, i.e. for any $R \in O(n)$,
\begin{equation*}
    K(\vect{z}_{1},\vect{z}_{2})= K(R^{\top} \vect{z}_{1}, R^{\top} \vect{z}_{2}), \quad \vect{z}_{1},\vect{z}_{2} \in \bB^{n}. 
\end{equation*}
\end{theorem}


For any $f \in \bH$, its reflection $Jf$ defined by $Jf(\vect{z}):=f(J \vect{z})$ leads to the same  projection image at a conjugacy angle $J R^{\top} J \in SO(n)$:
\begin{align*}
\sP_{R} Jf(\vect{x}) =
\int_{I_{\|\vect{x}\|}} Jf(R^{\top} [\vect{x}:z]) \rd z = \int_{I_{\|\vect{x}\|}} f(J R^{\top} J [\vect{x}:-z]) \rd z = \sP_{J R^{\top} J} f(\vect{x}),
\end{align*}
This reveals that when the orientations are not known, it is impossible to determine the handedness of the structure. It is often more convenient and mathematically natural to work with $O(n)$-invariance, as $O(n)$ includes all the unitary representations in $\bR^{n}$. Nonetheless, \cref{RIK2OIK} implies that $SO(n)$-invariance is sufficient to retrieve $O(n)$-invariance in the case where $n \ge 3$.

\begin{proposition}\label{RIK2RIK}
Let $K:\bB^{n} \times \bB^{n} \rightarrow \bR$ be a rotation-invariant kernel. Then $\tilde{K}: \bB^{n-1} \times \bB^{n-1} \rightarrow \bR$ is invariant under the left action of $O(n-1)$.
\end{proposition}

Drawing an analogy, one might expect that if $K:\bB^{n} \times \bB^{n} \rightarrow \bR$ is translation-invariant, then $\tilde{K}$ would inherit this property. However, this is not true. Although the integration intervals $I_{\|\vect{x}_{1}\|}, I_{\|\vect{x}_{2}\|}$ for $\tilde{K}^{R}$ are independent of rotations, they do rely on the values of $\vect{x}_{1}$ and $\vect{x}_{2}$ in $\bB^{n-1}$. One potential approach to maintain translation invariance might involve extending the domain of $K$ from $\bB^{n}$ to $\bR^{n}$. Unfortunately, this approach presents a more significant challenge: it results in blowing up $\tilde{K}(\vect{x}_{1}, \vect{x}_{2})$ to infinity as $\int_{-\infty}^{+\infty} K([\vect{x}_{1}:z_{1}], [\vect{x}_{2}:z_{2}]) \rd z_{1}$ does not depend on $z_{2}$. Consequently, the hope that the induced kernel $\tilde{K}$ becomes isotropic is abandoned.

\begin{corollary}\label{gene::rot}
Let $K:\bB^{n} \times \bB^{n} \rightarrow \bR$ be a rotation-invariant kernel. If $R_{1} \sim R_{2} \in SO(n)$ as defined in \eqref{rot::equiv}, then the range space of the backprojections $\sP_{R_{1}}^{*}$ and $\sP_{R_{2}}^{*}$ are the same, i.e.
$\bH_{R_{1}}=\bH_{R_{2}}$.
\end{corollary}

While the induced RKHS $\tilde{\bH}=\bH(\tilde{K})$ remains invariant to orientations, \cref{gene::rot} highlights that the range space $\bH_{R}=\cR (\sP_{R}^{*})$ of the backprojection depends on the axis of the orientation. Thus, we cannot hope to get an expression for $\sP_{R_{i}}^{*} \circ \sP_{R_{i'}}:\bH \rightarrow \bH_{R_{i}}$ in terms of the relative angles $R_{i} R_{i'}^{\top}$. On the other hand, \cref{repn:btf} indicates that $\sP_{R_{i}} \circ \sP_{R_{i'}}^{*}:\tilde{\bH} \rightarrow \tilde{\bH}$ depends solely on the relative angles, and consequently, so does the Gram matrix $\mW$ of the X-ray representer theorem, as well.


\subsection{Representation Theory}\label{sec::kernel::repn}
When $K:\bB^{n} \times \bB^{n} \rightarrow \bR$ is rotation-invariant, the corresponding RKHS $\bH=\bH(K)$ also inherits the same invariance in that $f \in \bH$ if and only if $f(R^{\top} \, \cdot) \in \bH$ for any $R \in SO(n)$. This is because 
\begin{equation}\label{gen::rot}
    k_{\vect{z}_{1}}(R^{\top} \vect{z}_{2})=K(\vect{z}_{1}, R^{\top} \vect{z}_{2})=K(R \vect{z}_{1}, \vect{z}_{2})=k_{R \vect{z}_{1}}(\vect{z}_{2}).
\end{equation}
for any $\vect{z}_{1}, \vect{z}_{2} \in \bB^{n}$ and $R \in SO(n)$. More formally:
\begin{definition}
For a continuous rotation-invariant kernel $K:\bB^{n} \times \bB^{n} \rightarrow \bR$, define the left-regular representation to be a map $\rho:SO(n) \rightarrow \cB(\bH), R \mapsto \rho(R)$,  given by
\begin{equation*}
    (\rho(R)f)(\vect{z}):=f(R^{\top} \vect{z}), \quad R \in SO(n), \, \vect{z} \in \bB^{n}, \, f \in \bH.
\end{equation*}
\end{definition}
In this notation \eqref{gen::rot} can be re-interpreted as $\rho(R) k_{\vect{z}} = k_{R\vect{z}}$, illustrating that applying a rotation operator to a function is equivalent to the inverse rotation of the evaluation point. This correspondence between rotating functions and rotating evaluation points highlights that consecutive rotations of a function corresponds to cumulative rotations of the evaluation point. Moreover, the norm of the function remains  unaffected during this process, as the kernel is rotation invariant. It is thus not surprising that we can state the following proposition.

\begin{proposition}\label{unit:repre}
Given a continuous rotation-invariant kernel $K:\bB^{n} \times \bB^{n} \rightarrow \bR$, the left-regular representation $(\rho, \bH)$ is a unitary representation of $SO(n)$.
\end{proposition}

As a direct consequence, for any $R \in SO(n)$, we have 
$\rho(R)^{*}=\rho(R)^{-1}=\rho(R^{\top})$.
It follows that for any $f_{1}, f_{2} \in \bH$, we get $\langle \rho(R) f_{1}, f_{2} \rangle_{\bH(K)} =\langle f_{1}, \rho(R^{\top}) f_{2} \rangle_{\bH(K)}$,
implying that taking an inner product between two functions after rotating one function is equivalent to taking an inner product after rotating the other function in the opposite direction. 

The projection image at orientation $R \in SO(n)$ is equivalent to a projection image  {along the $e_{n}$-axis}, after rotating the original object by the left action of $\rho(R)$.  We thus arrive at the following theorem, which again highlights the natural analogy between the RKHS X-ray transform and the Euclidean projection - \cref{Euclid::image}.

\begin{theorem}[Representation]\label{repn:btf}
Let $K:\bB^{n} \times \bB^{n} \rightarrow \bR$ be a continuous rotation-invariant kernel. For any $R, R_{1}, R_{2} \in SO(n)$, it holds that
\begin{enumerate}
    \item The X-ray transform is obtained as $\sP_{R}=\sP_{I} \circ \rho(R): \bH \rightarrow \tilde{\bH}$.
    \item The Backprojection is obtained as $\sP_{R}^{*}=\rho(R^{\top}) \circ \sP_{I}^{*}: \tilde{\bH} \rightarrow \bH$.
    \item The Back-then-forward projection operation is obtained as $\sP_{R_{1}} \circ \sP_{R_{2}}^{*}=\sP_{I} \circ \rho(R_{1}R_{2}^{\top}) \circ \sP_{I}^{*}: \tilde{\bH} \rightarrow \tilde{\bH}$.   
\end{enumerate}
Therefore, for any $\vect{x}_{j}, \vect{x}_{k} \in \bB^{n-1}$,
\begin{align*}
    (\sP_{R_{1}} \circ \sP_{R_{2}}^{*} \tilde{k}_{\vect{x}_{k}})(\vect{x}_{j})=\langle (\rho(R_{1}R_{2}^{\top}) \circ \sP_{I}^{*}) \tilde{k}_{\vect{x}_{k}}, \sP_{I}^{*} \tilde{k}_{\vect{x}_{j}} \rangle_{\bH}.
\end{align*}
\end{theorem}

As a direct consequence, we get  $\bH_{R}^{\perp}= \cN(\sP_{R})= \rho(R^{\top}) \cN(\sP_{I})= \rho(R^{\top}) \bH_{I}^{\perp}$. At this point, it is worthwhile to revisit our representer theorem \cref{tikhonov}, which made no invariance assumption. If we do assume rotation-invariance, we get a better formulation for the corresponding Gram matrix $\mW=[w_{ij, i'j'}]$. Specifically, we can write
\begin{align}\label{Gram::elem}
    w_{ij, i'j'} &=\langle \sP_{R_{i}}^{*} \tilde{k}_{\vect{x}_{j}}, \sP_{R_{i'}}^{*} \tilde{k}_{\vect{x}_{j'}} \rangle_{\bH} 
    =(\sP_{R_{i}} \circ \sP_{R_{i'}}^{*} \tilde{k}_{\vect{x}_{j'}})(\vect{x}_{j}) \\ \nonumber
    &=(\sP_{I} \circ \rho(R_{i} R_{i'}^{\top}) \circ \sP_{I}^{*} \tilde{k}_{\vect{x}_{j'}})(\vect{x}_{j}) \\ \nonumber
    &=\int_{I_{\|\vect{x}_{j'}\|}} \int_{I_{\|\vect{x}_{j}\|}} K([\vect{x}_{j}:z_{1}], R_{i} R_{i'}^{\top} [\vect{x}_{j'}:z_{2}]) \rd z_{1} \rd z_{2}.
\end{align}

And, recalling that when $R_{i}=R_{i'}$, the back-then-forward projection is the identity map on $\tilde{\bH}$, i.e. $\sP_{R_{i}} \circ \sP_{R_{i}}^{*}=\sP_{I} \circ \rho(I) \circ \sP_{I}^{*}=I$, we arrive at
\begin{align*}
    w_{ij, ij'} =\tilde{K}(\vect{x}_{j},\vect{x}_{j'})=\int_{I_{\|\vect{x}_{j'}\|}} \int_{I_{\|\vect{x}_{j}\|}} K([\vect{x}_{j}:z_{1}], [\vect{x}_{j'}:z_{2}]) \rd z_{1} \rd z_{2}.
\end{align*}

We now move on to the question of explicit decompositions/representations of continuous $O(n)$-invariant kernels and their corresponding RKHS. To this aim, we can make use of classical results concerning spherical harmonics, which can be found in \cref{sec::sphe::har}. For any $\lambda >-1/2$, the Gegenbauer polynomials $C_{l}^{\lambda}$ of degree $l$ are orthogonal polynomials on $[-1, +1]$ with the weight function $w^{2\lambda-1}$, normalized by $C_{l}^{\lambda}(1)=1$ as in \cite{natterer2001mathematics}. The RKHS $\bH(C_{l}^{n/2-1})$ spans $\cH_{l}(\bS^{n-1})$ and possesses an ONB $\{(p^{n/2-1}_{l} |\bS^{n-2}|)^{1/2} Y_{lk}(\theta): k=1,2,\dots, N(n,l) \}$, see \eqref{rk::Gegen} and \eqref{Gegen::norm} for the explicit formula of the constant $p^{n/2-1}_{l} = \| C_{l}^{n/2-1}\|^{2}_{\cL_{2}(w^{2 \lambda-1})} > 0$. Using this result,  we show that any continuous $O(n)$-invariant kernel $K:\bB^{n} \times \bB^{n} \rightarrow \bR$ can be decomposed as a countable sum of the multiplication of radial kernels $c_{l}$ and Gegenbauer polynomials $C_{l}^{n/2-1}$, thus $K$ is entirely determined by radial kernels.

\begin{theorem}[Generalized Schoenberg]\label{scheon::gen}
The following are equivalent.
\begin{enumerate}[1)]
\item A continuous kernel $K:\bB^{n} \times \bB^{n} \rightarrow \bR$ is invariant under the left action of $O(n)$.
\item There exists continuous radial kernels $c_{l}:[0,1] \times [0,1] \rightarrow \bR$ with $c_{l}(0,0)=0$ for $l \ge 1$ and $\max_{r \in [0,1]} \sum_{l=0}^{\infty} c_{l}(r,r) < \infty$
such that $K=\sum_{l=0}^{\infty} c_{l} \otimes C_{l}^{n/2-1}$, i.e. 
for any $r_{1}, r_{2} \in [0,1]$ and $\theta_{1}, \theta_{2} \in \bS^{n-1}$,
\begin{equation}\label{sch::series}
    K(r_{1} \theta_{1}, r_{2} \theta_{2})=\sum_{l=0}^{\infty} c_{l}(r_{1}, r_{2}) C_{l}^{n/2-1}(\theta_{1}^{\top} \theta_{2}),
\end{equation}
where the sum converges absolutely and uniformly.
\end{enumerate}
Also, the $l$-th radial kernel $c_{l}$ is given by
\begin{equation*}
    c_{l}(r_{1},r_{2})
    =\frac{1}{p^{n/2-1}_{l} |\bS^{n-2}|} \int_{\bS^{n-1}} K(r_{1}\theta_{1}, r_{2}\theta_{2}) C_{l}^{n/2-1}(\theta_{1}^{\top} \theta_{2}) \rd \theta_{2},    
\end{equation*}
independent of the choice of $\theta_{1} \in \bS^{n-1}$. \end{theorem}

Our next theorem is a generalization of Aronszajn \cite[Section 6]{aronszajn1950theory} and spectrally characterises the RKHS associated with an invariant kernel, where \eqref{Picard::cri} might be referred to as a Picard criterion.

\begin{theorem}[Generalized Mercer]\label{Xray::Aronszajn}
Let $K:\bB^{n} \times \bB^{n} \rightarrow \bR$ be an $O(n)$-invariant countinuous kernel, associated with continuous radial kernels $c_{l}:[0,1] \times [0,1] \rightarrow \bR$ as in \cref{scheon::gen}. Then, the corresponding RKHS and norm are given by 
\begin{align}\label{Picard::cri}
    \bH(K) =& \left\{\sum_{l=0}^{\infty} \sum_{k=1}^{N(n,l)} h_{lk}(r) Y_{lk}(\theta): h_{lk} \in \bH(c_{l}), \sum_{l=0}^{\infty} \sum_{k=1}^{N(n,l)} \frac{\|h_{lk}\|^{2}_{\bH(c_{l})}}{p^{n/2-1}_{l} |\bS^{n-2}|} <\infty \right\}, \nonumber \\ 
    &\left\| \sum_{l=0}^{\infty} \sum_{k=1}^{N(n,l)} h_{lk} Y_{lk} \right\|^{2}_{\bH(K)}= \sum_{l=0}^{\infty} \sum_{k=1}^{N(n,l)} \frac{\|h_{lk}\|^{2}_{\bH(c_{l})}}{p^{n/2-1}_{l} |\bS^{n-2}|} < \infty,
\end{align}
where $\left\{Y_{lk}(\theta): k=1,2,\dots, N(n,l) \right\}$ is an orthonormal basis (ONB) of $\cH_{l}(\bS^{n-1})$ with respect to the $\cL_{2}$ norm.
\end{theorem}

\begin{remark}\label{left::ang::ex}
Consider a simple case in \cref{Xray::Aronszajn} where $K=c \otimes C_{l}^{n/2-1}$ for some $l \in \bN$. We further assume that the radial kernel satisfies $c(r_{1}, r_{2})=e(r_{1}) e(r_{2})$ for some continuous function $e:[0,1] \rightarrow \bR$, so that $\bH(c)$ is simply the one-dimensional space spanned by $e$ with $\|e\|_{\bH(c)}=1$ by Proposition 2.19 in \cite{paulsen2016introduction}. In this case, $\bH(K)=\text{span} \left\{e(r)Y_{lk}(\theta): k=1,2,\dots, N(n,l) \right\}$ is an $N(n,l)$-dimensional space with an ONB
\begin{equation}\label{HK::ONB}
    \left\{(p^{n/2-1}_{l} |\bS^{n-2}|)^{1/2} e(r) Y_{lk}(\theta): k=1,2,\dots, N(n,l) \right\}.
\end{equation}

To describe the unitary representation $\rho(R)$ of $\bH(K)$,
recall that the Wigner D-matrix $D_{l}(R)$ of degree $l$ in \eqref{Wigner::repn} is the irreducible unitary representation of $\cH_{l}(\bS^{n-1})$. Therefore, if we identify $\bH(K)$ as an inner product space with respect to the ONB in \eqref{HK::ONB}, then the unitary representation $\rho(R)$ reduces to $D_{l}(R)$.
\end{remark}

The remark above sheds light on how truncated singular value decomposition \cite{huang2017restoration, zhao2013fourier} can be incorporated into our RKHS framework. When viewing the X-ray transform as an integral operator $\sP_{R}: \cL_{2}(\bB^{n}) \rightarrow \cL_{2}(SO(n) \times \bB^{n-1}, W^{-1/2})$, its singular functions with non-zero singular values become $\{ Z_{ml}(r) Y_{lk}(\theta): m \in \bN_{0}, 0 \le l \le m, 1 \le k \le N(n,l), 2|(m - l) \}$,
where $Z_{ml}$ is the Zernike radial polynomial, and $\{Y_{lk}(\theta): k=1, \dots, N(n,l) \}$ is a carefully chosen ONB for $\cH_{l}(\bS^{n-1})$ \cite{maass1987x, izen1988series}. Truncating with respect to the $m$ index at level $T \in \bN$ and applying the pseudo-inverse of $\sP_{R}$ for reconstruction, it is effectively the same as working within our RKHS framework with the kernel
\begin{equation*}
    K(r_{1} \theta_{1}, r_{2} \theta_{2}) =\sum_{m=0}^{T} \sum_{\substack{l \le m \\ 2|(m - l)}} Z_{ml}(r_{1}) Z_{ml}(r_{2}) C_{l}^{n/2-1}(\theta_{1}^{\top} \theta_{2}),
\end{equation*}
and the matrix notation of the unitary representation $\rho(R)$ of $\bH(K)$ becomes
\begin{equation*}
    [\rho(R)]_{mlk, m'l'k'} = \delta_{ll'} \cdot [D_{l}(R)]_{k k'},
\end{equation*}
where $\delta_{ll'}$ is the Kronecker delta. Note that this kernel is \emph{not} strictly p.d. as it spans a finite-dimensional space, as opposed to the Gaussian kernel. Similarly, if we wish to interpret the RKHS norm as a measure of smoothness, we can derive a kernel in terms of Fourier-Bessel functions, which are the eigenfunctions of the Laplacian in the unit disk with vanishing Dirichlet boundary conditions \cite{zhao2013fourier, zhao2016fast}.

The preceding argument illustrates why our RKHS approach is less likely to produce artifacts. In situations where viewing angles are limited, truncated SVD imposes restrictions on the angular frequency, adhering to the Nyquist-Shannon sampling theorem. The spherical average of any spherical harmonic with a positive degree is always zero, and consequently, the spherical average of the original structure can only be influenced by the first spherical harmonic $Y_{0}$, which is a constant function. This can result in generating negative signals in the reconstruction image, even when the true image only contains positive values, leading to a symmetric reconstruction error, as vividly demonstrated in \cref{sec::simul}. Even if one performs the sinogram interpolation to reduce the magnitude of error, they remain constrained by the limits of resolution, due to the discussion following \cref{HLCC}.

In contrast, \cref{scheon::gen} reveals that the RKHS approach can adeptly combine a countably infinite multitude of band-limited kernels, such as Gaussian kernels. Also, the reconstruction does not rely on the Fourier transform but instead employs linear regression. Additionally, the level of desired smoothness can be finely controlled via the Picard condition \eqref{Picard::cri}. Consequently, our approach appears to avoid generating negative signals while minimizing the squared loss based on the provided sinogram.

\section{Circulant Algorithm}\label{sec:circulant}
We now show that, in the case of a parallel geometry in the plane, one  {can} employ a rotation-invariant kernel, and implement an RKHS-based   reconstruction efficiently. Let $n=2$ and assume that the projection angles are equiangular. Using the isomorphic relationships 
\begin{equation}\label{eqn:SO2:iden}
    \phi \in \bR/2\pi\bZ \ \cong \ \theta(\phi)= \begin{pmatrix}
    \sin \phi \\ \cos \phi
    \end{pmatrix} \in \bS^{1}  \ \cong \ E(\phi)= \begin{pmatrix}
    \cos \phi & -\sin \phi \\
    \sin \phi & \cos \phi
    \end{pmatrix} \in SO(2),
\end{equation}
we identify $R_{i}=E(\phi_{i}) \in SO(2)$ with $\phi_{i}=i \cdot 2\pi/N$ for $i=1, \dots, N$, and the relative angles are given by $R_{i} R_{i'}^{\top}=E(\phi_{i}-\phi_{i'})$. 
The Gram matrix $\mW \in \bR^{NM \times NM}$ is a $N \times N$ block Toeplitz matrix of $M \times M$ matrices $\mW_{i i'}$, whose $j, j'$ element is given by \eqref{Gram::elem}: 
\begin{equation*}
    \left[ \mW_{i i'} \right]_{j j'} :=w_{ij, i'j'} =\langle \sP_{\phi_{i}}^{*} \tilde{k}_{\vect{x}_{j}}, \sP_{\phi_{i'}}^{*} \tilde{k}_{x_{j'}} \rangle_{\bH(K)} 
    =(\sP_{(\phi_{i}-\phi_{i'})} \circ \sP_{0}^{*} \tilde{k}_{x_{j'}})(x_{j}).
\end{equation*}
As $\mW_{i i'}$ only depends on $(i-i') (\text{mod } N)$, we denote each block $\mW_{i i'}$ by $\mW_{(i-i')}$ so that the Gram matrix becomes a block circulant matrix $\mW = [\mW_{i i'}]= [\mW_{(i-i')}]$
with $\mW_{-i}=\mW_{N-i}=\mW_{i}^{\top}$.
The normal equation in \eqref{normal::eq::tikho} in terms of the block matrix is given by
\begin{align*}
    \begin{pmatrix}
    \mW_{0}+\gamma I & \mW_{N-1} & \dots & \mW_{1} \\
    \mW_{1} & \mW_{0}+\gamma I & \dots & \mW_{2} \\
    \vdots & \vdots & \ddots & \vdots \\
    \mW_{N-1} & \mW_{N-2} & \dots & \mW_{0}+\gamma I
    \end{pmatrix}
    \begin{pmatrix}
    \ma_{1} \\ \ma_{2} \\ \vdots \\ \ma_{N-1} 
    \end{pmatrix}=
    \begin{pmatrix}
    \mY_{1} \\ \mY_{2} \\ \vdots \\ \mY_{N-1} 
    \end{pmatrix}.
\end{align*}
For $i=1, \dots, N$, $\mY_{i} = \sum_{k=0}^{N-1} \mW_{i-k} \ma_{k} + \gamma \ma_{i}$,
and its discrete ($M$-dimensional) Fourier transform becomes
\begin{align*}
    \widehat{\mY}_{i}:&=\frac{1}{N} \sum_{l=0}^{N-1} e^{-2 \pi \imath i l/N} \mY_{l} \\
    &=\frac{1}{N} \sum_{l=0}^{N-1} \sum_{k=0}^{N-1} e^{-2 \pi \imath i l/N} \mW_{l-k} \ma_{k} + \frac{\gamma}{N} \sum_{l=0}^{N-1} e^{-2 \pi \imath i l/N} \ma_{l} \\
    &=\frac{1}{N} \sum_{k=0}^{N-1} e^{-2 \pi \imath i k/N} \ma_{k} \left( \sum_{l=0}^{N-1}  e^{-2 \pi \imath i (l-k)/N} \mW_{l-k} \right) + \frac{\gamma}{N} \sum_{l=0}^{N-1} e^{-2 \pi \imath i l/N} \ma_{l} \\
    &= \left( \sum_{l=0}^{N-1} e^{-2 \pi \imath i l/N} \mW_{l} +\gamma I \right) \left( \frac{1}{N} \sum_{k=0}^{N-1} e^{-2 \pi \imath i k/N} \ma_{k} \right) \\
    &=N \left( \frac{1}{N} \sum_{l=0}^{N-1} e^{-2 \pi \imath i l/N} \mW_{l} +\frac{\gamma}{N} I \right) \widehat{\ma}_{i} =N \left( \widehat{\mW}_{i} +\frac{\gamma}{N} I \right) \widehat{\ma}_{i}.
\end{align*}
Hence, we obtain the following algorithm.

\begin{algorithm}
\caption{Kernel Reconstruction with Parallel Geometry}\label{alg::ker::recons}
\begin{algorithmic}
\Require $\vect{Y} \in \bR^{N \times M}$
\Ensure $\ma \in \bR^{N \times M}$
\Procedure{KR}{$Y$}\Comment{Solution of the normal equation}
    \For{$i \gets 1$ to $N$}
        \For{$j, j' \gets 1$ to $M$}
            \State {$[\mW_{i}]_{j j'} \gets (\sP_{\phi_{i}} \circ \sP_{0}^{*} \tilde{k}_{x_{j'}})(x_{j})$}
        \EndFor
    \EndFor
    \State{FFT for $\widehat{\mW}_{i} \gets \frac{1}{N} \sum_{l=0}^{N-1} e^{-2 \pi \imath i l/N} \mW_{l}$}
    \State{FFT for $\widehat{\mY}_{i} \gets \frac{1}{N} \sum_{l=0}^{N-1} e^{-2 \pi \imath i l/N} \mY_{l}$}\Comment{$O(NM \log N)$}
    \For{$i \gets 1$ to $N$}    
        \State {$\widehat{\ma}_{i} \gets \frac{1}{N} \left( \widehat{\mW}_{i} +\frac{\gamma}{N} I \right)^{-1} \widehat{\mY}_{i}$}\Comment{$O(NM^{2})$}
    \EndFor
    \State{Inverse FFT for $\ma_{i} \gets \sum_{l=0}^{N-1} e^{2 \pi \imath i l/N} \widehat{\ma}_{l}$}\Comment{$O(NM \log N)$}
    \State \textbf{return} $\ma$ 
\EndProcedure
\end{algorithmic}
\end{algorithm}

If the kernel and the grid are pre-determined so that regularized block inverse matrices $( \widehat{\mW}_{i} +\gamma I/N )^{-1}$ are stored, then the algorithm described above can be efficiently solved using fast Fourier transform (FFT) with a computational complexity of $O(NM^{2}+NM \log N)$.
This computational complexity is equivalent to that of the discretized FBP algorithm. Moreover, this approach not only offers computational efficiency but also enhances stability by performing inversions of $M \times M$ matrices instead of dealing with the Gram matrix of size $NM \times NM$.

\section{Concluding Remarks}\label{sec:fur}
While our primary focus has been on the X-ray transform, our RKHS framework can be readily extended to the $k$-plane transform \cite{christ1984estimates}. For $1 \le k \le n$, the $k$-plane transform $\sK_{R}: \cL_{2}(\bB^{n}) \rightarrow \cL_{2}(\bB^{n-k})$ at orientation $R \in SO(n)$ is given by
\begin{align*}
    \sK_{R} f(\vect{x}) =
    \int_{W(\vect{x}) \cdot \bB^{k}} f(R^{\top} [\vect{x}:z]) \rd z, \quad \vect{x} \in \bB^{n-k}.
\end{align*}
Define an induced kernel $\tilde{K}^{R}: \bB^{n-k} \times \bB^{n-k} \rightarrow \bR$ as follows:
\begin{equation*}
    \tilde{K}^{R}(\vect{x}_{1}, \vect{x}_{2}):=
    \int_{W(\vect{x}_{2}) \cdot \bB^{k}} \int_{W(\vect{x}_{1}) \cdot \bB^{k}} K(R^{\top} [\vect{x}_{1}:z_{1}], R^{\top} [\vect{x}_{2}:z_{2}]) \rd z_{1} \rd z_{2}.
\end{equation*}

With this extension, our fundamental theorem \cref{rkhs::image::re} can be easily generalized to the $k$-plane transform $\sK_{R}$ using a similar proof mechanism. Consequently, a unified treatment of both the Radon transform ($k= n-1$) and the X-ray transform ($k=1$) is possible.

\begin{theorem}\label{rkhs::kplane} For any $R \in SO(n)$, $\sK_{R}: \bH(K) \rightarrow \bH(\tilde{K}^{R})$ is a contractive, surjective, linear map. Its adjoint operator $\sK_{R}^{*}: \bH(\tilde{K}^{R}) \rightarrow \bH(K)$ is an isometry, uniquely determined by 
\begin{equation*}
    \sK_{R}^{*} (\tilde{k}^{R}_{\vect{x}} ) = \int_{W(\vect{x}) \cdot \bB^{k}} k_{R^{\top} [\vect{x}:z]} \rd z, \quad \vect{x} \in \bB^{n-1},
\end{equation*}
in the sense of Bochner integration.
Let $\bH_{R}(K):=\cR (\sK_{R}^{*})$ be the range space of the $k$-plane backprojection. Then,
\begin{enumerate}
    \item $\sK_{R} \circ \sK_{R}^{*}: \bH(\tilde{K}^{R}) \rightarrow \bH(\tilde{K}^{R})$ is the identity.
    \item $\bH_{R}(K)=(\cN(\sK_{R}))^{\perp}, (\bH_{R}(K))^{\perp}=\cN(\sK_{R})$.
    \item Any $f \in \bH(K)$ can be uniquely expressed as a direct sum:
    \begin{equation*}
        f=(\sK_{R}^{*} \circ \sK_{R})f \oplus (f-(\sK_{R}^{*} \circ \sK_{R})f) \in \bH_{R}(K) \oplus \cN(\sK_{R}),
    \end{equation*}
    thus $f \in \bH_{R}(K)$ if and only if $f=(\sK_{R}^{*} \circ \sK_{R})f$.
    \item For $g \in \bH(\tilde{K}^{R})$, we have
    \begin{align*}
        &\left\{f \in \bH(K): g=\sK_{R}f \right\}=\sK_{R}^{*} g \oplus \cN(\sK_{R}), \\ 
        &\|g\|_{\bH(\tilde{K}^{R})}=\|\sK_{R}^{*} g\|_{\bH(K)}= \min \left\{\|f\|_{\bH(K)} : g=\sK_{R}f \right\}.
    \end{align*}
\end{enumerate} 
\end{theorem}

\begin{appendix}
\section{Gaussian Kernel Reconstruction}\label{sec:Gauss:ker}
As discussed in \cref{sec::simul}, the computation of the Gram matrix $\mW=\left[w_{ij, i'j'} \right] \in \bR^{NM \times NM}$ can be demanding due to the necessity of performing the double integration in \eqref{weight::equiv} $O(N^{2} M^{2})$ times. However, when employing the Gaussian kernel, arguably the most frequently used isotropic kernel, this computational burden can be significantly mitigated due to the availability of an explicit formula. The (truncated) Gaussian kernel with a tuning parameter $\gamma>0$ is defined by
\begin{equation*}
    K(\vect{z}_{1},\vect{z}_{2}):=\exp(-\gamma \|\vect{z}_{1}-\vect{z}_{2}\|^{2}), \quad \vect{z}_{1}, \vect{z}_{2} \in \bB^{n}.
\end{equation*}


The error function on $\bR$ is defined by
\begin{equation*}
    \erf(z):=\frac{2}{\sqrt{\pi}} \int_{0}^{z} e^{-t^{2}} \rd t, \quad z \in \bR.
\end{equation*}
By the integral by parts, one can easily check that 
\begin{equation*}
    \Phi(z):=\sqrt{\pi} \int_{0}^{z} \erf(t) \rd t =\sqrt{\pi} z \cdot \erf(z) +e^{-z^{2}}-1.
\end{equation*}

\begin{proposition}\label{gauss::adj}
Let $K: \bB^{n} \times \bB^{n} \rightarrow \bR$ be the Gaussian kernel with $\gamma>0$.
\begin{enumerate}[1)]
    \item For $\vect{x}_{1},\vect{x}_{2} \in \bB^{n-1}$, the induced kernel $\tilde{K}: \bB^{n-1} \times \bB^{n-1} \rightarrow \bR$ is given by
    \begin{equation*}
    \tilde{K}(\vect{x}_{1},\vect{x}_{2})
    = \frac{e^{-\gamma \|\vect{x}_{1}-\vect{x}_{2}\|^{2}}}{2\gamma} \sum_{i, j \in \bZ_{2}} (-1)^{i+j} \Phi[\sqrt{\gamma}((-1)^{i} W(\vect{x}_{1})+ (-1)^{j}W(\vect{x}_{2}))].    
    \end{equation*}
    \item For $\vect{x} \in \bB^{n-1}, \vect{z} \in \bB^{n}$ and $R \in SO(n)$, the backprojection of the induced generator is given by
    \begin{equation*}
    \sP_{R}^{*} ( \tilde{k}_{\vect{x}} )(\vect{z}) 
    = \frac{\sqrt{\pi} e^{-\gamma \|\vect{x}-\mP_{R}(\vect{z})\|^{2}}}{2\sqrt{\gamma}} \sum_{i \in \bZ_{2}} (-1)^{i} \erf[\sqrt{\gamma} ((-1)^{i} W(\vect{x})-\mP_{R}(\vect{z}))].
    \end{equation*}
\end{enumerate}
\end{proposition}

\begin{lemma}\label{rot::dist}
Let $\vect{x}_{1}, \vect{x}_{2} \in \bB^{n-1}$, and $R \in SO(n)$. Also, let $(\tilde{R}, \vect{r}) \in SO(n-1) \times \bS^{n-1}$ be the Euler decomposition of $R \in SO(n)$ as in \cref{Euler::repre},
and $(\tilde{\vect{r}}, \arccos (r)) \in \bS^{n-2} \times [0, \pi]$ be the spherical decomposition of $\vect{r} \in \bS^{n-1}$ as in \eqref{sphe::red}, where $r=e_{n}^{\top} R e_{n}$. 
Denote by
\begin{alignat*}{2}
    &[\tilde{\vect{x}}_{2 R}:x_{2 R}] := \tilde{R}^{\top} \vect{x}_{2}, \quad &&\tilde{\vect{x}}_{2 R} \in \bB^{n-2}, \, x_{2 R}=e_{n-1}^{\top} \tilde{R}^{\top}\vect{x}_{2} \in [-1,+1], \\
    &[\tilde{\vect{x}}_{1 \vect{r}}:x_{1 \vect{r}}] := E(\tilde{\vect{r}}) \vect{x}_{1}, \quad &&\tilde{\vect{x}}_{1 \vect{r}} \in \bB^{n-2}, \, x_{1\vect{r}} = \tilde{\vect{r}}^{\top} \vect{x}_{1} \in [-1,+1], \\
    &\mu_{1}^{R} := r \cdot x_{1 \vect{r}}-x_{2 R}, \quad
    &&\mu_{2}^{R} := x_{1 \vect{r}}-r \cdot x_{2 R}.
\end{alignat*}

\begin{enumerate}[1)]
    \item If $R \sim I$, i.e. $|r| = 1$, then $\vect{x}_{1}=E(\tilde{\vect{r}}) \vect{x}_{1}$ and
    \begin{align*}
    \|[\vect{x}_{1}:z_{1}]-R^{\top}[\vect{x}_{2}:z_{2}]\|^{2}
    =\|[\tilde{\vect{x}}_{1 \vect{r}}:r x_{1 \vect{r}}]-\tilde{R}^{\top} \vect{x}_{2}\|^{2}+ (z_{1} -r z_{2})^{2}.
    \end{align*}
    \item If $R \nsim I$, i.e. $|r| < 1$, then
    \begin{align*}
        &\|[\vect{x}_{1}:z_{1}]-R^{\top}[\vect{x}_{2}:z_{2}]\|^{2} \\
        =&\|\tilde{\vect{x}}_{1 \vect{r}}-\tilde{\vect{x}}_{2 R}\|^{2}+
        \left[\begin{pmatrix}
        z_{1} \\ z_{2}
        \end{pmatrix}-\frac{1}{w(r)}\begin{pmatrix}
        \mu_{1}^{R} \\ \mu_{2}^{R}
        \end{pmatrix}\right]^{\top}
        \begin{pmatrix}
        1 & -r \\
        -r & 1
        \end{pmatrix}
        \left[\begin{pmatrix}
        z_{1} \\ z_{2}
        \end{pmatrix}-\frac{1}{w(r)}\begin{pmatrix}
        \mu_{1}^{R} \\ \mu_{2}^{R}
        \end{pmatrix}\right].
    \end{align*}
\end{enumerate}
\end{lemma}

 {The following proposition shows that the Gram matrix $\mW$ in \eqref{Gram::elem} can be computed via the cumulative distribution function (CDF) of the standard bivariate normal distribution, regardless of the dimension $n \in \mathbb{N}$. It also demonstrates that the relative angles between the axes of rotations $R_{1}, R_{2} \in SO(n)$, which is $r= \vect{r}_{1}^{\top} \vect{r}_{2}$, could be interepreted as the correlation.}

\begin{proposition}\label{gauss::weight}
Let $K: \bB^{n} \times \bB^{n} \rightarrow \bR$ be the Gaussian kernel with $\gamma>0$, and let $R_{1}, R_{2} \in SO(n)$ be orientations with the relative angle $R=R_{1}R_{2}^{\top}$. Using the same notation to \cref{rot::dist}, the back-then-forward projection of the induced generator is given as follows:
\begin{enumerate}[1)]
    \item If $R_{1} \sim R_{2}$, then
    \begin{align*}
        &(\sP_{R_{1}} \circ \sP_{R_{2}}^{*} \tilde{k}_{\vect{x}_{2}})(\vect{x}_{1})
        = (\sP_{R} \circ \sP_{I}^{*} \tilde{k}_{\vect{x}_{2}})(\vect{x}_{1})
        = \frac{\exp(-\gamma \|[\tilde{\vect{x}}_{1 \vect{r}}:r x_{1 \vect{r}}]-\tilde{R}^{\top} \vect{x}_{2}\|^{2})}{2\gamma} \times \\
        & \quad \sum_{i, j \in \bZ_{2}} (-1)^{i+j} \Phi [\sqrt{\gamma}((-1)^{i} W(\vect{x}_{1})+ (-1)^{j}W(\vect{x}_{2}))].
    \end{align*}
    \item If $R_{1} \nsim R_{2}$, then    
    \begin{align*}
        &(\sP_{R} \circ \sP_{I}^{*} \tilde{k}_{\vect{x}_{2}})(\vect{x}_{1})
        = (\sP_{R} \circ \sP_{I}^{*} \tilde{k}_{\vect{x}_{2}})(\vect{x}_{1}) 
        =\frac{\pi \cdot \exp(-\gamma \|\tilde{\vect{x}}_{1 \vect{r}}-\tilde{\vect{x}}_{2 R}\|^{2})}{\gamma w(r)} \times \\
        &\quad \sum_{i, j \in \bZ_{2}} (-1)^{i+j} 
        \Phi_{2}^{r}\left[\sqrt{2\gamma} ((-1)^{i} w(r) W(\vect{x}_{1})-\mu_{1}^{R}), \sqrt{2\gamma} ((-1)^{j} w(r) W(\vect{x}_{2})-\mu_{2}^{R})\right] ,
    \end{align*}
    where $\Phi_{2}^{\rho}(a,b)$ is the CDF of the standardized bivariate normal distribution $(Z_{1}, Z_{2})$ with correlation $\rho \in (-1, +1)$, i.e. $\Phi_{2}^{\rho}(a,b)=\bP(Z_{1} \le a, Z_{2} \le b)$.
\end{enumerate}    
\end{proposition}
In the case where $|r| < 1$, we could approximate the bivariate normal integral $\Phi_{2}^{r}$ in terms of the error function for fast computations, see \cite{tsay2021simple}. Moreover, in the case where $n=2$, we get a simpler expression for \cref{gauss::adj,gauss::weight}. Using an isomorphism \eqref{eqn:SO2:iden},
we identify $R_{1}, \dots, R_{N} \in SO(2)$ with $\phi_{1}, \dots, \phi_{N} \in \bR/2\pi\bZ$, i.e. $R_{i}=E(\phi_{i})$ for $i=1, \dots, N$. Note that the relative angles are given by $R_{i} R_{i'}^{\top}=E(\phi_{i}-\phi_{i'})$.
\begin{corollary}\label{cor::gauss::weight}
Let $\vect{z}=[z_{1}:z_{2}] \in \bB^{2}$, $x_{1}, \dots, x_{M} \in \bB^{1}=[-1,+1]$, and $\phi_{1}, \dots, \phi_{N} \in \bR/2\pi\bZ$. Also, let $K: \bB^{n} \times \bB^{n} \rightarrow \bR$ be the Gaussian kernel with $\gamma>0$. The backprojection of the induced generator is given by
\begin{align*}
\sP_{\phi_{i}}^{*} ( \tilde{k}_{x_{j}} )(\vect{z}) 
=\frac{e^{-\gamma \|x_{j}-\tilde{z}_{\phi_{i}}\|^{2}} \sqrt{\pi}}{2\sqrt{\gamma}} \left[\erf(\sqrt{\gamma}(w(x_{j})-z_{\phi_{i}}))-\erf(\sqrt{\gamma}(-w(x_{j})-z_{\phi_{i}})) \right],
\end{align*}
where $\tilde{z}_{\phi_{i}}+ \imath z_{\phi_{i}} = e^{\imath \phi_{i}} (z_{1} + \imath z_{2})$.
The element of the Gram matrix $\mW$ in \eqref{Gram::elem} is given as follows: Denoting by $s_{ii'}=\sin (\phi_{i}-\phi_{i'})$ and $c_{ii'}=\cos (\phi_{i}-\phi_{i'})$, it holds that
\begin{enumerate}[1)]
    \item If $s_{ii'}=0$, then
    \begin{align*}
        &(\sP_{\phi_{i}} \circ \sP_{\phi_{i'}}^{*} \tilde{k}_{x_{j'}})(x_{j})
        =\tilde{K}(c_{ii'} \cdot x_{j}, x_{j'}) \\
        &=\frac{\exp(-\gamma (c_{ii'} \cdot x_{j}-x_{j'})^{2})}{2\gamma} \sum_{k, l \in \bZ_{2}} (-1)^{k+l} \Phi[\sqrt{\gamma}((-1)^{k} w(x_{j})+(-1)^{l} w(x_{j'}))].
    \end{align*}
    \item If $s_{ii'} \neq 0$, then
    \begin{align*}
        &(\sP_{\phi_{i}} \circ \sP_{\phi_{i'}}^{*} \tilde{k}_{x_{2}})(x_{1})
        = \frac{\pi}{\gamma |s_{ii'}|} \times \\
        &\quad \sum_{k, l \in \bZ_{2}} (-1)^{k+l} \Phi_{2}^{c_{ii'}}\left[\sqrt{2\gamma} ((-1)^{k} s_{ii'} w(x_{j})-\mu_{1}^{ii'}), \sqrt{2\gamma} ((-1)^{l} s_{ii'} w(x_{j'})-\mu_{2}^{ii'})\right],
    \end{align*}
    where
    \begin{equation*}
\begin{pmatrix}
\mu_{1}^{ii'} \\ \mu_{2}^{ii'}
\end{pmatrix}=
\begin{pmatrix}
c_{ii'} & -1 \\
1 & -c_{ii'}
\end{pmatrix}
\begin{pmatrix}
x_{j} \\ x_{j'}
\end{pmatrix}=
\begin{pmatrix}
c_{ii'} \cdot x_{j} - x_{j'} \\
x_{j} - c_{ii'} \cdot x_{j'}
\end{pmatrix},
\end{equation*}
and $\Phi_{2}^{\rho}(a,b)$ is defined as in \cref{gauss::weight}.
\end{enumerate}
\end{corollary}

\section{Special Orthogonal Group}\label{sec:SO(n)}
In \cref{sec::RIK}, we extensively utilize the special orthogonal group, which we introduce briefly. For more comprehensive details, please refer to \cite{dai2013approximation}.
The special orthogonal group, also called the rotation group, is defined by $SO(n)=\{R \in \bR^{n \times n}: RR^{\top}=R^{\top}R=I, \det R=1 \}$. It consists of all orthogonal matrices of determinant 1 and is a Lie group of dimension $n(n-1)/2$.

For $n \ge 2$, let $\theta \in \bS^{n-1}$ have the spherical coordinates $\phi_{1} \in [-\pi, \pi)$ and $\phi_{2}, \dots, \phi_{n-1} \in [0, \pi]$, i.e.
\begin{equation*}
    \theta=
    \begin{pmatrix}
    \sin \phi_{n-1} \dots \sin \phi_{3} \sin \phi_{2} \sin \phi_{1}  \\
    \sin \phi_{n-1} \dots \sin \phi_{3} \sin \phi_{2} \cos \phi_{1}  \\
    \sin \phi_{n-1} \dots \sin \phi_{3} \cos \phi_{2} \\
    \vdots \\
    \sin \phi_{n-1} \cos \phi_{n-2}\\
    \cos \phi_{n-1}
    \end{pmatrix}.
\end{equation*}
If $\tilde{\theta} \in \bS^{n-2}$ has the spherical coordinates $\phi_{1} \in [-\pi, \pi)$ and $\phi_{2}, \dots, \phi_{n-2} \in [0, \pi]$, then the spherical decomposition reads
\begin{equation}\label{sphe::red}
    \theta=[\sin \phi_{n-1} \tilde{\theta}: \cos \phi_{n-1}].
\end{equation}
For $l=n-1, \dots, 1$, let $R^{l}_{\phi}$ be the rotation in $x_{l}$-$x_{l+1}$ plane with angle $\phi$, i.e.
\begin{equation*}
    R^{l}_{\phi}:=
    \begin{pmatrix}
    \mI_{l-1} & \rvline & \mzero & \rvline & \mzero \\
    \hline
    \mzero & \rvline &
    \begin{matrix}
        \cos \phi & \sin \phi \\
        -\sin \phi & \cos \phi
    \end{matrix}
    & \rvline & \mzero \\
    \hline
    \mzero & \rvline & \mzero & \rvline & \mI_{n-l-1}
    \end{pmatrix} \in SO(n).    
\end{equation*}
The Euler rotation matrix for $\theta \in \bS^{n-1}$ is defined by
\begin{equation*}
    E(\theta)=E(\phi_{1},\dots,\phi_{n-1}):=
    R^{n-1}_{-\phi_{n-1}} \dots R^{2}_{-\phi_{2}} R^{1}_{-\phi_{1}} \in SO(n).
\end{equation*}
Note that
\begin{align*}
    E(\theta)^{\top}&=R^{1}_{\phi_{1}} R^{2}_{\phi_{2}} \dots R^{n-1}_{\phi_{n-1}}
    =\begin{pmatrix}
    E(\tilde{\theta})^{\top} & \rvline & \mzero \\
    \hline
    \mzero & \rvline & 1
    \end{pmatrix} \cdot R^{n-1}_{\phi_{n-1}}. 
\end{align*}

\begin{proposition}\label{Euler::repre} 
For any $\theta \in \bS^{n-1}$ and $R \in SO(n)$,
\begin{enumerate}
    \item We have $E(\theta)^{\top} e_{n}=\theta, E(\theta) \theta=e_{n}$. Also, there is a unique $\tilde{R}_{\theta} \in SO(n-1)$ such that
    \begin{equation*}
    E(\theta)=
    \begin{pmatrix}
    \tilde{J} \tilde{R}_{\theta} & \rvline & \mzero \\
    \hline
    \mzero & \rvline & -1
    \end{pmatrix} \cdot E(-\theta).
    \end{equation*}
    \item Let $\vect{r}=R^{\top} e_{n} \in \bS^{n-1}$. Then there is a unique $\tilde{R} \in SO(n-1)$ such that
    \begin{equation*}
        R=
        \begin{pmatrix}
    \tilde{R} & \rvline & \mzero \\
    \hline
    \mzero & \rvline & 1
    \end{pmatrix} \cdot E(\vect{r}).
    \end{equation*}
\end{enumerate}
\end{proposition}
\begin{proof}
Observe that in terms of the spherical coordinates,
\begin{equation*}
   E(\theta)^{\top}=R^{1}_{\phi_{1}} R^{2}_{\phi_{2}} \dots R^{n-1}_{\phi_{n-1}}: (0,\dots,0) \mapsto 
   (0,\dots,0,\phi_{n-1}) \mapsto \dots \mapsto
   (\phi_{1},\dots,\phi_{n-1}),
\end{equation*}
i.e. $E(\theta)^{\top} e_{n}=\theta$, and since $E(\theta)^{-1}=E(\theta)^{\top}$, we get $E(\theta) \theta=e_{n}$. 

Let 
\begin{equation*}
    \tilde{J}_{+}:=
        \begin{pmatrix}
    \tilde{J} & \rvline & \mzero \\
    \hline
    \mzero & \rvline & 1
    \end{pmatrix}.
\end{equation*}
Then, $J \tilde{J}_{+} E(\theta) \cdot E(-\theta)^{\top} e_{n}=- J \tilde{J}_{+} E(\theta) \theta= - J \tilde{J}_{+} e_{n} =e_{n}$, thus $J \tilde{J}_{+} E(\theta) \cdot E(-\theta)^{\top}$ is a rotation matrix that leaves $e_{n}$ fixed, so there is a unique corresponding $\tilde{R}_{\theta} \in SO(n-1)$ such that
\begin{equation*}
   J \tilde{J}_{+} E(\theta) \cdot E(-\theta)^{\top}=
   \begin{pmatrix}
   \tilde{R}_{\theta} & \rvline & \mzero \\
   \hline
   \mzero & \rvline & 1
   \end{pmatrix}
   \, \Longleftrightarrow \,
   E(\theta)=
    \begin{pmatrix}
    \tilde{J} \tilde{R}_{\theta} & \rvline & \mzero \\
    \hline
    \mzero & \rvline & -1
    \end{pmatrix} \cdot E(-\theta).
\end{equation*}
Given $R \in SO(n)$, note that $R \cdot E(\vect{r})^{\top} e_{n}=R \vect{r}=R R^{\top} e_{n}=e_{n}$,
$R \cdot E(\vect{r})^{\top}$ is a rotation matrix that leaves $e_{n}$ fixed, and there is a unique corresponding $\tilde{R} \in SO(n-1)$ such that
\begin{equation*}
   R \cdot E(\vect{r})^{\top}=
   \begin{pmatrix}
   \tilde{R} & \rvline & \mzero \\
   \hline
   \mzero & \rvline & 1
   \end{pmatrix}
   \quad \Longleftrightarrow \quad
   R=\begin{pmatrix}
   \tilde{R} & \rvline & \mzero \\
   \hline
   \mzero & \rvline & 1
   \end{pmatrix} \cdot
   E(\vect{r}).
\end{equation*}
\end{proof}

For $\tilde{R} \in SO(n-1)$, let $\tilde{R}_{+} \in SO(n)$ denote a unique extension that leaves $e_{n}$ fixed:
\begin{equation}\label{rot::exten}
    \tilde{R}_{+}:=
        \begin{pmatrix}
    \tilde{R} & \rvline & \mzero \\
    \hline
    \mzero & \rvline & 1
    \end{pmatrix}.
\end{equation}

Note that $(\tilde{R}_{+}) \tilde{} =\tilde{R}$, and under the equivalence relation on $SO(n)$ as in \eqref{rot::equiv}, $R \sim E(\vect{r})$ and $\tilde{R}_{+} \sim I$.
\cref{Euler::repre} indicates that the left action of $G=SO(n)$ on $\bS^{n-1}$ is transitive, i.e. for any $\theta_{1}, \theta_{2} \in \bS^{n-1}$, there is some $R \in SO(n)$ such that $\theta_{1}=R \theta_{2}$ (choose $R=E(\theta_{1})^{\top} E(\theta_{2})$). The following is straightforward:

\begin{corollary}\label{SO::stab}
For a fixed $\theta \in \bS^{n-1}$, the stabilizer subgroup of $G=SO(n)$ with respect to $\theta$ is given by
\begin{equation*}
   G_{\theta} := \left\{R \in SO(n) : R \theta = \theta \right\} = E(\theta)^{\top} G_{e_{n}} E(\theta), 
\end{equation*}
where $G_{e_{n}} = \{ \tilde{R}_{+} : \tilde{R} \in SO(n-1) \}$.
\end{corollary}

\begin{proposition}\label{SO::inte}
For any $\vect{z} \in \bR^{n}$ and a generic function $f$, we have
\begin{equation}\label{SO2sphe::inte}
    \int_{SO(n)} f(R \vect{z}) \rd R = \frac{1}{|\bS^{n-1}|} \int_{\bS^{n-1}} f(|\vect{z}| \theta) \rd \theta.
\end{equation}
Consequently, we obtain
\begin{equation}\label{SOR2R::inte}
    \int_{SO(n)} \int_{\bR^{n-1}} f(R^{\top} [\vect{x}:0]) \rd \vect{x} \rd R
    =\frac{|\bS^{n-2}|}{|\bS^{n-1}|} \int_{\bR^{n}}  |\vect{z}|^{-1} f(\vect{z}) \rd \vect{z},
\end{equation}
\begin{equation}\label{SOB2B::inte}
    \int_{SO(n)} \int_{\bB^{n-1}} f(R^{\top} [\vect{x}:0]) \rd \vect{x} \rd R
    =\frac{|\bS^{n-2}|}{|\bS^{n-1}|} \int_{\bB^{n}}  |\vect{z}|^{-1} f(\vect{z}) \rd \vect{z}.
\end{equation}
\end{proposition}

\begin{proof}
\eqref{SO2sphe::inte} is a direct consequence of (2.7) in p191 of \cite{natterer2001mathematics}. For \eqref{SOR2R::inte}, by \eqref{SO2sphe::inte} and the Fubini theorem,
\begin{align*}
    \int_{SO(n)} \int_{\bR^{n-1}} f(R^{\top} [\vect{x}:0]) \rd \vect{x} \rd R
    =& \frac{1}{|\bS^{n-1}|} \int_{\bS^{n-1}} \int_{\bR^{n-1}}  f(|\vect{x}| \theta) \rd \vect{x} \rd \theta.
\end{align*}
Let $\vect{x}=r \tilde{\theta}$ be the polar decomposition with $r>0 , \tilde{\theta} \in \bS^{n-2}$. Then
\begin{align*}
    \int_{SO(n)} \int_{\bR^{n-1}} f(R^{\top} [\vect{x}:0]) \rd \vect{x} \rd R
    =& \frac{1}{|\bS^{n-1}|} \int_{\bS^{n-1}} \int_{0}^{\infty} \int_{\bS^{n-2}}  r^{n-2} f(r \theta) \rd \tilde{\theta} \rd r \rd \theta \\
    =& \frac{|\bS^{n-2}|}{|\bS^{n-1}|} \int_{\bS^{n-1}} \int_{0}^{\infty}  r^{n-2} f(r \theta) \rd r \rd \theta 
    = \frac{|\bS^{n-2}|}{|\bS^{n-1}|} \int_{\bR^{n}}  |\vect{z}|^{-1} f(\vect{z}) \rd \vect{z}.
\end{align*}
Same goes for \eqref{SOB2B::inte} if we let $r \in (0, 1]$.
\end{proof}

\section{Spherical Harmonics}\label{sec::sphe::har}
Here, we only present a few results on the spherical harmonics, focusing on the RKHS viewpoint that is useful to develop our theory, especially in \cref{sec::kernel::repn}. For further details on the topic, we refer to \cite{dai2013approximation, dai2013spherical}. Let $\cP_{l}(\bR^n)$ be the set of homogeneous polynomials of degree $l$:
\begin{align*}
    &\cP_{l}(\bR^n)=\spann \{z_{1}^{l_{1}}\cdot \dots \cdot z_{n}^{l_{n}}: l_{1}, \dots, l_{n} \ge 0, l_{1}+\dots+l_{n}=l\}, \\
    &\dim \cP_{l}(\bR^n)= \multiset{n}{l}= \binom{n+l-1}{l}.
\end{align*}

Let $\cH_{l}(\bR^n) \subset \cP_{l}(\bR^n)$ be the set of homogeneous harmonic polynomials of degree $l$, i.e. $Y_{l} \in \cP_{l}(\bR^n)$ with $\Laplace Y_{l}=0$. Its restriction on a sphere $Y_{l} \in \cH_{l}(\bS^{n-1})$ is called a spherical harmonic of degree $l$. Any spherical polynomials $Y_{l}, Y_{m}$ of different degree $l \neq m$ is orthogonal, i.e. $\langle Y_{l}, Y_{m} \rangle_{\cL_{2}(\bS^{n-1})}=0$, which yields the following decompositions:
\begin{align*}
    \cP_{l}(\bR^n)&=\cH_{l}(\bR^n) \oplus r^{2} \cP_{l-2}(\bR^n) \\
    &=\cH_{l}(\bR^n) \oplus r^{2} \cH_{l-2}(\bR^n) \oplus r^{4} \cH_{l-4}(\bR^n) \oplus \dots, \\
    N(n,l):&=\dim \cH_{l}(\bS^{n-1})=\dim \cH_{l}(\bR^n)=\dim \cP_{l}(\bR^n)-\dim \cP_{l}(\bR^{n-2}) \\
    &=\multiset{n}{l}-\multiset{n}{l-2}=\frac{(2l+n-2)(n+l-3)!}{l! (n-2)!}, \quad N(n,0)=1.
\end{align*}

It is customary to denote by $\{Y_{lk}(\theta): k=1, \dots, N(n,l) \}$ an orthonormal basis (ONB) of $\cH_{l}(\bS^{n-1})$, equipped with the $\cL_{2}$ inner product. Since the Laplace operator $\Laplace$ is $O(n)$-invariant, for any $R \in SO(n)$,  $\{Y_{lk}(R^{\top} \theta): k=1,2,\dots, N(n,l) \}$ again forms an ONB of $\cH_{l}(\bS^{n-1})$. The relationship between these two ONBs can be represented by an orthogonal matrix, called the \textbf{Wigner D-matrix} of degree $l$: for $1 \le k \le N(n,l)$,
\begin{equation}\label{Wigner::repn}
    Y_{lk}(R^{\top}\theta)=\sum_{k^{'}=1}^{N(n,l)} [D_{l}(R)]_{kk^{'}} \cdot Y_{lk^{'}}(\theta).
\end{equation}
In other words, the Wigner D-matrix $D_{l}(R)$ of degree $l$ is the irreducible unitary representation of $\cH_{l}(\bS^{n-1})$. Note that the matrix depends on the choice of an ONB.

As $(\cH_{l}(\bS^{n-1}), \langle \cdot, \cdot \rangle_{\cL_{2}(\bS^{n-1})})$ is a finite-dimensional subspace, it directly follows that the reproducing kernel is given by
\begin{equation*}
    Z_{l}(\theta_{1}, \theta_{2}) = \sum_{k=1}^{N(n,l)} Y_{lk}(\theta_{1}) Y_{lk}(\theta_{2}),
\end{equation*}
which is independent of the choice of ONB. For any spherical harmonic $Y_{l} \in \cH_{l}(\bS^{n-1})$ of degree $l$, the reproducing property reads
\begin{equation}\label{sphe::har::rk}
    \int_{\bS^{n-1}} Z_{l}(\theta_{1}, \theta_{2}) Y_{l}(\theta_{2}) \rd \theta_{2} = Y_{l}(\theta_{1}).
\end{equation}

To achieve an explicit formula for the kernel $Z_{l}$, we introduce a special type of polynomials. Recall that $w(s)=\sqrt{1-s^2}, \, |s| \le 1$. For any $\lambda >-1/2$, the Gegenbauer polynomials $C_{l}^{\lambda}$ of degree $l$ are orthogonal polynomials in $\cL_{2}(w^{2 \lambda-1}):=\cL_{2}([-1,+1], w^{2 \lambda-1})$, the space of functions on $[-1, +1]$ with the weight function $w^{2\lambda-1}$. Here, we \textbf{normalize} by setting $C_{l}^{\lambda}(1)=1$ as in \cite{natterer2001mathematics}. With this scaling, for any $l, k \ge 0$, we have
\begin{align*}
    \langle C_{l}^{\lambda}, C_{k}^{\lambda} \rangle_{\cL_{2}(w^{2 \lambda-1})} 
    = \int_{-1}^{+1} w^{2\lambda-1}(s) C_{l}^{\lambda}(s) C_{k}^{\lambda}(s) \rd s 
    = \frac{2^{2\lambda-1}(\Gamma(\lambda+1/2))^{2} l!}{(l+\lambda) \Gamma(l+2\lambda)} \delta_{lk}
    =:p^{\lambda}_{l} \delta_{lk}.
\end{align*}
By induction, we verify that
\begin{equation*}
    (\Gamma((n-1)/2))^{2} = \frac{2^{-n/2+1} |\bS^{n-1}|}{|\bS^{n-2}|} \Gamma(n-1),
\end{equation*}
hence, when $\lambda=n/2-1$, the normalizing constant can be simplified into
\begin{align}\label{Gegen::norm}
    p^{n/2-1}_{l} 
    = \frac{2^{n-3}(\Gamma((n-1)/2))^{2} l!}{(l+n/2-1) \Gamma(l+n-2)} 
    = \frac{2^{n/2-1} |\bS^{n-1}|}{|\bS^{n-2}|} \frac{l! (n-2)!}{(2l+n-2) (l+n-3)!} 
    = \frac{2^{n/2-1} |\bS^{n-1}|}{N(n,l) |\bS^{n-2}|}.
\end{align}
The reproducing kernel $Z_{l}$ for $\cH_{l}(\bS^{n-1})$ can be expressed in terms of a Gegenbauer polynomial \cite{dai2013approximation}:
\begin{equation}\label{rk::Gegen}
    Z_{l}(\theta_{1}, \theta_{2}) 
    = \frac{1}{p^{n/2-1}_{l} |\bS^{n-2}|} C_{l}^{n/2-1}(\theta_{1}^{\top} \theta_{2}).
\end{equation}
For fixed $\theta_{1}, \theta_{2} \in \bS^{n-1}$, the reproducing property reads
\begin{align}\label{sphe::gene}
    C_{l}^{n/2-1}(\theta_{1}^{\top} \theta_{2}) 
    =& p^{n/2-1}_{l} |\bS^{n-2}| Z_{l}(\theta_{1}, \theta_{2}) \\ \nonumber 
    =& p^{n/2-1}_{l} |\bS^{n-2}| \int_{\bS^{n-1}} Z_{l}(\theta_{1}, \theta_{3}) Z_{l}(\theta_{2}, \theta_{3}) \rd \theta_{3} \\ \nonumber
    =& \frac{1}{p^{n/2-1}_{l} |\bS^{n-2}|} \int_{\bS^{n-1}} C_{l}^{n/2-1}(\theta_{1}^{\top} \theta_{3}) C_{l}^{n/2-1}(\theta_{2}^{\top} \theta_{3}) \rd \theta_{3},    
\end{align}
and the Funk-Hecke formula follows:
\begin{theorem}[Funk-Hecke]\label{Funk::Hecke}
For any $l \ge 0, h \in \cL_{2}(w^{n-3})$ and $\theta_{1} \in \bS^{n-1}$,
\begin{equation*}
    \int_{\bS^{n-1}} h(\theta_{1}^{\top} \theta_{2}) Y_{l}(\theta_{2}) \rd \theta_{2}
    =|\bS^{n-2}| \langle h, C_{l}^{n/2-1} \rangle_{\cL_{2}(w^{n-3})} Y_{l}(\theta_{1}).
\end{equation*}
\end{theorem}

\begin{remark}
When $n=2$, $Y_{0}\equiv 1/\sqrt{2 \pi}$ and for $l \in \bN$,
there are $N(2,l)=2$ linearly independent spherical harmonics:
\begin{equation*}
    Y_{l, 1}(\theta)=\frac{\cos l \phi}{\sqrt{\pi}}, \quad Y_{l, 2}(\theta)=\frac{\sin l \phi}{\sqrt{\pi}}, \quad \theta=
    \begin{pmatrix}
    \cos \phi \\ \sin \phi
    \end{pmatrix}.
\end{equation*}
As $\lambda=(n-2)/2=0$ in this case, the corresponding Gegenbauer polynomials become the Chebyshev polynomials of the first kind:
\begin{equation*}
    T_{l}(x)=C_{l}^{0}(x)=\cos(l \arccos x), \quad |x| \le 1. 
\end{equation*}
If $n=3$, then $N(3,l)=2l+1$ for $l>0$, and the corresponding Gegenbauer polynomials become the Legendre polynomials $P_{l}=C_{l}^{1/2}$.
\end{remark}

\section{Additional experimental results}\label{sec:add:exper}
As in \cref{sec::simul}, we have employed a noiseless sinogram acquired from 40 different angles over $[0, \pi)$ (rad). This sinogram was generated using a 10-times intensified, $M \times M$ ($M=200$) pixelated 2D Shepp-Logan phantom. The follwoing results illustrate the influence of two key parameters on reconstruction: \cref{RKHS::recons::wide::ker::img} shows that the tuning parameter $\gamma$ affects the resolution, and \cref{RKHS::recons::no::pen::img} demonstrates that the penalty parameter $\nu > 0$ affects the stability.

\begin{figure}[h]
\centering
\includegraphics[width=0.8\textwidth, height=7cm]
{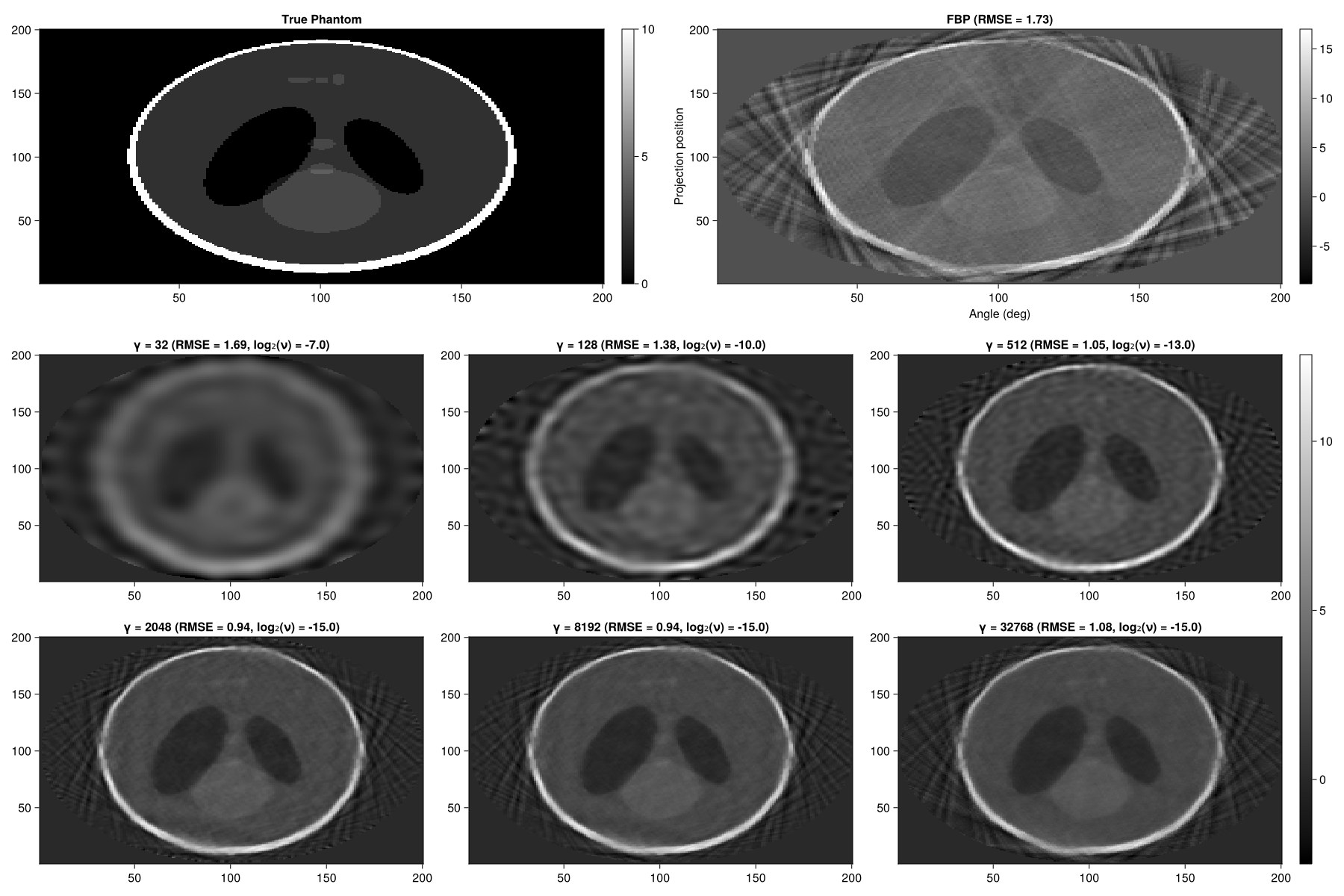}
\caption{Reconstructions from an ideal sinogram with randomly chosen $N=40$ angles over $[-0, \pi)$ (rad), generated by the $M \times M$ ($M=200$) pixelated 2D Shepp-Logan phantom (top left). FBP algorithm is depicted on top right, while kernel reconstructions with six values of $\gamma > 0$, each with optimal values $\nu^{*}$, are depicted at the bottom.}
\label{RKHS::recons::wide::ker::img}
\end{figure}

\begin{figure}[h]
\centering
\includegraphics[width=0.8\textwidth, height=4cm]{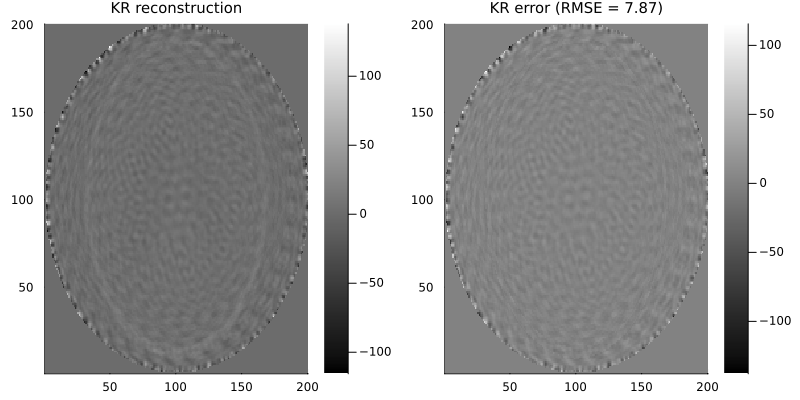}
\caption{Kernel reconstruction of an ideal sinogram (top right) at $N=40$ equidistant angles over $[-0, \pi)$ (rad), generated by the $M \times M$ ($M=200$) pixelated 2D Shepp-Logan phantom (top left). The Gaussian kernel $\gamma = 2000$ in \cref{Xray::MLE} was used for the kernel MLE reconstruction (bottom left) without the penalty.}
\label{RKHS::recons::no::pen::img}
\end{figure}

\section{Proofs of theorems}
\subsection{Proof in Section \ref{sec:prelim}}\label{sec::proof::convention}
\begin{proof}[Proof of \cref{L2::conti}]
For $f \in \cL_{2}(\bR^{n})$, by the Cauchy-Schwarz inequality,
\begin{align*}
    \|\sP_{R} f\|^{2}_{\cL_{2}}
    = &\int_{\bB^{n-1}} \lvert \sP_{R} f(\vect{x}) \rvert^{2} \rd \vect{x} 
    = \int_{\bB^{n-1}} \bigg\lvert \int_{I_{\|\vect{x}\|}} f(R^{\top} [\vect{x}:z]) \rd z \bigg\rvert^{2} \rd \vect{x} \\
    \le & \int_{\bB^{n-1}} 2 W(\vect{x}) \int_{I_{\|\vect{x}\|}} \lvert f(R^{\top} [\vect{x}:z]) \rvert^{2} \rd z \rd \vect{x} \\
    \le & 2 \int_{\bB^{n-1}} \int_{I_{\|\vect{x}\|}} \lvert f(R^{\top} [\vect{x}:z]) \rvert^{2} \rd z \rd \vect{x} 
    = 2 \int_{\bB^{n}} \lvert f(\vect{z}) \rvert^{2} \rd \vect{z}=2 \|f\|^{2}_{\cL_{2}},
\end{align*}
where we have put $\vect{z}=R^{\top} [\vect{x}:z]$ so that $\rd \vect{z} =\rd \vect{x} \rd z$. Similarly,
\begin{align*}
    \|\sP f\|^{2}_{\cL_{2}}
    = \int_{SO(n)} \int_{\bB^{n-1}} \lvert \sP f(R, \vect{x}) \rvert^{2} \rd \vect{x} \rd R 
    \le \int_{SO(n)} 2 \|f\|^{2}_{\cL_{2}} \rd R 
    = 2 \|f\|^{2}_{\cL_{2}}.
\end{align*}
To calculate the adjoint operator, let $f \in \cL_{2}(\bB^{n})$ and $g \in \cL_{2}(\bB^{n-1})$. Then
\begin{align*}
   \langle f, \sP_{R}^{*} g \rangle_{\cL_{2}}
   =&\langle \sP_{R} f, g \rangle_{\cL_{2}}
   =\int_{\bB^{n-1}} \sP_{R} f(\vect{x}) g(\vect{x}) \rd \vect{x} \\
   =& \int_{\bB^{n-1}} \int_{I_{\|\vect{x}\|}} f(R^{\top} [\vect{x}:z]) g(\vect{x}) \rd z \rd \vect{x}
   =\int_{\bB^{n}} f(\vect{z}) g(\mP_{R}(\vect{z})) \rd \vect{z},
\end{align*}
where we have put $\vect{z}=R^{\top} [\vect{x}:z]$ so that $\vect{x}=\mP_{R}(\vect{z})$.
Similarly, for $g \in \cL_{2}(SO(n) \times \bB^{n-1})$,
\begin{align*}
   \langle f, \sP^{*} g \rangle_{\cL_{2}}
   =\langle \sP_{R} f, g \rangle_{\cL_{2}}
   =&\int_{SO(n)} \int_{\bB^{n-1}} \sP_{R} f(\vect{x}) g(R, \vect{x}) \rd \vect{x} \rd R\\
   =&\int_{SO(n)} \int_{\bB^{n-1}} \int_{I_{\|\vect{x}\|}} f(R^{\top} [\vect{x}:z]) g(R, \vect{x}) \rd z \rd \vect{x} \rd R \\
   =&\int_{\bB^{n}} f(\vect{z}) \left( \int_{SO(n)} g(R, \mP_{R}(\vect{z})) \rd R \right) \rd \vect{z}.
\end{align*}
\end{proof}

\begin{proof}[Proof of \cref{Fou::slice}]
For any $\eta \in \bR^{n-1}$,
\begin{align*}
    \widehat{\sP_{R} f}(\eta) 
    =& (2 \pi)^{-(n-1)/2} \int_{\bR^{n-1}} e^{-\imath \vect{x}^{\top} \eta} \sP_{R} f(\vect{x}) \rd \vect{x} \\    
    =& (2 \pi)^{-(n-1)/2} \int_{\bR^{n-1}} e^{-\imath \vect{x}^{\top} \eta} \int_{-\infty}^{+\infty} f(R^{\top} [\vect{x}:z]) \rd z \rd \vect{x} \\    
    =& (2 \pi)^{-(n-1)/2} \int_{\bR^{n}} e^{-\imath \vect{z}^{\top} R^{\top}[\eta:0]} f(\vect{z}) \rd \vect{z}
    = \sqrt{2 \pi} \hat{f}(R^{\top} [\eta:0]), 
\end{align*}
where we have put $\vect{z}=R^{\top} [\vect{x}:z]$. 
\end{proof}

\begin{proof}[Proof of \cref{inv::for}]
By the Fourier inversion formula and \cref{SO::inte},
\begin{align*}
    \sI^{\alpha} f(\vect{z})
    =& (2 \pi)^{-n/2} \int_{\bR^{n}} e^{\imath \vect{z}^{\top} \xi} |\xi|^{-\alpha} \hat{f}(\xi) \rd \xi \\
    =& (2 \pi)^{-n/2} \frac{|\bS^{n-1}|}{|\bS^{n-2}|} \int_{SO(n)} \int_{\bR^{n-1}} e^{\imath \vect{z}^{\top} R^{\top} [\eta:0]} |\eta|^{1-\alpha} \hat{f}(R^{\top} [\eta:0]) \rd \eta \rd R.
\end{align*}
Note that $\vect{z}^{\top} R^{\top} [\eta:0]=(R \vect{z})^{\top} [\eta:0]= \mP_{R}(\vect{z})^{\top} \eta$. Thus, by \cref{SO::inte,Fou::slice,L2::conti},
\begin{align*}
    \sI^{\alpha} f(\vect{z})
    =& (2 \pi)^{-(n+1)/2} \frac{|\bS^{n-1}|}{|\bS^{n-2}|} \int_{SO(n)} \int_{\bR^{n-1}} e^{\imath \mP_{R}(\vect{z})^{\top} \eta} |\eta|^{1-\alpha} \widehat{\sP_{R} f}(\eta) \rd \eta \rd R \\
    =& \frac{|\bS^{n-1}|}{2 \pi |\bS^{n-2}|} \int_{SO(n)} |\eta|^{1-\alpha} (\widehat{\sP_{R} f}) \check{} (\mP_{R}(\vect{z})) \rd R \\
    =& \frac{|\bS^{n-1}|}{2 \pi |\bS^{n-2}|} \int_{SO(n)} \sI^{\alpha-1} \sP f(R, \mP_{R}(\vect{z})) \rd R 
    = \frac{|\bS^{n-1}|}{2 \pi |\bS^{n-2}|} \sP^{*} \sI^{\alpha-1} \sP f (\vect{z}).
\end{align*}
\end{proof}

\subsection{Proofs in Section \ref{sec::RKHS}}\label{sec:proof:RKHS}
\begin{proof}[Proof of \cref{push::kernel}]
Given distinct points $\vect{x}_{i} \in \bB^{n-1}$ and given $\alpha_{i} \in \bC$ for $1 \le i \le M$,
\begin{align*}
    \sum_{i,j=1}^{M} \overline{\alpha_{i}} \alpha_{j} &\tilde{K}^{R}(\vect{x}_{i},\vect{x}_{j})
    = \sum_{i,j=1}^{M} \overline{\alpha_{i}} \alpha_{j} \int_{I_{\|\vect{x}_{j}\|}} \int_{I_{\|\vect{x}_{i}\|}} K(R^{\top} [\vect{x}_{i}:z_{1}], R^{\top} [\vect{x}_{j}:z_{2}]) \rd z_{1} \rd z_{2} \\
    =&\int_{-1}^{+1} \int_{-1}^{+1} \left( \sum_{i,j=1}^{M} \overline{\alpha_{i}} \alpha_{j} K(R^{\top} [\vect{x}_{i}:z_{1}], R^{\top} [\vect{x}_{j}:z_{2}]) \right) \rd z_{1} \rd z_{2} \ge 0,
\end{align*}
since the integrand is always nonnegative, thus $\tilde{K}^{R}$ is also a kernel. Also, equality holds if and only if the integrand is always 0. Thus, if $\vect{x}_{1}, \vect{x}_{2}, \dots, \vect{x}_{M} \in (\bB^{n-1})^{\circ}$ and $\tilde{K}^{R}$ is strictly p.d., then equality holds if and only if $\alpha_{1}=\dots=\alpha_{M}=0$. This entails that the restriction of $\tilde{K}^{R}$ to $(\bB^{n-1})^{\circ}$ is strictly p.d, provided that $K$ is strictly p.d.

Let $K: \bB^{n} \times \bB^{n} \rightarrow \bR$ be continuous. Then, it is uniformly continuous due to the compactness of $\bB^{n}$, so for any $\varepsilon>0$, there is some $\delta>0$ such that $\|(\vect{z}_{1}, \vect{z}_{2})-(\vect{z}_{1}^{'}, \vect{z}_{2}^{'})\|<\delta$ implies $|K(\vect{z}_{1}, \vect{z}_{2})-K(\vect{z}_{1}^{'}, \vect{z}_{2}^{'})|<\varepsilon$. Then, whenever $\|(\vect{x}_{1}, \vect{x}_{2})-(\vect{x}_{1}^{'}, \vect{x}_{2}^{'})\|<\delta$, we obtain 
\begin{align*}
    &\left|\tilde{K}^{R}(\vect{x}_{1}, \vect{x}_{2})-\tilde{K}^{R}(\vect{x}_{1}^{'}, \vect{x}_{2}^{'})\right| \\
    &\le \int_{-1}^{+1} \int_{-1}^{+1} \left| K(R^{\top} [\vect{x}_{1}:z_{1}], R^{\top} [\vect{x}_{2}:z_{2}])-K(R^{\top} [\vect{x}_{1}^{'}:z_{1}], R^{\top} [\vect{x}_{2}^{'}:z_{2}]) \right| \rd z_{1} \rd z_{2} 
    \le 4\varepsilon,
\end{align*}
hence $\tilde{K}^{R}$ is also uniformly continuous.
\end{proof}

To prove our central theorem (\cref{rkhs::image::re}) that constitutes our framework, we use the following lemma, of which the proof can be found in Theorem 3.11 of \cite{paulsen2016introduction}.

\begin{lemma}\label{rkhs::tfae}
Let $\bH(K)$ be the RKHS on $E$ with a reproducing kernel $K$. Then $f \in \bH(K)$ if and only if there exists a constant $c \ge 0$ such that $c^{2} K(x,y)-f(x)\overline{f(y)}$ is a kernel function. Moreover, if $f \in \bH(K)$, then $\|f\|_{\bH(K)}$ is the least $c$ that satisfies the statement above.
\end{lemma}

\begin{theorem}\label{rkhs::image} For any $R \in SO(n)$,
\begin{enumerate}
    \item $\sP_{R}: \bH(K) \rightarrow \bH(\tilde{K}^{R})$ is a contractive, surjective, linear map.
    \item For $g \in \bH(\tilde{K}^{R})$, we have $\|g\|_{\bH(\tilde{K}^{R})}= \min \left\{\|f\|_{\bH(K)} : g=\sP_{R}f \right\}$.
\end{enumerate} 
\end{theorem}

\begin{proof}
First, we show that if $f \in \bH(K)$, then $\sP_{R}f \in \bH(\tilde{K}^{R})$. Let $\|f\|_{\bH(K)}=c$. By \cref{rkhs::tfae}, $\hat{K}(\vect{z}_{1},\vect{z}_{2}):=c^{2} K(\vect{z}_{1},\vect{z}_{2})-f(\vect{z}_{1})\overline{f(\vect{z}_{2})}$ is a kernel function on $\bB^{n}$. Note that for any $\vect{x}_{1}, \vect{x}_{2}, \dots, \vect{x}_{M} \in \bB^{n-1}$ and $\alpha_{1}, \alpha_{2}, \dots, \alpha_{M} \in \bC$,
\begin{align*}
    &\sum_{i,j=1}^{M} \overline{\alpha_{i}} \alpha_{j} \left(c^{2} \tilde{K}^{R}(\vect{x}_{1},\vect{x}_{2})-\sP_{R}f(\vect{x}_{1}) \overline{\sP_{R}f(\vect{x}_{2})} \right) \\
    =&\sum_{i,j=1}^{M} \overline{\alpha_{i}} \alpha_{j} \int_{I_{\|\vect{x}_{j}\|}} \int_{I_{\|\vect{x}_{i}\|}} \hat{K}(R^{\top} [\vect{x}_{i}:z_{1}], R^{\top} [\vect{x}_{j}:z_{2}]) \rd z_{1} \rd z_{2} \\
    =&\int_{-1}^{+1} \int_{-1}^{+1} \left( \sum_{i,j=1}^{M} \overline{\alpha_{i}} \alpha_{j} \hat{K}(R^{\top} [\vect{x}_{i}:z_{1}], R^{\top} [\vect{x}_{j}:z_{2}]) \right) \rd z_{1} \rd z_{2} \ge 0,
\end{align*}
since the integrand is always non-negative. Hence, $\hat{K}^{R} (\vect{x}_{1},\vect{x}_{2}):=c^{2} \tilde{K}^{R}(\vect{x}_{1},\vect{x}_{2})-\sP_{R}f(\vect{x}_{1}) \overline{\sP_{R}f(\vect{x}_{2})}$ is a kernel on $\bB^{n-1}$, and by \cref{rkhs::tfae} again, we get $\sP_{R}f \in \bH(\tilde{K}^{R})$ with $\|\sP_{R}f\|_{\bH(\tilde{K}^{R})} \le \|f\|_{\bH(K)}$, i.e. $\sP_{R}: \bH(K) \rightarrow \bH(\tilde{K}^{R})$ is a contractive linear map.

To establish surjectivity, let $\bV(\tilde{K}^{R})$ denote the vector space of functions spanned by $\{\tilde{k}^{R}_{\vect{x}}: \vect{x} \in \bB^{n-1} \}$. 
For any $\vect{x}_{1}, \dots, \vect{x}_{M} \in \bB^{n-1}$ and $\alpha_{1}, \dots, \alpha_{M} \in \bF$ $(=\bR \text{ or } \bC)$,
\begin{align*}
    \left\|\sum_{i=1}^{M} \alpha_{i} \tilde{k}^{R}_{\vect{x}_{i}} \right\|^{2}_{\bH(\tilde{K}^{R})}
    =&\sum_{i,j=1}^{M} \overline{\alpha_{i}} \alpha_{j} \int_{I_{\|\vect{x}_{j}\|}} \int_{I_{\|\vect{x}_{i}\|}} K(R^{\top} [\vect{x}_{i}:z_{1}], R^{\top} [\vect{x}_{j}:z_{2}]) \rd z_{1} \rd z_{2} \\
    =&\int_{-1}^{+1} \int_{-1}^{+1} \sum_{i,j=1}^{M} \overline{\alpha_{i}} \alpha_{j} \langle k_{R^{\top} [\vect{x}_{i}:z_{1}]}, k_{R^{\top} [\vect{x}_{j}:z_{2}]} \rangle_{\bH(K)} \rd z_{1} \rd z_{2} \\
    =&\left\| \sum_{i=1}^{M} \alpha_{i} \int_{I_{\|\vect{x}_{i}\|}} k_{R^{\top} [\vect{x}_{i}:z]} \rd z \right\|^{2}_{\bH(K)},
\end{align*}
in the sense of the Bochner integration. Therefore, there is a well-defined isometry (injective, but not necessarily surjective) $\mG_{R}: \bV(\tilde{K}^{R}) \rightarrow \bH(K)$ given by
\begin{equation*}
    \mG_{R} \left(\sum_{i=1}^{n} \alpha_{i} \tilde{k}^{R}_{\vect{x}_{i}} \right)
    :=\sum_{i=1}^{n} \alpha_{i} \int_{I_{\|\vect{x}_{i}\|}} k_{R^{\top} [\vect{x}_{i}:z]} \rd z,
\end{equation*}
and it induces a unique isometric extension $\mG_{R}: \bH(\tilde{K}^{R}) \rightarrow \bH(K)$ since $\bV(\tilde{K}^{R})$ forms a dense subspace of $\bH(\tilde{K}^{R})$. Additionally, for any induced generator $\tilde{k}^{R}_{\vect{x}_{1}}$ at $\vect{x}_{1} \in \bB^{n-1}$ and any point $\vect{x}_{2} \in \bB^{n-1}$, 
\begin{align*}
    (\sP_{R} \circ \mG_{R})(\tilde{k}^{R}_{\vect{x}_{1}})(\vect{x}_{2})
    =&\int_{I_{\|\vect{x}_{2}\|}} \mG_{R}(\tilde{k}^{R}_{\vect{x}_{1}})(R^{\top} [\vect{x}_{2}:z_{2}]) \rd z_{2} \\
    =&\int_{I_{\|\vect{x}_{2}\|}} \int_{I_{\|\vect{x}_{1}\|}} k_{R^{\top} [\vect{x}_{1}:z_{1}]} (R^{\top} [\vect{x}_{2}:z_{2}]) \rd z_{1} \rd z_{2} \\
    =&\int_{I_{\|\vect{x}_{2}\|}} \int_{I_{\|\vect{x}_{1}\|}} K(R^{\top} [\vect{x}_{2}:z_{2}], R^{\top} [\vect{x}_{1}:z_{1}]) \rd z_{1} \rd z_{2}
    =\tilde{K}^{R} (\vect{x}_{2}, \vect{x}_{1})= \tilde{k}^{R}_{\vect{x}_{1}}(\vect{x}_{2}),
\end{align*}
because the convergence in RKHS implies pointwise convergence.
Since $\bV(\tilde{K}^{R})$ is dense in $\bH(\tilde{K}^{R})$, the above equation shows that $\sP_{R} \circ \mG_{R}$ is the identity map on $\bH(\tilde{K}^{R})$. 
Therefore, for any $g \in \bH(\tilde{K}^{R})$, the function $f=\mG_{R}(g) \in \bH(K)$ satisfies $g=\sP_{R} f$, which shows the surjectivity of $\sP_{R}: \bH(K) \rightarrow \bH(\tilde{K}^{R})$. Finally, we obtain $\|g\|_{\bH(\tilde{K}^{R})}= \|\mG_{R}(g)\|_{\bH(K)}= \min \left\{\|f\|_{\bH(K)} : g=\sP_{R}f \right\}$.
\end{proof}

As can be seen in \cref{rkhs::image}, $\sP_{R}: \bH(K) \rightarrow \bH(\tilde{K}^{R})$ is contractive, meaning its operator norm is at most 1. This naturally prompts the question of its adjoint operator. Interestingly, the isometry $\mG_{R}: \bH(\tilde{K}^{R}) \rightarrow \bH(K)$ serves as the adjoint operator of $\sP_{R}$ for any $R \in SO(n)$. Since the adjoint operator $\mG_{R}$ is isometric, its range forms a closed subspace of $\bH(K)$, and thus itself becomes a Hilbert space.

\begin{proposition}\label{proj::decomp}
Let $R \in SO(n)$ and $\bH_{R}(K):=\{\mG_{R} g: g \in \bH(\tilde{K}^{R}) \}$.
\begin{enumerate}
\item Any $f \in \bH(K)$ can be uniquely expressed as a direct sum:
\begin{equation*}
    f=(\mG_{R} \circ \sP_{R})f \oplus (f-(\mG_{R} \circ \sP_{R})f) \in \bH_{R}(K) \oplus (\bH_{R}(K))^{\perp},
\end{equation*}
thus $f \in \bH_{R}(K)$ if and only if $f=(\mG_{R} \circ \sP_{R})f$.
\item $(\bH_{R}(K))^{\perp}=\cN(\sP_{R}),\ \bH_{R}(K)=(\cN(\sP_{R}))^{\perp}$.
\end{enumerate}
\end{proposition}

\begin{proof}
\quad
\begin{enumerate}
\item It suffices to show that for any $f \in \bH(K)$, $g \in \bH(\tilde{K}^{R})$, $\langle \mG_{R}g, (f-(\mG_{R} \circ \sP_{R})f) \rangle_{\bH(K)}=0$. Given $t \in \bR$, let $h=\mG_{R}g+t(f-(\mG_{R} \circ \sP_{R})f) \in \bH(K)$. Note that $\sP_{R}h=g$. Since $\sP_{R}: \bH(K) \rightarrow \bH(\tilde{K}^{R})$ is contractive and $\mG_{R}: \bH(\tilde{K}^{R}) \rightarrow \bH(K)$ is isometric, we have for any $t \in \bR$,
\begin{align*}
    0 &\le \|h\|^{2}_{\bH(K)}-\|\sP_{R}h\|^{2}_{\bH(\tilde{K}^{R})} \\
    &=2t \langle \mG_{R}g, (f-(\mG_{R} \circ \sP_{R})f) \rangle_{\bH(K)} +t^{2} \| (f-(\mG_{R} \circ \sP_{R})f) \|^{2}_{\bH(K)},
\end{align*}
which reveals that $\langle \mG_{R}g, (f-(\mG_{R} \circ \sP_{R})f) \rangle_{\bH(K)}=0$.
\item By (1) and the fact that $\sP_{R} \circ \mG_{R}$ is the identity map on $\bH(\tilde{K}^{R})$,
\begin{equation*}
    f \in \cN(\sP_{R}) \Longleftrightarrow \sP_{R}f=0 \Longleftrightarrow (\mG_{R} \circ \sP_{R})f=0 \Longleftrightarrow f \in (\bH_{R}(K))^{\perp}.
\end{equation*}
Hence, $(\bH_{R}(K))^{\perp}=\cN(\sP_{R})$. Also, we get $\bH_{R}(K)=(\cN(\sP_{R}))^{\perp}$ as $\bH_{R}(K)$ is closed.
\end{enumerate}    
\end{proof}

\begin{corollary}\label{gamma::adj}
For any $R \in SO(n)$, $\sP_{R}^{*}=\mG_{R}: \bH(\tilde{K}^{R}) \rightarrow \bH(K)$. The adjoint operator is uniquely determined by
\begin{equation*}
    \sP_{R}^{*} ( \tilde{k}^{R}_{\vect{x}} )
    = \int_{I_{\|\vect{x}\|}} k_{R^{\top} [\vect{x}:z]} \rd z, \quad \vect{x} \in \bB^{n-1},
\end{equation*}
or equivalently, $\sP_{R}^{*} ( \tilde{k}^{R}_{\vect{x}} )(\vect{z})=\int_{I_{\|\vect{x}\|}} K(\vect{z}, R^{\top} [\vect{x}:z]) \rd z$ for any $\vect{x} \in \bB^{n-1}, \vect{z} \in \bB^{n}$.
\end{corollary}

\begin{proof}
Given $f \in \bH(K)$, there are unique $g_{1}\in \bH(\tilde{K}^{R})$ and $f_{1} \in (\bH_{R}(K))^{\perp}$ such that $f=\mG_{R} g_{1} \oplus f_{1} \in \bH_{R}(K) \oplus (\bH_{R}(K))^{\perp}$.
Then for any $g \in \bH(\tilde{K}^{R})$, by the proof of \cref{rkhs::image} and \cref{proj::decomp},
\begin{align*}
    \langle f, \sP_{R}^{*} g \rangle_{\bH(K)}=\langle \sP_{R}f, g \rangle_{\bH(\tilde{K}^{R})}
    =\langle g_{1}, g \rangle_{\bH(\tilde{K}^{R})}=\langle \mG_{R} g_{1}, \mG_{R} g \rangle_{\bH(K)}
    =\langle f, \mG_{R} g \rangle_{\bH(K)}.
\end{align*}
Since $f \in \bH(K)$ and $g \in \bH(\tilde{K}^{R})$ were chosen arbitrarily, we have $\sP_{R}^{*}=\mG_{R}: \bH(\tilde{K}^{R}) \rightarrow \bH(K)$.
\end{proof}

\cref{rkhs::image}, \cref{proj::decomp}, and \cref{gamma::adj} collectively prove the \cref{rkhs::image::re}.

\begin{proof}[Proof of \cref{repre::proj}]
\begin{enumerate}
\item $\langle \sP_{R}f, k^{R}_{\vect{x}_{k}} \rangle_{\bH(\tilde{K}^{R})}=\sP_{R}f(\vect{x}_{k})=0$ for all $k=1, \dots, M$ if and only if $\sP_{R}f \in (\bV_{\mF}(\tilde{K}^{R}))^{\perp}$, which is equivalent to $f \in \sP^{-1}_{R}((\bV_{\mF}(\tilde{K}^{R}))^{\perp})$.
\item First, note that $(\bV_{\mF}(\tilde{K}^{R}))^{\perp}$ is closed, thus so is $\sP^{-1}_{R}((\bV_{\mF}(\tilde{K}^{R}))^{\perp})$. Now, for any $f \in \bH(K)$, let $f=\sP_{R}^{*} g+f^{'} \in \bH_{R}(K) \oplus \cN(\sP_{R})$ be its decomposition as in \cref{rkhs::image::re}. Then, $f \in \sP^{-1}_{R}((\bV_{\mF}(\tilde{K}^{R}))^{\perp})$ if and only if $g= \sP_{R}f \in (\bV_{\mF}(\tilde{K}^{R}))^{\perp}$. 
\item From (2) and \cref{rkhs::image::re},
\begin{align*}
    (\sP^{-1}_{R}((\bV_{\mF}(\tilde{K}^{R}))^{\perp}))^{\perp}
    =& [\sP_{R}^{*}((\bV_{\mF}(\tilde{K}^{R}))^{\perp}) \oplus \cN(\sP_{R})]^{\perp} \\
    =& [\sP_{R}^{*}((\bV_{\mF}(\tilde{K}^{R}))^{\perp})]^{\perp} \cap \bH_{R}(K) \\
    =& [\sP_{R}^{*}((\bV_{\mF}(\tilde{K}^{R}))^{\perp})]^{\perp} \cap \sP_{R}^{*}(\bH(K)) = \sP_{R}^{*}(\bV_{\mF}(\tilde{K}^{R})),
\end{align*}
where we have used the fact that $\sP_{R}^{*}: \bH(\tilde{K}^{R}) \rightarrow \bH(K)$ is an isometry for the last equality. Indeed, $\sP_{R}^{*}(\bV_{\mF}(\tilde{K}^{R}))$ is a closed subspace of $\bH(K)$.
\end{enumerate}
\end{proof}

\begin{proof}[Proof of \cref{iso::full::rk}]
The first part of the statement and \eqref{iso::proj} are straightforward from \eqref{eval::proj}. Let us show that if $K: \bB^{n} \times \bB^{n} \rightarrow \bR$ is strictly p.d. and $\mF= \{\vect{x}_{1}, \dots, \vect{x}_{M} \} \subset (\bB^{n-1})^{\circ}$, then $\mW \in \bR^{NM \times NM}$ is also strictly p.d.. Suppose $\ma^{\top} \mW \ma=0$. By \eqref{weight::equiv},
\begin{align*}
    0&=\ma^{\top} \mW \ma
    = \sum_{i,i'=1}^{N} \sum_{j,j'=1}^{M} \alpha_{ij} \alpha_{i'j'} w_{ij, i'j'} \\
    &=\int_{-1}^{+1} \int_{-1}^{+1} \left( \sum_{i,i'=1}^{N} \sum_{j,j'=1}^{M} \alpha_{ij} \alpha_{i'j'} K(R_{i}^{\top} [\vect{x}_{j}:z_{1}], R_{i'}^{\top}  [\vect{x}_{j'}:z_{2}]) \right) \rd z_{1} \rd z_{2} \ge 0.
\end{align*}
Because the integrand is always nonnegative, equality holds if and only if the integrand is always 0. Since $K$ is strictly p.d and $\|\vect{x}_{j}\|<1$ for all $j=1,\dots, M$, we obtain $\ma=0$, i.e. $\mW$ is strictly p.d..
\end{proof}

\begin{proof}[Proof of \cref{Xray::MLE}]
For (1), note that any solution of \eqref{normal::eq::lse} belongs to the set $\mW^{\dagger} \mY \oplus \cN(\mW)$ since $\mW^{\dagger} \mY= \mW (\mW^{\dagger})^{2} \mY \in \cR(\mW)=(\cN(\mW))^{\perp}$, thus $\hat{\ma}^{0}:=\mW^{\dagger} \mY \in \cR(\mW)$ is the minimizer with the minimum norm. 
Also, by \eqref{space::repre::decomp}, $\hat{f}^{0}=(\cI^{\mR}_{\mF})^{-1} (\mW \hat{\ma}^{0})$ is the minimizer with the minimum norm. For (2), repeat the argument with $\mY$ replaced by $\mW \ma^{0}$.
For (3), $\hat{\ma}^{0} = \mW^{\dagger} \mY \sim N(\mW \mW^{\dagger} \ma^{0}, \sigma^{2} (\mW^{\dagger})^{2})$ since $\mY \sim N(\mW \ma^{0}, \sigma^{2} \mI)$. Also, by the isometry in \eqref{iso::proj}, for any $f \in \bH^{\mR}_{\mF}$,
\begin{align*}
    \text{Cov}(\hat{f}^{0})f &= \bE [(\hat{f}^{0}-\bE[\hat{f}^{0}]) \cdot \langle \hat{f}^{0}-\bE[\hat{f}^{0}], f \rangle_{\bH(K)}] \\
    &= \bE[(\cI^{\mR}_{\mF})^{-1}(\mW (\hat{\ma}^{0}-\mW \mW^{\dagger} \ma^{0})) \cdot \langle (\mW (\hat{\ma}^{0}-\mW \mW^{\dagger} \ma^{0}), \cI^{\mR}_{\mF} f \rangle_{\mW}] \\
    &= \bE[(\cI^{\mR}_{\mF})^{-1} \mW (\hat{\ma}^{0}-\mW \mW^{\dagger} \ma^{0})) \cdot (\hat{\ma}^{0}-\mW \mW^{\dagger} \ma^{0})^{\top} \cI^{\mR}_{\mF} f] \\
    &= \sigma^{2} (\cI^{\mR}_{\mF})^{-1} \circ  \mW^{\dagger} \circ (\cI^{\mR}_{\mF}) f,
\end{align*}
hence $\hat{f}^{0} \sim N(\cP_{\mF}^{\mR} f^{0}, \sigma^{2} (\cI^{\mR}_{\mF})^{-1} \circ  \mW^{\dagger} \circ (\cI^{\mR}_{\mF}))$.
\end{proof}

\begin{proof}[Proof of \cref{tikhonov}]
(1) is a direct consequence of \eqref{normal::eq::tikho}.  (3) can be shown in a similar way as \cref{Xray::MLE}. For (4), note that $\lim_{\nu \downarrow 0} (\mW+\nu \mI)^{-1}=\mW^{\dagger}$ and use the fact that $\cI^{\mR}_{\mF}: \bH^{\mR}_{\mF} \rightarrow \cR(\mW)$ is an isomorphism by \eqref{iso::proj}. Finally, for (2),
\begin{align*}
    \hat{L}^{\nu}(\hat{\ma}^{\nu})&= \|\mY-\mW \hat{\ma}^{\nu}\|^{2}+\nu (\hat{\ma}^{\nu})^{\top} \mW \hat{\ma}^{\nu} \\
    &=\mY^{\top} \left[(\mI-\mW (\mW + \nu \mI)^{-1})^{2}+ \nu \mW (\mW + \nu I)^{-2} \right] \mY \\
    &=\mY^{\top} \left[(\nu \mI)^{2}+ \nu \mW \right] (\mW + \nu \mI)^{-2} \mY 
    =\nu \mY^{\top} (\mW + \nu \mI)^{-1} \mY.
\end{align*}
\end{proof}

\begin{proof}[Proof of \cref{RKHS:HLCC}]
The first part is trivial since 
\begin{equation*}
    \sP \hat{f}^{\nu}(R, \vect{x}) = \langle \sP_{R} \hat{f}^{\nu}, \tilde{k}^{R}_{\vect{x}} \rangle_{\bH(\tilde{K}^{R})} = \langle \sP_{R}^{*} \tilde{k}^{R}_{\vect{x}}, \hat{f}^{\nu}. \rangle_{\bH(K)},
\end{equation*}
For the HLCC, by \eqref{adj::Boch},
\begin{align*}
    \int_{\bB^{n-1}} (\vect{x}^{\top} \vect{y})^{l} \sP \hat{f}^{\nu}(R,\vect{x}) \rd \vect{x} &= \int_{\bB^{n-1}} (\vect{x}^{\top} \vect{y})^{l} \langle \sP_{R}^{*} \tilde{k}^{R}_{\vect{x}}, \hat{f}^{\nu} \rangle_{\bH(K)} \rd \vect{x} \\
    &= \int_{\bB^{n-1}} \int_{I_{\|\vect{x}\|}} (\vect{x}^{\top} \vect{y})^{l} \langle k_{R^{\top} [\vect{x}:z]}, \hat{f}^{\nu} \rangle_{\bH(K)} \rd z \rd \vect{x} \\
    &= \int_{\bB^{n}} (\vect{z}^{\top} R^{\top}[\vect{y}:0])^{l} \langle k_{\vect{z}}, \hat{f}^{\nu} \rangle_{\bH(K)} \rd \vect{z}
    =q_{l}^{\nu}(\mP_{R}^{*}(\vect{y})),
\end{align*}
where we have set $\vect{z}=R^{\top} [\vect{x}:z]$.
\end{proof}

\begin{proof}[Proof of \cref{MSE:Tikho}]
Since $\hat{f}^{\nu} \in \bH^{\mR}_{\mF}$, by \eqref{space::repre::decomp} and \cref{iso::full::rk}, we get
\begin{align*}
    \bE[ & \|\hat{f}^{\nu}-f^{0}\|^{2}_{\bH(K)}] 
    = \bE[ \|\hat{f}^{\nu}-\cP_{\mF}^{\mR} f^{0}\|^{2}_{\bH(K)}] + \|f^{0}-\cP_{\mF}^{\mR} f^{0}\|^{2}_{\bH(K)} \\
    =& \bE[ (\hat{\ma}^{\nu} - \ma^{0})^{\top} \mW (\hat{\ma}^{\nu} - \ma^{0})] + \|f^{0}-\cP_{\mF}^{\mR} f^{0}\|^{2}_{\bH(K)}.    
\end{align*}
Note that
$(\hat{\ma}^{\nu} - \ma^{0}) \sim N( -\nu (\mW+\nu \mI)^{-1} \ma^{0}, \sigma^{2} (\mW+\nu \mI)^{-2})$ from \cref{tikhonov}, thus
\begin{align*}
    & \bE[ \|\hat{f}^{\nu}-f^{0}\|^{2}_{\bH(K)}]  \\
    =& \bE[\hat{\ma}^{\nu} - \ma^{0}]^{\top} \mW \bE[\hat{\ma}^{\nu} - \ma^{0}] + \text{tr} (\mW \text{Cov} [\hat{\ma}^{\nu} - \ma^{0}]) + \|f^{0}-\cP_{\mF}^{\mR} f^{0}\|^{2}_{\bH(K)} \\
    =& \nu^{2} (\ma^{0})^{\top} \mW (\mW + \nu \mI)^{-2} \ma^{0} + \sigma^{2} \text{tr} (\mW (\mW + \nu \mI)^{-2}) + \|f^{0}-\cP_{\mF}^{\mR} f^{0}\|^{2}_{\bH(K)}.    
\end{align*}
\end{proof}

\begin{proof}[Proof of \cref{stable::recons}]
By \eqref{space::repre::decomp} and \cref{iso::full::rk}, we get
\begin{align*}
    \|\hat{f}^{\nu}-f^{0}\|^{2}_{\bH(K)} 
    =& \|\hat{f}^{\nu}-\cP_{\mF}^{\mR} f^{0}\|^{2}_{\bH(K)} + \|f^{0}-\cP_{\mF}^{\mR} f^{0}\|^{2}_{\bH(K)} \\
    =& (\hat{\ma}^{\nu} - \ma^{0})^{\top} \mW (\hat{\ma}^{\nu} - \ma^{0}) + \|f^{0}-\cP_{\mF}^{\mR} f^{0}\|^{2}_{\bH(K)}.    
\end{align*}
Let $\mW = U D U^{\top}$ be the eigendecomposition of the Gram matrix, where $U \in O(NM)$ and $D = \text{diag}[d_{ij}]$. Denote $C(\nu):= (d + 2 \nu)^{-1}$, 
$\vect{v}:=U^{\top} (\mY - \mW \ma^{0}) \in \bR^{NM}$ and $\vect{w}:= U^{\top} \ma^{0} \in \bR^{NM}$. Due to \eqref{normal::eq::tikho}, we obtain
\begin{align*}
    (\hat{\ma}^{\nu} - \ma^{0})^{\top} \mW (\hat{\ma}^{\nu} - \ma^{0}) 
    =& [(\mW + \nu \mI)(\hat{\ma}^{\nu} - \ma^{0})]^{\top} [\mW (\mW + \nu \mI)^{-2}] [(\mW + \nu \mI)(\hat{\ma}^{\nu} - \ma^{0})] \\
    =& [(\mY - \mW \ma^{0}) - \nu \ma^{0}]^{\top} [\mW (\mW + \nu \mI)^{-2}][(\mY - \mW \ma^{0}) - \nu \ma^{0}] \\
    =& [\vect{v}-\nu \vect{w}]^{\top} [D (D + \nu D)^{-2}] [\vect{v}-\nu \vect{w}] 
    = \sum_{i=1}^{N} \sum_{j=1}^{M} \frac{d_{ij}}{(d_{ij}+\nu)^{2}} (v_{ij}- \nu w_{ij})^{2} \\
    \le& C(\nu) \sum_{i=1}^{N} \sum_{j=1}^{M} v_{ij}^{2} + \sum_{i=1}^{N} \sum_{j=1}^{M} d_{ij} w_{ij}^{2} \\
    =& C(\nu) \|\mY - \mW \ma^{0}\|^{2} + (\ma^{0})^{\top} \mW \ma^{0} 
    = C(\nu) \|\mY - \mW \ma^{0}\|^{2} + \|\cP_{\mF}^{\mR} f^{0}\|^{2}_{\bH(K)},    
\end{align*}
where we have used the following quadratic equation for the inequality: for any $v, w \in \bR$ and $d > 0$, it holds that
\begin{align*}
    &(d + \nu)^{2} \left(\frac{v^{2}}{d+2 \nu}  + d w^{2} \right) - d (v- \nu w)^{2} \\
    &= \left[\frac{(d + \nu)^{2}}{d+2 \nu}-d \right] v^{2} + 2 d \nu v w + d[(d+\nu)^{2}-\nu^{2}] w^{2}
    = \frac{[\nu v + d(d+ 2\nu) w]^{2}}{d+2 \nu}  \ge 0,
\end{align*}
with equality if and only if $v = -d(d+2 \nu) w/\nu$. Therefore,
\begin{align*}
    \|\hat{f}^{\nu}-f^{0}\|^{2}_{\bH(K)} 
    \le& C(\nu) \|\mY - \mW \ma^{0}\|^{2} + \|\cP_{\mF}^{\mR} f^{0}\|^{2}_{\bH(K)} + \|f^{0}-\cP_{\mF}^{\mR} f^{0}\|^{2}_{\bH(K)} \\
    =& \|f^{0}\|^{2}_{\bH(K)} + C(\nu) \| \mY - \mW \ma^{0} \|^{2} 
    = \|f^{0}\|^{2}_{\bH(K)} + C(\nu) \sum_{i=1}^{N} \sum_{j=1}^{M} \varepsilon_{ij}^{2}.
\end{align*}
Equality is achieved if $\vect{v} = -c d(d+2 \nu)/\nu \, e_{IJ}, \vect{w} = c \, e_{IJ}$ for any $c \in \bR$ where the index $(I, J)$ is determined by $d_{IJ}=d>0$.  In this case, $\|\cP_{\mF}^{\mR} f^{0}\|^{2}_{\bH(K)} = (\ma^{0})^{\top} \mW \ma^{0} = \vect{w}^{\top} D \vect{w} = d c^{2}$, so we may take $c= \|\cP_{\mF}^{\mR} f^{0}\|_{\bH(K)}/\sqrt{d}$.
\end{proof}

\subsection{Proofs in Section \ref{sec:invker}}
\begin{proof}[Proof of \cref{RIK2OIK}]
Since $O(n) = SO(n) \rtimes \{I, J \}$, it suffices to show that $K(\vect{z}_{1},\vect{z}_{2})= K(J \vect{z}_{1}, J \vect{z}_{2})$ for fixed $\vect{z}_{1},\vect{z}_{2} \in \bB^{n}$. Let $\vect{z}_{i}=r_{i} \theta_{i}$ ($i= 1, 2$) be the polar decomposition with $r_{i} \in [0, 1]$ and $\theta_{i} \in \bS^{n-1}$, respectively. We may assume that $\theta_{1}=e_{n}$, since the left action of $SO(n)$ on $\bS^{n-1}$ is transitive.
Note that $J \vect{z}_{1}= -\vect{z}_{1}$, thus
\begin{equation*}
    K(J \vect{z}_{1}, J \vect{z}_{2})
    = K(E(-e_{n}) J \vect{z}_{1}, E(-e_{n}) J \vect{z}_{2})    
    = K( \vect{z}_{1}, \vect{z}'_{2}),
\end{equation*}
where $\vect{z}'_{2} := E(-e_{n}) J \vect{z}_{2}$. Let $\vect{z}'_{2}=r_{2} \theta '_{2}$ be the polar decomposition with $\theta '_{i} \in \bS^{n-1}$. We obtain
\begin{align*}
    e_{n}^{\top} \vect{z}'_{2} 
    = e_{n}^{\top} E(-e_{n}) J \vect{z}_{2}
    = (E(-e_{n})^{\top} e_{n})^{\top} J \vect{z}_{2}
    = - e_{n}^{\top} J \vect{z}_{2} = e_{n}^{\top} \vect{z}_{2}=: \cos \phi_{n-1},
\end{align*}
for some $\phi_{n-1} \in [0, \pi]$, thus $\vect{z}_{2}, \vect{z}'_{2} \in \{r_{2} [\sin \phi_{n-1} \tilde{\theta}: \cos \phi_{n-1}]: \tilde{\theta} \in \bS^{n-2} \}$.
As the left action of $SO(n-1)$ on $\bS^{n-2}$ is transitive by \cref{Euler::repre} for $n \ge 3$, there is some $\tilde{R} \in SO(n-1)$ such that $\vect{z}_{2}= \tilde{R}_{+} \vect{z}'_{2}$, where $\tilde{R}_{+}$ is defined as in \eqref{rot::exten}. Because  $\tilde{R}_{+} \in G_{e_{n}}$ leaves $\vect{z}_{1}$ fixed by \cref{SO::stab}, we have
\begin{equation*}
    K(J \vect{z}_{1}, J \vect{z}_{2})  
    = K( \vect{z}_{1}, \vect{z}'_{2})
    = K(\tilde{R}_{+} \vect{z}_{1}, \tilde{R}_{+} \vect{z}'_{2})
    = K(\vect{z}_{1}, \vect{z}_{2}).
\end{equation*}
\end{proof}

\begin{proof}[Proof of \cref{RIK2RIK}]
For any $\tilde{R} \in SO(n-1)$ and $\vect{x}_{1}, \vect{x}_{2} \in \bB^{n-1}$,
\begin{align*}
    \tilde{K}(\tilde{R}\vect{x}_{1}, \tilde{R}\vect{x}_{2})
    &=\tilde{K}^{\tilde{R}_{+}}(\tilde{R}\vect{x}_{1}, \tilde{R}\vect{x}_{2}) \\
    &=\int_{I_{\|\vect{x}_{2}\|}} \int_{I_{\|\vect{x}_{1}\|}} K(\tilde{R}_{+}^{\top}[\tilde{R}\vect{x}_{1}:z_{1}], \tilde{R}_{+}^{\top}[\tilde{R}\vect{x}_{2}:z_{2}]) \rd z_{1} \rd z_{2} \\
    &=\int_{I_{\|\vect{x}_{2}\|}} \int_{I_{\|\vect{x}_{1}\|}} K([\vect{x}_{1}:z_{1}], [\vect{x}_{2}:z_{2}]) \rd z_{1} \rd z_{2} =\tilde{K}(\vect{x}_{1}, \vect{x}_{2}),
\end{align*}
which shows the rotation-invariance of $\tilde{K}$. To show the reflection-invariance, note that $(\tilde{J}_{+} J) \in SO(n)$. For any $\vect{x}_{1}, \vect{x}_{2} \in \bB^{n-1}$,
\begin{align*}
    \tilde{K}(\tilde{J}\vect{x}_{1}, \tilde{J}\vect{x}_{2})
    &=\tilde{K}^{(\tilde{J}_{+} J)}(\tilde{J}\vect{x}_{1}, \tilde{R}\vect{x}_{2}) \\
    &=\int_{I_{\|\vect{x}_{2}\|}} \int_{I_{\|\vect{x}_{1}\|}} K([\vect{x}_{1}: -z_{1}], [\vect{x}_{2}: -z_{2}]) \rd z_{1} \rd z_{2} \\
    &=\int_{I_{\|\vect{x}_{2}\|}} \int_{I_{\|\vect{x}_{1}\|}} K([\vect{x}_{1}: z_{1}], [\vect{x}_{2}: z_{2}]) \rd z_{1} \rd z_{2} =\tilde{K}(\vect{x}_{1}, \vect{x}_{2}).
\end{align*}
\end{proof}

\begin{proof}[Proof of \cref{gene::rot}]
We may assume $R_{2}=E(\vect{r})$ for some $\vect{r} \in \bS^{n-1}$.
\begin{enumerate}[1)]
\item If $\vect{r}=R_{1}^{\top} e_{n}$, by \eqref{adj::ind::gen}, we have
\begin{align*}
    \sP^{*}_{R_{1}}(\tilde{k}_{\tilde{R_{1}} \vect{x}})(\vect{z})
    = \int_{I_{\|\vect{x}\|}} K(\vect{z}, R_{1}^{\top} [\tilde{R_{1}} \vect{x}:z]) \rd z 
    = \int_{I_{\|\vect{x}\|}} K(\vect{z}, E(\vect{r})^{\top} [\vect{x}:z]) \rd z = \sP^{*}_{E(\vect{r})}(\tilde{k}_{\vect{x}})(\vect{z}),  
\end{align*}
hence $\sP^{*}_{R_{1}}(\tilde{k}_{\tilde{R_{1}} \vect{x}})=\sP^{*}_{E(\vect{r})}(\tilde{k}_{\vect{x}})$ and $\bH_{R_{1}}=\bH_{R_{2}}$ follows.
\item If $\vect{r}=-R_{1}^{\top} e_{n}$, by \cref{Euler::repre}, there is a unique $\tilde{R}_{1} \in SO(n-1)$ such that
\begin{equation*}
    R_{1}=
    \begin{pmatrix}
    \tilde{J} \tilde{R}_{1} & \rvline & \mzero \\
    \hline
    \mzero & \rvline & -1
    \end{pmatrix} \cdot E(\vect{r}).
\end{equation*}
Again, we deduce $\bH_{R_{1}}=\bH_{R_{2}}$ since
\begin{align*}
    \sP^{*}_{R_{1}}(\tilde{k}_{\tilde{J} \tilde{R_{1}} \vect{x}})(\vect{z})
    = \int_{I_{\|\vect{x}\|}} K(\vect{z}, R_{1}^{\top} [\tilde{J} \tilde{R_{1}} \vect{x}: -z]) \rd z 
    = \int_{I_{\|\vect{x}\|}} K(\vect{z}, E(\vect{r})^{\top} [\vect{x}:z]) \rd z = \sP^{*}_{E(\vect{r})}(\tilde{k}_{\vect{x}})(\vect{z}).  
\end{align*}
\end{enumerate}
\end{proof}

\begin{proof}[Proof of \cref{unit:repre}]
It is trivial that $\rho(I)$ is the identity map and $\rho(R_{1} R_{2})=\rho(R_{1})\rho(R_{2})$ for any $R_{1}, R_{2} \in SO(n)$. Also, \eqref{gen::rot} and $\|k_{R\vect{z}}\|_{\bH(K)}^{2} =K(R\vect{z}, R\vect{z})=K(\vect{z}, \vect{z})=\|k_{\vect{z}}\|_{\bH(K)}^{2}$ together
imply that $\rho$ is unitary, so it only remains to show that the map $SO(n) \rightarrow \bH(K), R \mapsto \rho(R)f=f(R^{\top} \, \cdot)$ is continuous for any $f \in \bH(K)$. As $\rho$ is unitary, it suffices to show that this map is continuous at $I \in SO(n)$. We further assume that $f=\sum_{i=1}^{n} \alpha_{i} k_{\vect{z}_{i}}$ for some $\alpha_{1}, \alpha_{2}, \dots, \alpha_{n} \in \bC$ and $\vect{z}_{1}, \vect{z}_{2}, \dots, \vect{z}_{n} \in \bB^{n}$, since the linear span of generators is dense in $\bH(K)$. By \eqref{gen::rot} and triangular inequality,
\begin{align*}
    \|\rho(R)f-\rho(I)f\|_{\bH(K)} 
    \le \sum_{i=1}^{n} |\alpha_{i}| \|k_{R \vect{z}_{i}} - k_{\vect{z}_{i}} \|_{\bH(K)} 
    = \sum_{i=1}^{n} |\alpha_{i}| \sqrt{2(K(\vect{z}_{i}, \vect{z}_{i})-K(R \vect{z}_{i}, \vect{z}_{i}))} \rightarrow 0,
\end{align*}
as $R \rightarrow I$, due to the continuity of $K:\bB^{n} \times \bB^{n} \rightarrow \bR$.
\end{proof}

\begin{proof}[Proof of \cref{repn:btf}]
For any $f \in \bH$ and $\vect{x} \in \bB^{n-1}$,
\begin{equation*}
    \sP_{R} f(\vect{x}) =
    \int_{I_{\|\vect{x}\|}} f(R^{\top} [\vect{x}:z]) \rd z=\int_{I_{\|\vect{x}\|}} (\rho(R)f)([\vect{x}:z]) \rd z=\sP_{I} \circ \rho(R) f(\vect{x}),
\end{equation*}
hence $\sP_{R}=\sP_{I} \circ \rho(R)$. Then, it follows that
$\sP_{R}^{*}=(\sP_{I} \circ \rho(R))^{*}=\rho(R)^{*} \circ \sP_{I}^{*}=\rho(R^{\top}) \circ \sP_{I}^{*}$ for the back-projection, and by combining two formulas, we get $\sP_{R_{1}} \circ \sP_{R_{2}}^{*}=\sP_{I} \circ \rho(R_{1}) \circ \rho(R_{2}^{\top}) \circ \sP_{I}^{*}=\sP_{I} \circ \rho(R_{1}R_{2}^{\top}) \circ \sP_{I}^{*}$.
\end{proof}

\begin{lemma}[Schoenberg, \cite{schoenberg1942positive}]\label{schoenberg}
A continuous kernel $K:\bS^{n-1} \times \bS^{n-1} \rightarrow \bR$ is rotation-invariant if and only if there exists $c_{l} \ge 0$ for $l=0,1,\dots$ with $\sum_{l=0}^{\infty} c_{l} < \infty$ such that for any $\theta_{1}, \theta_{2} \in \bS^{n-1}$,
\begin{equation*}
    K(\theta_{1}, \theta_{2})=\sum_{l=0}^{\infty} c_{l} \cdot C_{l}^{n/2-1}(\theta_{1}^{\top} \theta_{2}),
\end{equation*}
where the sum converges absolutely and uniformly. 
\end{lemma}

\begin{proof}[Proof of \cref{scheon::gen}]
Let $K:\bB^{n} \times \bB^{n} \rightarrow \bR$ be an $O(n)$-invariant continuous kernel. 
Fix $r_{1}, r_{2} \in [0,1]$ and define a new kernel on $\bS^{n-1}$ as
\begin{equation*}
    K_{r_{1}, r_{2}}(\theta_{1}, \theta_{2}) := 
    K(r_{1}\theta_{1}, r_{1}\theta_{2}) +
    K(r_{1}\theta_{1}, r_{2}\theta_{2}) + K(r_{2}\theta_{1}, r_{1}\theta_{2}) + K(r_{2}\theta_{1}, r_{2}\theta_{2}),
\end{equation*}
which is continuous and $O(n)$-invariant. By \cref{schoenberg}, it follows that there are $d_{l}(r_{1},r_{2})=d_{l}(r_{2},r_{1}) \ge 0$ with $\sum_{l=0}^{\infty} d_{l}(r_{1},r_{2}) < \infty$ such that for any $\theta_{1}, \theta_{2} \in \bS^{n-1}$,
\begin{equation*}
    K_{r_{1}, r_{2}}(\theta_{1}, \theta_{2}) =\sum_{l=0}^{\infty} d_{l}(r_{1}, r_{2}) C_{l}^{n/2-1}(\theta_{1}^{\top} \theta_{2}),
\end{equation*}
where the sum converges absolutely and uniformly on $\bS^{n-1}$. Since $K_{r_{1}, r_{2}}= K_{r_{2}, r_{1}}$, we get $K(r_{1}\theta_{1}, r_{2}\theta_{2}) = \left( 4 K_{r_{1}, r_{2}}(\theta_{1}, \theta_{2})- K_{r_{1}, r_{1}}(\theta_{1}, \theta_{2}) - K_{r_{2}, r_{2}}(\theta_{1}, \theta_{2}) \right)/8$, which results in
\begin{equation*}
    K(r_{1} \theta_{1}, r_{2} \theta_{2})=\sum_{l=0}^{\infty} c_{l}(r_{1}, r_{2}) C_{l}^{n/2-1}(\theta_{1}^{\top} \theta_{2}),
\end{equation*}
with $c_{l}(r_{1}, r_{2})=(4 d_{l}(r_{1}, r_{2})-d_{l}(r_{1}, r_{1})-c_{l}(r_{2}, r_{2}))/8$. The absolute and uniform convergence on $\bB^{n}$ follows from the fact that $[0,1]$ is compact. Also, note that $c_{l}(0,0)=0$ for $l \ge 1$ due to continuity, and
\begin{equation*}
    \sup_{r_{1},r_{2},\theta_{1},\theta_{2}} |K(r_{1} \theta_{1}, r_{2} \theta_{2})|
    =\max_{r \in [0,1]} K(r \theta, r \theta)
    =\max_{r \in [0,1]} \sum_{l=0}^{\infty} c_{l}(r,r) < \infty.
\end{equation*}
By \eqref{sphe::gene}, we obtain
\begin{align*}
    c_{l}(r_{1},r_{2})
    &=c_{l}(r_{1},r_{2}) \cdot C_{l}^{n/2-1}(\theta_{1}^{\top} \theta_{1}) \\
    &= \frac{c_{l}(r_{1},r_{2})}{p^{n/2-1}_{l} |\bS^{n-2}|} \int_{\bS^{n-1}} C_{l}^{n/2-1}(\theta_{1}^{\top} \theta_{2}) C_{l}^{n/2-1}(\theta_{1}^{\top} \theta_{2}) \rd \theta_{2}\\
    &=\frac{1}{p^{n/2-1}_{l} |\bS^{n-2}|} \int_{\bS^{n-1}} K(r_{1}\theta_{1}, r_{2}\theta_{2}) C_{l}^{n/2-1}(\theta_{1}^{\top} \theta_{2}) \rd \theta_{2},        
\end{align*}
independent of the choice of $\theta_{1} \in \bS^{n-1}$. Hence,
\begin{align*}
    c_{l}(r_{1},r_{2})
    &=\frac{1}{p^{n/2-1}_{l} |\bS^{n-2}| |\bS^{n-1}|} \int_{\bS^{n-1}} \int_{\bS^{n-1}} K(r_{1}\theta_{1}, r_{2}\theta_{2}) C_{l}^{n/2-1}(\theta_{1}^{\top} \theta_{2}) \rd \theta_{2} \rd \theta_{1},        
\end{align*}
which demonstrates that $c_{l}$ is a continuous kernel on $[0,1]$.

To show the opposite direction, define a truncated kernel $K_{T}=\sum_{l=0}^{T} c_{l} \otimes C_{l}^{n/2-1}$ for some $T \in \mathbb{N}$, which is $O(n)$-invariant and continuous. The bivariate function $K$ defined in \eqref{sch::series} converges absolutely and uniformly since
\begin{equation*}
    \sup_{r_{1},r_{2},\theta_{1},\theta_{2}} \|K(r_{1} \theta_{1}, r_{2} \theta_{2})-K_{T}(r_{1} \theta_{1}, r_{2} \theta_{2}) \| \le \max_{r \in [0,1]} \sum_{l=T+1}^{\infty} c_{l}(r,r) \rightarrow 0,    
\end{equation*}
as $T \rightarrow \infty$. Hence, $K:\bB^{n} \times \bB^{n} \rightarrow \bR$ is also an $O(n)$-invariant continuous kernel.
\end{proof}

\begin{proof}[Proof of \cref{Xray::Aronszajn}]
For $l=0,1, \dots$, let $K_{l}=c_{l} \otimes C_{l}^{n/2-1}$. From \eqref{rk::Gegen} in \cref{sec::sphe::har}, we note that 
$\bH(C_{l}^{n/2-1})$ spans $\cH_{l}(\bS^{n-1})$ with an ONB 
\begin{equation*}
    \{(p^{n/2-1}_{l} |\bS^{n-2}|)^{1/2} Y_{lk}(\theta): k=1,2,\dots, N(n,l) \}.
\end{equation*}
As the kernel $K_{l}$ is the tensor product of $c_{l}$ and $C_{l}^{n/2-1}$, its RKHS becomes $\bH(K_{l})=\bH(c_{l}) \otimes \bH(C_{l}^{n/2-1})= \{\sum_{k=1}^{N(n,l)} h_{lk}(r) Y_{lk}(\theta): h_{lk} \in \bH(c_{l}) \}$.
Now, consider the countable direct sum
\begin{equation*}
    \widehat{\bH}:=\bigoplus_{l=0}^{\infty} \bH(K_{l}) = \left\{ \hat{f}= (f_{l})_{l=0}^{\infty}: f_{l} \in \bH(K_{l}), \sum_{l=0}^{\infty} \|f_{l}\|^{2}_{\bH(K_{l})} < \infty \right\},
\end{equation*}
which is a Hilbert space, equipped with the inner product $\langle \hat{f}, \hat{f '} \rangle = \sum_{l=0}^{\infty} \langle f_{l}, f '_{l} \rangle_{\bH(K_{l})}$. For any $z \in \bB^{n}$, define $\widehat{k_{\vect{z}}} :=((k_{l})_{\vect{z}})_{l=0}^{\infty}$. Note that $\widehat{k_{\vect{z}}} \in \widehat{\bH}$ since
\begin{equation*}
    \|\widehat{k_{\vect{z}}}\|^{2}= \sum_{l=0}^{\infty} \|(k_{l})_{\vect{z}}\|^{2}_{\bH(K_{l})} \le \max_{r \in [0,1]} \sum_{l=0}^{\infty} c_{l}(r,r) < \infty.
\end{equation*}

Therefore, there is a well-defined summation map $S: \widehat{\bH} \rightarrow (\mathcal{C}(\bB^{n}), \|\cdot\|_{\infty})$ given by $(S \hat{f}) (\vect{z}) := \sum_{l=0}^{\infty} f_{l} (\vect{z}) = \langle \hat{f}, \hat{k_{\vect{z}}} \rangle, \ \vect{z} \in \bB^{n}$,
which is continuous by the Cauchy-Schwarz inequality. By the Funk-Hecke formula (\cref{Funk::Hecke}), for any $l \ge 0$ and $r \in [0, 1]$,
\begin{align*}
    \int_{\bS^{n-1}} (S \hat{f}) (r \theta_{2}) C_{l}^{n/2-1}(\theta_{1}^{\top} \theta_{2}) \rd \theta_{2}
    &=\int_{\bS^{n-1}} \left( \sum_{l=0}^{\infty} f_{l} (r \theta_{2}) \right) C_{l}^{n/2-1}(\theta_{1}^{\top} \theta_{2})  \rd \theta_{2} \\
    &=p^{n/2-1}_{l} |\bS^{n-2}| f_{l}(r \theta_{1}),
\end{align*}
so the null space of $S$ is trivial, i.e. $\cN(S)=\{0\}$. Define
\begin{align*}
    \bH: &=\cR(S) = \left\{ \sum_{l=0}^{\infty} f_{l}: f_{l} \in \bH(K_{l}), \sum_{l=0}^{\infty} \|f_{l}\|^{2}_{\bH(K_{l})} < \infty \right\} \\
    &= \left\{\sum_{l=0}^{\infty} \sum_{k=1}^{N(n,l)} h_{lk}(r) Y_{lk}(\theta): h_{lk} \in \bH(c_{l}), \sum_{l=0}^{\infty} \sum_{k=1}^{N(n,l)} \frac{\|h_{lk}\|^{2}_{\bH(c_{l})}}{p^{n/2-1}_{l} |\bS^{n-2}|} <\infty \right\},
\end{align*}
then $S: \widehat{\bH} \rightarrow \bH$ is a vector space isomorphism. Thus, if we endow $\bH$ the inner product $\langle S \hat{f}, S \hat{f '} \rangle_{\bH}= \langle \hat{f}, \hat{f '} \rangle$, $\bH$ becomes a Hilbert space. Finally, recall that $S\hat{k_{\vect{z}}}=\sum_{l=0}^{\infty} K_{l}(\cdot, \vect{z})= K(\cdot, \vect{z})=k_{\vect{z}}$ for any $\vect{z} \in \bB^{n}$, hence for any $\hat{f} \in \widehat{\bH}$, we get $\langle S \hat{f}, k_{\vect{z}} \rangle_{\bH} = \langle \hat{f}, \hat{k_{\vect{z}}} \rangle = S \hat{f} (\vect{z})$,
which demonstrates that $\bH=\bH(K)$.
\end{proof}

\subsection{Proofs in Appendix \ref{sec:Gauss:ker}}
\begin{lemma}\label{double::erf}
For any $c_{1}, c_{2} \in \bR$,
\begin{equation*}
    \int_{-c_{2}}^{c_{2}} \int_{-c_{1}}^{c_{1}} \exp(-\gamma (z_{1}-z_{2})^{2}) \rd z_{1} \rd z_{2}  
    = \frac{1}{2\gamma} \sum_{i, j \in \bZ_{2}} (-1)^{i+j} \Phi[\sqrt{\gamma}((-1)^{i} c_{1}+ (-1)^{j} c_{2})].
\end{equation*}
\end{lemma}
\begin{proof}
For any $c_{1}, c_{2} \in \bR$,
\begin{align*}
    &\int_{-c_{2}}^{c_{2}} \int_{-c_{1}}^{c_{1}} \exp(-\gamma (z_{1}-z_{2})^{2}) \rd z_{1} \rd z_{2} = \frac{1}{\sqrt{\gamma}} \int_{-c_{2}}^{c_{2}} \left(\int_{\sqrt{\gamma}(-c_{1}-z_{2})}^{\sqrt{\gamma}(c_{1}-z_{2})} e^{-t^{2}} \rd t \right) \rd z_{2} \nonumber \\
    &= \frac{\sqrt{\pi}}{2\sqrt{\gamma}} \int_{-c_{2}}^{c_{2}} \left[\erf(\sqrt{\gamma}(c_{1}-z_{2}))-\erf(\sqrt{\gamma}(-c_{1}-z_{2})) \right] \rd z_{2} \nonumber \\
    &= \frac{1}{2\gamma} \sum_{i, j \in \bZ_{2}} (-1)^{i+j} \Phi[\sqrt{\gamma}((-1)^{i} c_{1}+ (-1)^{j} c_{2})].
\end{align*}        
\end{proof}

\begin{proof}[Proof of \cref{gauss::adj}]
By \cref{double::erf}, 
\begin{align*}
    \tilde{K}(\vect{x}_{1},\vect{x}_{2}) 
    &= \int_{I_{\|\vect{x}_{2}\|}} \int_{I_{\|\vect{x}_{1}\|}} K([\vect{x}_{1}:z_{1}], [\vect{x}_{2}:z_{2}]) \rd z_{1} \rd z_{2} \\
    &= e^{-\gamma \|\vect{x}_{1}-\vect{x}_{2}\|^{2}} \int_{-W(\vect{x}_{2})}^{W(\vect{x}_{2})} \int_{-W(\vect{x}_{1})}^{W(\vect{x}_{1})}  \exp(-\gamma (z_{1}-z_{2})^{2}) \rd z_{1} \rd z_{2} \\
    &= \frac{e^{-\gamma \|\vect{x}_{1}-\vect{x}_{2}\|^{2}}}{2\gamma} \sum_{i, j \in \bZ_{2}} (-1)^{i+j} \Phi[\sqrt{\gamma}((-1)^{i} W(\vect{x}_{1})+ (-1)^{j}W(\vect{x}_{2}))].
\end{align*}
For the adjoint map of the induced generator, note that we define the Euclidean projection to be $R\vect{z} = [\mP_{R}(\vect{z}): \vect{z}^{\top} \vect{r}]$, hence by \eqref{adj::ind::gen},
\begin{align*}
    \sP_{R}^{*} ( \tilde{k}_{\vect{x}} )(\vect{z})
    =&\int_{I_{\|\vect{x}\|}} K(R \vect{z}, [\vect{x}:z]) \rd z
    =\int_{I_{\|\vect{x}\|}} K([\mP_{R}(\vect{z}): \vect{z}^{\top} \vect{r}], [\vect{x}:z]) \rd z \\
    =&e^{-\gamma \|\vect{x}-\mP_{R}(\vect{z})\|^{2}} \int_{-W(\vect{x})}^{W(\vect{x})}  \exp(-\gamma(z-\vect{z}^{\top} \vect{r})^{2}) \rd z \\
    =& \frac{\sqrt{\pi} e^{-\gamma \|\vect{x}-\mP_{R}(\vect{z})\|^{2}}}{2\sqrt{\gamma}} \sum_{i \in \bZ_{2}} (-1)^{i} \erf[\sqrt{\gamma} ((-1)^{i} W(\vect{x})-\mP_{R}(\vect{z}))].
\end{align*}
\end{proof}

\begin{proof}[Proof of \cref{rot::dist}]
To ease confusion, we replace by $\theta = \vect{r}$ and $(\tilde{\theta}, \phi_{n-1})=(\tilde{\vect{r}}, \arccos (r))$ in the proof. Note that
\begin{align*}
    \|[\vect{x}_{1}:z_{1}]-R^{\top}[\vect{x}_{2}:z_{2}]\|^{2}
    =& \|[\vect{x}_{1}:z_{1}]-E(\theta)^{\top}[\tilde{R}^{\top} \vect{x}_{2}:z_{2}]\|^{2}.
\end{align*}
Since $\cos \phi_{n-1} = r$ and $\sin \phi_{n-1} = w(r)$, 
\begin{equation*}
    E(\theta)^{\top}=
    \begin{pmatrix}
    E(\tilde{\theta})^{\top} & \rvline & \mzero \\
    \hline
    \mzero & \rvline & 1
    \end{pmatrix} \cdot R_{\phi_{n-1}}^{n-1}=
    \begin{pmatrix}
    E(\tilde{\theta})^{\top} & \rvline & \mzero \\
    \hline
    \mzero & \rvline & 1
    \end{pmatrix} \cdot
    \begin{pmatrix}
    \mI_{n-2} & \rvline & \mzero \\
    \hline
    \mzero & \rvline &
    \begin{matrix}
        r & w(r) \\
        -w(r) & r
    \end{matrix}
    \end{pmatrix},
\end{equation*}
we get
\begin{align*}
    R^{\top}[\vect{x}_{2}:z_{2}]
    &=E(\theta)^{\top}[\tilde{R}^{\top} \vect{x}_{2}:z_{2}]=
    \begin{pmatrix}
    E(\tilde{\theta})^{\top} & \rvline & \mzero \\
    \hline
    \mzero & \rvline & 1
    \end{pmatrix} \cdot
    \begin{pmatrix}
    \mI_{n-2} & \rvline & \mzero \\
    \hline
    \mzero & \rvline &
    \begin{matrix}
        r & w(r) \\
        -w(r) & r
    \end{matrix}
    \end{pmatrix} \cdot
    \begin{pmatrix}
        \tilde{\vect{x}}_{2 \tilde{R}} \\ x_{2 \tilde{R}} \\ z_{2}
    \end{pmatrix} \\
    &=\begin{pmatrix}
    E(\tilde{\theta})^{\top} & \rvline & \mzero \\
    \hline
    \mzero & \rvline & 1
    \end{pmatrix} \cdot
    \begin{pmatrix}
        \tilde{\vect{x}}_{2 \tilde{R}} \\ r \cdot x_{2 \tilde{R}} +w(r) \cdot z_{2} \\ \hline -w(r) \cdot x_{2 \tilde{R}} +r \cdot z_{2}
    \end{pmatrix}.
\end{align*}
Therefore, 
\begin{align*}
    &\|[\vect{x}_{1}:z_{1}]-R^{\top}[\vect{x}_{2}:z_{2}]\|^{2} \\
    &=\left\|\vect{x}_{1}-E(\tilde{\theta})^{\top}
    \begin{pmatrix}
        \tilde{\vect{x}}_{2 \tilde{R}} \\ r \cdot x_{2 \tilde{R}} +w(r) \cdot z_{2}
    \end{pmatrix}\right\|^{2} + (z_{1}+w(r) \cdot x_{2 \tilde{R}} -r \cdot z_{2})^{2} \\
    &=\left\|E(\tilde{\theta}) \vect{x}_{1}-
    \begin{pmatrix}
        \tilde{\vect{x}}_{2 \tilde{R}} \\ r \cdot x_{2 \tilde{R}} +w(r) \cdot z_{2}
    \end{pmatrix}\right\|^{2} + (z_{1}+w(r) \cdot x_{2 \tilde{R}} -r \cdot z_{2})^{2} \\
    &=\left\|\begin{pmatrix}
        \tilde{\vect{x}}_{1 \theta} \\ x_{1 \theta}
    \end{pmatrix}-
    \begin{pmatrix}
        \tilde{\vect{x}}_{2 \tilde{R}} \\ r \cdot x_{2 \tilde{R}} +w(r) \cdot z_{2}
    \end{pmatrix}\right\|^{2} + (z_{1}+w(r) \cdot x_{2 \tilde{R}} -r \cdot z_{2})^{2} \\
    &=\|\tilde{\vect{x}}_{1 \theta}-\tilde{\vect{x}}_{2 \tilde{R}}\|^{2}+(x_{1 \theta}-r \cdot x_{2 \tilde{R}} -w(r) \cdot z_{2})^{2}+ (z_{1}+w(r) \cdot x_{2 \tilde{R}} -r \cdot z_{2})^{2}.
\end{align*}
If $r= \pm 1$, then $w(r)=0$, thus
\begin{align*}
    \|[\vect{x}_{1}:z_{1}]-R^{\top}[\vect{x}_{2}:z_{2}]\|^{2}
    &=\|\tilde{\vect{x}}_{1 \theta}-\tilde{\vect{x}}_{2 \tilde{R}}\|^{2}+(x_{1 \theta}-r \cdot x_{2 \tilde{R}})^{2}+ (z_{1} -r \cdot z_{2})^{2} \\
    &=\|[\tilde{\vect{x}}_{1 \theta}:r \cdot x_{1 \theta}]-[\tilde{\vect{x}}_{2 \tilde{R}}:x_{2 \tilde{R}}]\|^{2}+ (z_{1} -r \cdot z_{2})^{2}.
\end{align*}
If $r \in (-1,+1)$, then $w(r)>0$, and we get
\begin{align*}
    &\|[\vect{x}_{1}:z_{1}]-R^{\top}[\vect{x}_{2}:z_{2}]\|^{2} \\
    &=\|\tilde{\vect{x}}_{1 \theta}-\tilde{\vect{x}}_{2 \tilde{R}}\|^{2}+ (\mu_{2}^{R}-w(r) \cdot z_{2})^{2}  +\left\{\left( z_{1}-\frac{\mu_{1}^{R}}{w(r)} \right)-r \left( z_{2}-\frac{\mu_{2}^{R}}{w(r)} \right) \right\}^{2} \\
    &=\|\tilde{\vect{x}}_{1 \theta}-\tilde{\vect{x}}_{2 \tilde{R}}\|^{2} +
    \left[\begin{pmatrix}
    z_{1} \\ z_{2}
    \end{pmatrix}-\frac{1}{w(r)}\begin{pmatrix}
    \mu_{1}^{R} \\ \mu_{2}^{R}
    \end{pmatrix}\right]^{\top}
    \begin{pmatrix}
    1 & -r \\
    -r & 1
    \end{pmatrix}
    \left[\begin{pmatrix}
    z_{1} \\ z_{2}
    \end{pmatrix}-\frac{1}{w(r)}\begin{pmatrix}
    \mu_{1}^{R} \\ \mu_{2}^{R}
    \end{pmatrix}\right],
\end{align*}
where $\mu_{1}^{R}=r \cdot x_{1 \theta}-x_{2 \tilde{R}}, \    \mu_{2}^{R}=x_{1 \theta}-r \cdot x_{2 \tilde{R}}$.
\end{proof}

\begin{proof}[Proof of \cref{gauss::weight}]
Since the back-then-forward projection  only depends on the relative angle $R=R_{2}R_{1}^{\top} \in SO(n)$ by \cref{repn:btf}, we may assume $R_{1}=R$ and $R_{2}=I$.
If $R \sim I$, then by \eqref{adj::ind::gen} and \cref{rot::dist},
\begin{align*}
    &(\sP_{R} \circ \sP_{I}^{*} \tilde{k}_{\vect{x}_{2}})(\vect{x}_{1}) \\
    =&\int_{I_{\|\vect{x}_{2}\|}} \int_{I_{\|\vect{x}_{1}\|}} K( [\vect{x}_{1}:z_{1}], R^{\top} [\vect{x}_{2}:z_{2}]) \rd z_{1} \rd z_{2} \\
    =&\int_{I_{\|\vect{x}_{2}\|}} \int_{I_{\|\vect{x}_{1}\|}} \exp(-\gamma \|[\vect{x}_{1}:z_{1}]-R^{\top}[\vect{x}_{2}:z_{2}]\|^{2}) \rd z_{1} \rd z_{2} \\
    =&\exp(-\gamma \|[\tilde{\vect{x}}_{1 \vect{r}}:r \cdot x_{1 \vect{r}}]-\tilde{R}^{\top} \vect{x}_{2}\|^{2}) \int_{I_{\|\vect{x}_{2}\|}} \int_{I_{\|\vect{x}_{1}\|}} \exp(-\gamma (z_{1}-r \cdot z_{2})^{2}) \rd z_{1} \rd z_{2}.
\end{align*}
Applying \cref{double::erf} with a change of variables, the above formula becomes
\begin{align*}
    &(\sP_{R} \circ \sP_{I}^{*} \tilde{k}_{\vect{x}_{2}})(\vect{x}_{1}) 
    = \frac{\exp(-\gamma \|[\tilde{\vect{x}}_{1 \vect{r}}:r x_{1 \vect{r}}]-\tilde{R}^{\top} \vect{x}_{2}\|^{2})}{2\gamma} \times \\
    & \quad \sum_{i, j \in \bZ_{2}} (-1)^{i+j} \Phi [\sqrt{\gamma}((-1)^{i} W(\vect{x}_{1})+ (-1)^{j}W(\vect{x}_{2}))].
\end{align*}
On the other hand, if $R \nsim I$, i.e. $r \in (-1,+1)$, then by \cref{rot::dist},    
\begin{align*}
    &(\sP_{R} \circ \sP_{I}^{*} \tilde{k}_{\vect{x}_{2}})(\vect{x}_{1}) \\
    =&\int_{I_{\|\vect{x}_{2}\|}} \int_{I_{\|\vect{x}_{1}\|}} K( [\vect{x}_{1}:z_{1}], R^{\top} [\vect{x}_{2}:z_{2}]) \rd z_{1} \rd z_{2} \\
    =&\int_{I_{\|\vect{x}_{2}\|}} \int_{I_{\|\vect{x}_{1}\|}} \exp(-\gamma \|[\vect{x}_{1}:z_{1}]-R^{\top}[\vect{x}_{2}:z_{2}]\|^{2}) \rd z_{1} \rd z_{2} \\
    =&\exp(-\gamma \|\tilde{\vect{x}}_{1 \vect{r}}-\tilde{\vect{x}}_{2 R}\|^{2}) \times \int_{I_{\|\vect{x}_{2}\|}} \int_{I_{\|\vect{x}_{1}\|}} \exp (-\gamma \left[\begin{pmatrix}
    z_{1} \\ z_{2}
    \end{pmatrix}-\frac{1}{w(r)}\begin{pmatrix}
    \mu_{1}^{R} \\ \mu_{2}^{R}
    \end{pmatrix}\right]^{\top}
    \begin{pmatrix}
    1 & -r \\
    -r & 1
    \end{pmatrix} \\
    &\qquad \left[\begin{pmatrix}
    z_{1} \\ z_{2}
    \end{pmatrix}-\frac{1}{w(r)}\begin{pmatrix}
    \mu_{1}^{R} \\ \mu_{2}^{R}
    \end{pmatrix}\right] ) \rd z_{1} \rd z_{2}.
\end{align*}
Let $(X_{1}, X_{2})$ be the bivariate normal distribution with mean and covariance matrix
\begin{equation*}
    \mu=\frac{1}{w(r)}(\mu_{1}^{R}, \mu_{2}^{R}), \quad
    \Sigma=\frac{1}{2\gamma}
    \begin{pmatrix}
    1 & -r \\
    -r & 1
    \end{pmatrix}^{-1}
    =\frac{1}{2\gamma(1-r^{2})}
    \begin{pmatrix}
    1 & r \\
    r & 1
    \end{pmatrix},
\end{equation*}
respectively. Then, we get
\begin{align*}
    &(\sP_{R} \circ \sP_{I}^{*} \tilde{k}_{\vect{x}_{2}})(\vect{x}_{1}) \\
    =&2\pi |\Sigma|^{1/2} \cdot \exp(-\gamma \|\tilde{\vect{x}}_{1 \vect{r}}-\tilde{\vect{x}}_{2 R}\|^{2}) \cdot \bP((X_{1},X_{2}) \in I_{\|\vect{x}_{1}\|} \times I_{\|\vect{x}_{2}\|}) \\
    =&\frac{\pi \cdot \exp(-\gamma \|\tilde{\vect{x}}_{1 \vect{r}}-\tilde{\vect{x}}_{2 R}\|^{2})}{\gamma w(r)} \bP((X_{1},X_{2}) \in I_{\|\vect{x}_{1}\|} \times I_{\|\vect{x}_{2}\|}) \\ 
    =&\frac{\pi \cdot \exp(-\gamma \|\tilde{\vect{x}}_{1 \vect{r}}-\tilde{\vect{x}}_{2 R}\|^{2})}{\gamma w(r)} \times \\
    &\sum_{i, j \in \bZ_{2}} (-1)^{i+j} 
    \Phi_{2}^{r}\left[\sqrt{2\gamma} ((-1)^{i} w(r) W(\vect{x}_{1})-\mu_{1}^{R}), \sqrt{2\gamma} ((-1)^{j} w(r) W(\vect{x}_{2})-\mu_{2}^{R})\right],
\end{align*}
since $(Z_{1}, Z_{2}):=\sqrt{2\gamma}(w(r) (X_{1}, X_{2})-(\mu_{1}^{R}, \mu_{2}^{R}))$ is the standardized bivariate normal distribution with the correlation $\rho=r$.
\end{proof}

\end{appendix}

\begin{acks}[Acknowledgments]
We thank Prof. Micha\"el Unser (EPFL) for several constructive comments.
\end{acks}

\bibliographystyle{imsart-number}
\bibliography{bibliography}   

\begin{thebibliography}{45}

\bibitem{adler2017solving}
\begin{barticle}[author]
\bauthor{\bsnm{Adler},~\bfnm{Jonas}\binits{J.}} \AND \bauthor{\bsnm{{\"O}ktem},~\bfnm{Ozan}\binits{O.}}
(\byear{2017}).
\btitle{Solving ill-posed inverse problems using iterative deep neural networks}.
\bjournal{Inverse Problems}
\bvolume{33}
\bpages{124007}.
\end{barticle}
\endbibitem

\bibitem{anden2018structural}
\begin{barticle}[author]
\bauthor{\bsnm{And{\'e}n},~\bfnm{Joakim}\binits{J.}} \AND \bauthor{\bsnm{Singer},~\bfnm{Amit}\binits{A.}}
(\byear{2018}).
\btitle{Structural variability from noisy tomographic projections}.
\bjournal{SIAM Journal on Imaging Sciences}
\bvolume{11}
\bpages{1441--1492}.
\end{barticle}
\endbibitem

\bibitem{aronszajn1950theory}
\begin{barticle}[author]
\bauthor{\bsnm{Aronszajn},~\bfnm{Nachman}\binits{N.}}
(\byear{1950}).
\btitle{Theory of reproducing kernels}.
\bjournal{Transactions of the American Mathematical Society}
\bvolume{68}
\bpages{337--404}.
\end{barticle}
\endbibitem

\bibitem{basu2000uniqueness}
\begin{barticle}[author]
\bauthor{\bsnm{Basu},~\bfnm{Samit}\binits{S.}} \AND \bauthor{\bsnm{Bresler},~\bfnm{Yoram}\binits{Y.}}
(\byear{2000}).
\btitle{Uniqueness of tomography with unknown view angles}.
\bjournal{IEEE Transactions on Image Processing}
\bvolume{9}
\bpages{1094--1106}.
\end{barticle}
\endbibitem

\bibitem{christ1984estimates}
\begin{barticle}[author]
\bauthor{\bsnm{Christ},~\bfnm{Michael}\binits{M.}}
(\byear{1984}).
\btitle{Estimates for the k-plane transform}.
\bjournal{Indiana University Mathematics Journal}
\bvolume{33}
\bpages{891--910}.
\end{barticle}
\endbibitem

\bibitem{da2014stochastic}
\begin{bbook}[author]
\bauthor{\bsnm{Da~Prato},~\bfnm{Giuseppe}\binits{G.}} \AND \bauthor{\bsnm{Zabczyk},~\bfnm{Jerzy}\binits{J.}}
(\byear{2014}).
\btitle{Stochastic equations in infinite dimensions}.
\bpublisher{Cambridge university press}.
\end{bbook}
\endbibitem

\bibitem{dai2013approximation}
\begin{bbook}[author]
\bauthor{\bsnm{Dai},~\bfnm{Feng}\binits{F.}}
(\byear{2013}).
\btitle{Approximation theory and harmonic analysis on spheres and balls}.
\bpublisher{Springer}.
\end{bbook}
\endbibitem

\bibitem{dai2013spherical}
\begin{barticle}[author]
\bauthor{\bsnm{Dai},~\bfnm{Feng}\binits{F.}}, \bauthor{\bsnm{Xu},~\bfnm{Yuan}\binits{Y.}}, \bauthor{\bsnm{Dai},~\bfnm{Feng}\binits{F.}} \AND \bauthor{\bsnm{Xu},~\bfnm{Yuan}\binits{Y.}}
(\byear{2013}).
\btitle{Spherical harmonics}.
\bjournal{Approximation theory and harmonic analysis on spheres and balls}
\bpages{1--27}.
\end{barticle}
\endbibitem

\bibitem{frank2006three}
\begin{bbook}[author]
\bauthor{\bsnm{Frank},~\bfnm{Joachim}\binits{J.}}
(\byear{2006}).
\btitle{Three-dimensional electron microscopy of macromolecular assemblies: visualization of biological molecules in their native state}.
\bpublisher{Oxford university press}.
\end{bbook}
\endbibitem

\bibitem{hadamard1902problemes}
\begin{barticle}[author]
\bauthor{\bsnm{Hadamard},~\bfnm{Jacques}\binits{J.}}
(\byear{1902}).
\btitle{Sur les probl{\`e}mes aux d{\'e}riv{\'e}es partielles et leur signification physique}.
\bjournal{Princeton university bulletin}
\bpages{49--52}.
\end{barticle}
\endbibitem

\bibitem{hanke2017taste}
\begin{bbook}[author]
\bauthor{\bsnm{Hanke},~\bfnm{Martin}\binits{M.}}
(\byear{2017}).
\btitle{A taste of inverse problems: basic theory and examples}.
\bpublisher{SIAM}.
\end{bbook}
\endbibitem

\bibitem{helgason1965radon}
\begin{barticle}[author]
\bauthor{\bsnm{Helgason},~\bfnm{Sigurdur}\binits{S.}}
(\byear{1965}).
\btitle{The Radon transform on Euclidean spaces, compact two-point homogeneous spaces and Grassmann manifolds}.
\bjournal{Acta Mathematica}
\bvolume{113}
\bpages{153 -- 180}.
\end{barticle}
\endbibitem

\bibitem{hsing2015theoretical}
\begin{bbook}[author]
\bauthor{\bsnm{Hsing},~\bfnm{Tailen}\binits{T.}} \AND \bauthor{\bsnm{Eubank},~\bfnm{Randall}\binits{R.}}
(\byear{2015}).
\btitle{Theoretical foundations of functional data analysis, with an introduction to linear operators}
\bvolume{997}.
\bpublisher{John Wiley \& Sons}.
\end{bbook}
\endbibitem

\bibitem{huang2017restoration}
\begin{barticle}[author]
\bauthor{\bsnm{Huang},~\bfnm{Yixing}\binits{Y.}}, \bauthor{\bsnm{Huang},~\bfnm{Xiaolin}\binits{X.}}, \bauthor{\bsnm{Taubmann},~\bfnm{Oliver}\binits{O.}}, \bauthor{\bsnm{Xia},~\bfnm{Yan}\binits{Y.}}, \bauthor{\bsnm{Haase},~\bfnm{Viktor}\binits{V.}}, \bauthor{\bsnm{Hornegger},~\bfnm{Joachim}\binits{J.}}, \bauthor{\bsnm{Lauritsch},~\bfnm{Guenter}\binits{G.}} \AND \bauthor{\bsnm{Maier},~\bfnm{Andreas}\binits{A.}}
(\byear{2017}).
\btitle{Restoration of missing data in limited angle tomography based on Helgason--Ludwig consistency conditions}.
\bjournal{Biomedical Physics \& Engineering Express}
\bvolume{3}
\bpages{035015}.
\end{barticle}
\endbibitem

\bibitem{iyer2016smoothing}
\begin{barticle}[author]
\bauthor{\bsnm{Iyer},~\bfnm{Ram~V}\binits{R.~V.}}, \bauthor{\bsnm{Nasrin},~\bfnm{Farzana}\binits{F.}}, \bauthor{\bsnm{See},~\bfnm{E}\binits{E.}} \AND \bauthor{\bsnm{Mathews},~\bfnm{S}\binits{S.}}
(\byear{2016}).
\btitle{Smoothing splines on unit ball domains with application to corneal topography}.
\bjournal{IEEE Transactions on Medical Imaging}
\bvolume{36}
\bpages{518--526}.
\end{barticle}
\endbibitem

\bibitem{izen1988series}
\begin{barticle}[author]
\bauthor{\bsnm{Izen},~\bfnm{Steven~H}\binits{S.~H.}}
(\byear{1988}).
\btitle{A series inversion for the x-ray transform in n dimensions}.
\bjournal{Inverse Problems}
\bvolume{4}
\bpages{725}.
\end{barticle}
\endbibitem

\bibitem{jia2011gpu}
\begin{barticle}[author]
\bauthor{\bsnm{Jia},~\bfnm{Xun}\binits{X.}}, \bauthor{\bsnm{Lou},~\bfnm{Yifei}\binits{Y.}}, \bauthor{\bsnm{Lewis},~\bfnm{John}\binits{J.}}, \bauthor{\bsnm{Li},~\bfnm{Ruijiang}\binits{R.}}, \bauthor{\bsnm{Gu},~\bfnm{Xuejun}\binits{X.}}, \bauthor{\bsnm{Men},~\bfnm{Chunhua}\binits{C.}}, \bauthor{\bsnm{Song},~\bfnm{William~Y}\binits{W.~Y.}} \AND \bauthor{\bsnm{Jiang},~\bfnm{Steve~B}\binits{S.~B.}}
(\byear{2011}).
\btitle{GPU-based fast low-dose cone beam CT reconstruction via total variation}.
\bjournal{Journal of X-ray science and technology}
\bvolume{19}
\bpages{139--154}.
\end{barticle}
\endbibitem

\bibitem{jin2017deep}
\begin{barticle}[author]
\bauthor{\bsnm{Jin},~\bfnm{Kyong~Hwan}\binits{K.~H.}}, \bauthor{\bsnm{McCann},~\bfnm{Michael~T}\binits{M.~T.}}, \bauthor{\bsnm{Froustey},~\bfnm{Emmanuel}\binits{E.}} \AND \bauthor{\bsnm{Unser},~\bfnm{Michael}\binits{M.}}
(\byear{2017}).
\btitle{Deep convolutional neural network for inverse problems in imaging}.
\bjournal{IEEE transactions on image processing}
\bvolume{26}
\bpages{4509--4522}.
\end{barticle}
\endbibitem

\bibitem{katsevich2024analysis}
\begin{barticle}[author]
\bauthor{\bsnm{Katsevich},~\bfnm{Alexander}\binits{A.}}
(\byear{2024}).
\btitle{Analysis of View Aliasing for the Generalized Radon Transform in R2}.
\bjournal{SIAM Journal on Imaging Sciences}
\bvolume{17}
\bpages{415--440}.
\end{barticle}
\endbibitem

\bibitem{ludwig1966radon}
\begin{barticle}[author]
\bauthor{\bsnm{Ludwig},~\bfnm{Donald}\binits{D.}}
(\byear{1966}).
\btitle{The Radon transform on Euclidean space}.
\bjournal{Communications on pure and applied mathematics}
\bvolume{19}
\bpages{49--81}.
\end{barticle}
\endbibitem

\bibitem{maass1987x}
\begin{barticle}[author]
\bauthor{\bsnm{Maass},~\bfnm{P}\binits{P.}}
(\byear{1987}).
\btitle{The x-ray transform: singular value decomposition and resolution}.
\bjournal{Inverse problems}
\bvolume{3}
\bpages{729}.
\end{barticle}
\endbibitem

\bibitem{monard2019efficient}
\begin{barticle}[author]
\bauthor{\bsnm{Monard},~\bfnm{Fran{\c{c}}ois}\binits{F.}}, \bauthor{\bsnm{Nickl},~\bfnm{Richard}\binits{R.}} \AND \bauthor{\bsnm{Paternain},~\bfnm{Gabriel~P}\binits{G.~P.}}
(\byear{2019}).
\btitle{Efficient nonparametric Bayesian inference for X-ray transforms}.
\bjournal{The Annals of Statistics}
\bvolume{47}
\bpages{1113--1147}.
\end{barticle}
\endbibitem

\bibitem{monard2023sampling}
\begin{barticle}[author]
\bauthor{\bsnm{Monard},~\bfnm{Fran{\c{c}}ois}\binits{F.}} \AND \bauthor{\bsnm{Stefanov},~\bfnm{Plamen}\binits{P.}}
(\byear{2023}).
\btitle{Sampling the X-ray transform on simple surfaces}.
\bjournal{SIAM Journal on Mathematical Analysis}
\bvolume{55}
\bpages{1707--1736}.
\end{barticle}
\endbibitem

\bibitem{natterer1980sobolev}
\begin{barticle}[author]
\bauthor{\bsnm{Natterer},~\bfnm{Frank}\binits{F.}}
(\byear{1980}).
\btitle{A Sobolev space analysis of picture reconstruction}.
\bjournal{SIAM Journal on Applied Mathematics}
\bvolume{39}
\bpages{402--411}.
\end{barticle}
\endbibitem

\bibitem{natterer2001mathematics}
\begin{bbook}[author]
\bauthor{\bsnm{Natterer},~\bfnm{Frank}\binits{F.}}
(\byear{2001}).
\btitle{The mathematics of computerized tomography}.
\bpublisher{SIAM}.
\end{bbook}
\endbibitem

\bibitem{panaretos2009random}
\begin{barticle}[author]
\bauthor{\bsnm{Panaretos},~\bfnm{Victor~M}\binits{V.~M.}}
(\byear{2009}).
\btitle{On random tomography with unobservable projection angles}.
\bjournal{The Annals of Statistics}
\bvolume{37}
\bpages{3272 -- 3306}.
\end{barticle}
\endbibitem

\bibitem{paulsen2016introduction}
\begin{bbook}[author]
\bauthor{\bsnm{Paulsen},~\bfnm{Vern~I}\binits{V.~I.}} \AND \bauthor{\bsnm{Raghupathi},~\bfnm{Mrinal}\binits{M.}}
(\byear{2016}).
\btitle{An introduction to the theory of reproducing kernel Hilbert spaces}
\bvolume{152}.
\bpublisher{Cambridge university press}.
\end{bbook}
\endbibitem

\bibitem{rullgaard2004stability}
\begin{barticle}[author]
\bauthor{\bsnm{Rullg{\aa}rd},~\bfnm{Hans}\binits{H.}}
(\byear{2004}).
\btitle{Stability of the inverse problem for the attenuated Radon transform with 180 data}.
\bjournal{Inverse problems}
\bvolume{20}
\bpages{781}.
\end{barticle}
\endbibitem

\bibitem{schoenberg1942positive}
\begin{barticle}[author]
\bauthor{\bsnm{Schoenberg},~\bfnm{Isaac~J}\binits{I.~J.}}
(\byear{1942}).
\btitle{Positive definite functions on spheres}.
\bjournal{Duke Mathematical Journal}
\bvolume{9}
\bpages{96 -- 108}.
\end{barticle}
\endbibitem

\bibitem{sharafutdinov2012integral}
\begin{bbook}[author]
\bauthor{\bsnm{Sharafutdinov},~\bfnm{Vladimir~Altafovich}\binits{V.~A.}}
(\byear{2012}).
\btitle{Integral geometry of tensor fields}
\bvolume{1}.
\bpublisher{Walter de Gruyter}.
\end{bbook}
\endbibitem

\bibitem{sharafutdinov2016reshetnyak}
\begin{barticle}[author]
\bauthor{\bsnm{Sharafutdinov},~\bfnm{Vladimir~A}\binits{V.~A.}}
(\byear{2016}).
\btitle{The Reshetnyak formula and Natterer stability estimates in tensor tomography}.
\bjournal{Inverse Problems}
\bvolume{33}
\bpages{025002}.
\end{barticle}
\endbibitem

\bibitem{shi2015novel}
\begin{barticle}[author]
\bauthor{\bsnm{Shi},~\bfnm{Hongli}\binits{H.}}, \bauthor{\bsnm{Luo},~\bfnm{Shuqian}\binits{S.}}, \bauthor{\bsnm{Yang},~\bfnm{Zhi}\binits{Z.}} \AND \bauthor{\bsnm{Wu},~\bfnm{Geming}\binits{G.}}
(\byear{2015}).
\btitle{A novel iterative CT reconstruction approach based on FBP algorithm}.
\bjournal{PLoS one}
\bvolume{10}
\bpages{e0138498}.
\end{barticle}
\endbibitem

\bibitem{singer2018mathematics}
\begin{binproceedings}[author]
\bauthor{\bsnm{Singer},~\bfnm{Amit}\binits{A.}}
(\byear{2018}).
\btitle{Mathematics for cryo-electron microscopy}.
In \bbooktitle{Proceedings of the International Congress of Mathematicians: Rio de Janeiro 2018}
\bpages{3995--4014}.
\bpublisher{World Scientific}.
\end{binproceedings}
\endbibitem

\bibitem{smith1975lower}
\begin{barticle}[author]
\bauthor{\bsnm{Smith},~\bfnm{Kennan~T}\binits{K.~T.}} \AND \bauthor{\bsnm{Solmon},~\bfnm{Donald~C}\binits{D.~C.}}
(\byear{1975}).
\btitle{Lower dimensional integrability of L2 functions}.
\bjournal{Journal of Mathematical Analysis and Applications}
\bvolume{51}
\bpages{539--549}.
\end{barticle}
\endbibitem

\bibitem{solmon1976x}
\begin{barticle}[author]
\bauthor{\bsnm{Solmon},~\bfnm{Donald~C}\binits{D.~C.}}
(\byear{1976}).
\btitle{The X-ray transform}.
\bjournal{Journal of Mathematical Analysis and Applications}
\bvolume{56}
\bpages{61-83}.
\end{barticle}
\endbibitem

\bibitem{sorenson1987physics}
\begin{bbook}[author]
\bauthor{\bsnm{Sorenson},~\bfnm{James~A}\binits{J.~A.}}, \bauthor{\bsnm{Phelps},~\bfnm{Michael~E}\binits{M.~E.}} \betal{et~al.}
(\byear{1987}).
\btitle{Physics in nuclear medicine}.
\bpublisher{Grune \& Stratton New York}.
\end{bbook}
\endbibitem

\bibitem{stefanov2020semiclassical}
\begin{barticle}[author]
\bauthor{\bsnm{Stefanov},~\bfnm{Plamen}\binits{P.}}
(\byear{2020}).
\btitle{Semiclassical sampling and discretization of certain linear inverse problems}.
\bjournal{SIAM Journal on Mathematical Analysis}
\bvolume{52}
\bpages{5554--5597}.
\end{barticle}
\endbibitem

\bibitem{stefanov2022radon}
\begin{barticle}[author]
\bauthor{\bsnm{Stefanov},~\bfnm{Plamen}\binits{P.}}
(\byear{2022}).
\btitle{The Radon transform with finitely many angles}.
\bjournal{arXiv preprint arXiv:2208.05936}.
\end{barticle}
\endbibitem

\bibitem{tsay2021simple}
\begin{barticle}[author]
\bauthor{\bsnm{Tsay},~\bfnm{Wen-Jen}\binits{W.-J.}} \AND \bauthor{\bsnm{Ke},~\bfnm{Peng-Hsuan}\binits{P.-H.}}
(\byear{2021}).
\btitle{A simple approximation for the bivariate normal integral}.
\bjournal{Communications in Statistics-Simulation and Computation}
\bpages{1--14}.
\end{barticle}
\endbibitem

\bibitem{van1987angular}
\begin{barticle}[author]
\bauthor{\bsnm{Van~Heel},~\bfnm{Marin}\binits{M.}}
(\byear{1987}).
\btitle{Angular reconstitution: a posteriori assignment of projection directions for 3D reconstruction}.
\bjournal{Ultramicroscopy}
\bvolume{21}
\bpages{111--123}.
\end{barticle}
\endbibitem

\bibitem{wahba1981spline}
\begin{barticle}[author]
\bauthor{\bsnm{Wahba},~\bfnm{Grace}\binits{G.}}
(\byear{1981}).
\btitle{Spline interpolation and smoothing on the sphere}.
\bjournal{SIAM Journal on Scientific and Statistical Computing}
\bvolume{2}
\bpages{5--16}.
\end{barticle}
\endbibitem

\bibitem{zhang2016low}
\begin{barticle}[author]
\bauthor{\bsnm{Zhang},~\bfnm{Cheng}\binits{C.}}, \bauthor{\bsnm{Zhang},~\bfnm{Tao}\binits{T.}}, \bauthor{\bsnm{Li},~\bfnm{Ming}\binits{M.}}, \bauthor{\bsnm{Peng},~\bfnm{Chengtao}\binits{C.}}, \bauthor{\bsnm{Liu},~\bfnm{Zhaobang}\binits{Z.}} \AND \bauthor{\bsnm{Zheng},~\bfnm{Jian}\binits{J.}}
(\byear{2016}).
\btitle{Low-dose CT reconstruction via L1 dictionary learning regularization using iteratively reweighted least-squares}.
\bjournal{Biomedical engineering online}
\bvolume{15}
\bpages{1--21}.
\end{barticle}
\endbibitem

\bibitem{zhao2016fast}
\begin{barticle}[author]
\bauthor{\bsnm{Zhao},~\bfnm{Zhizhen}\binits{Z.}}, \bauthor{\bsnm{Shkolnisky},~\bfnm{Yoel}\binits{Y.}} \AND \bauthor{\bsnm{Singer},~\bfnm{Amit}\binits{A.}}
(\byear{2016}).
\btitle{Fast steerable principal component analysis}.
\bjournal{IEEE transactions on computational imaging}
\bvolume{2}
\bpages{1--12}.
\end{barticle}
\endbibitem

\bibitem{zhao2013fourier}
\begin{barticle}[author]
\bauthor{\bsnm{Zhao},~\bfnm{Zhizhen}\binits{Z.}} \AND \bauthor{\bsnm{Singer},~\bfnm{Amit}\binits{A.}}
(\byear{2013}).
\btitle{Fourier--Bessel rotational invariant eigenimages}.
\bjournal{JOSA A}
\bvolume{30}
\bpages{871--877}.
\end{barticle}
\endbibitem

\bibitem{zhou2020limited}
\begin{barticle}[author]
\bauthor{\bsnm{Zhou},~\bfnm{Bo}\binits{B.}}, \bauthor{\bsnm{Zhou},~\bfnm{S~Kevin}\binits{S.~K.}}, \bauthor{\bsnm{Duncan},~\bfnm{James~S}\binits{J.~S.}} \AND \bauthor{\bsnm{Liu},~\bfnm{Chi}\binits{C.}}
(\byear{2020}).
\btitle{Limited view tomographic reconstruction using a deep recurrent framework with residual dense spatial-channel attention network and sinogram consistency}.
\bjournal{arXiv preprint arXiv:2009.01782}.
\end{barticle}
\endbibitem

\end{thebibliography}

\end{document}